\documentclass[12pt, a4paper]{article}

\usepackage{authblk}
\usepackage[english]{babel}
\usepackage[OT1]{fontenc}
\usepackage{graphicx}
\usepackage{indentfirst}
\usepackage{amsmath,amssymb,amsthm,empheq}
\usepackage{ dsfont }
\usepackage[italian]{varioref}
\usepackage{booktabs}
\usepackage{siunitx}
\usepackage{microtype}
\usepackage{epstopdf}
\usepackage[margin=2cm]{geometry} 
\usepackage[hidelinks]{hyperref}
\usepackage{fancyhdr}
\usepackage{version}
\usepackage[utf8]{inputenc}
\usepackage{tikz}
\usepackage{tikz-cd}
\usepackage[mathscr]{eucal}
\usepackage{enumerate}
\usepackage{addfont}
\usepackage{aurical}
\usepackage[titletoc]{appendix}

\usepackage{tocloft}
\usepackage[autostyle,italian=guillemets]{csquotes}
\usepackage[backend=biber, style=alphabetic 
]{biblatex}
\newcommand{\R}{\mathbb{R}}

\newcommand{\N}{\mathbb{N}}

\renewcommand{\S}{\mathbb{S}}
\newcommand{\norm}[1]{\|#1\|}
\newcommand{\comp}{\Subset}
\newcommand{\cont}{\mathcal{C}}
\renewcommand{\epsilon}{\varepsilon}
\renewcommand{\phi}{\varphi}
\newcommand{\haus}[1]{\mathcal{H}^{#1}}
\newcommand{\loc}{\rm{loc}}

\newcommand{\Per}{\operatorname{Per}}
\newcommand{\graph}{\operatorname{graph}}
\newcommand{\subgr}{\operatorname{subgraph}}

\newcommand{\divergence}{\operatorname{div}}
\newcommand{\trace}{\operatorname{tr}}
\newcommand{\supp}{\operatorname{supp}}
\newcommand{\diam}{\operatorname{diam}}
\newcommand{\dist}{\operatorname{dist}}

\theoremstyle{plain}
\newtheorem{theorem}{Theorem}[section]
\newtheorem{lemma}[theorem]{Lemma}
\newtheorem{prop}[theorem]{Proposition}
\newtheorem*{prop*}{Proposition}
\newtheorem{cor}[theorem]{Corollary}
\theoremstyle{definition}
\newtheorem{definition}[theorem]{Definition}
\newtheorem*{definition*}{Definition}
\theoremstyle{remark}
\newtheorem{rem}[theorem]{Remark}
\newtheorem{ex}[theorem]{Example}

\newtheorem*{recall*}{Recall}
\numberwithin{equation}{section}

\title{
	On non-local almost minimal sets \\
	and 
	an application to the non-local Massari's Problem}

\author[1]{\href{https://research-repository.uwa.edu.au/en/persons/serena-dipierro}{Serena Dipierro}}

\author[1]{\href{https://research-repository.uwa.edu.au/en/persons/enrico-valdinoci}{Enrico Valdinoci}}

\author[1,2]{{Riccardo Villa}}

\affil[ ]{{\footnotesize\tt serena.dipierro@uwa.edu.au},
	{\footnotesize\tt enrico.valdinoci@uwa.edu.au},
	{\footnotesize\tt riccardo.villa@research.uwa.edu.au}}

\affil[1]{ {\footnotesize Department of Mathematics and Statistics}\\
	{\footnotesize The University of Western Australia,} 
	{\footnotesize 35 Stirling Highway,
		Perth, WA 6009, Australia}}

\affil[2]{ {\footnotesize Dipartimento di Matematica e Applicazioni}\\
	{\footnotesize Universit\`a degli Studi di Milano--Bicocca,} 
	{\footnotesize Via Cozzi 55, 20125 Milano, Italy}}

\date{ }

\begin{document}
	
	\clearpage\maketitle\thispagestyle{empty}
	\begin{abstract}
		We consider a fractional Plateau's problem
		dealing with sets with prescribed non-local mean curvature. This problem can be seen
		as a non-local counterpart of the classical Massari's Problem.
		
		We obtain
		existence and regularity results, relying on a suitable version of the non-local theory for almost minimal sets. 
		In this framework, the fractional curvature term in the energy functional can be interpreted as a perturbation of the fractional perimeter.
		
		In addition, we also discuss stickiness phenomena for
		non-local almost minimal sets.
	\end{abstract}
	\setcounter{tocdepth}{1}
	\tableofcontents
	
	\pagenumbering{arabic}
	\setcounter{page}{1}
	
	
	\section{Introduction and main results} \label{sec::intro}
	In this paper we revisit the regularity theory for non-local almost minimal sets, i.e. for sets minimizing the~$s$-perimeter functional up to a remainder depending on their volume, and we apply this
	theory to a new type of prescribed non-local mean curvature problem, which we refer to as non-local Massari's Problem.
	\medskip
	
	Almost minimizers are a classical topic in the calculus of variations for several reasons.
	First of all, in practical applications, exact minimizers may be too difficult to find, and the solutions of optimization problems may only be approximate. Real-world systems are also often influenced by measurement errors and modeling approximations, making the notion of an exact minimizer too strict. Similarly, many physical systems are subject to small fluctuations that prevent them from reaching perfect equilibrium, thus exhibiting behaviors that are close, but not precisely equal, to the theoretical minimum energy state.
	
	In addition, from a theoretical point of view, exact minimizers might not exist in suitable functional spaces due to lack of compactness; in these situations, almost minimizers often exist and provide a valuable alternative.
	
	Moreover, almost minimizers serve as useful building blocks for many results in the regularity theory of minimizers. Often, the detailed study of minimizers of ``complicated''
	energy functionals is out of reach; however, one can show that these minimizers are, in fact, almost minimizers of another, ``simpler'' functional, for which a well-developed regularity theory exists. In this way, regularity results for the ``difficult'' setting become a byproduct of the regularity theory for almost minimizers in the ``easier'' setting (i.e., one
	trades the complexity of the system for a broader class of objects).
	
	In this spirit, from the viewpoint of regularity theory, almost minimizers capture the ``worst'' allowable energy errors that still preserve favorable scaling properties at small scales.
	\medskip
	
	Non-local almost minimizers theory also finds application in many problems such as obstacle and free boundary problems, capillarity problems, and regularity of sets with prescribed non-local mean curvature. See for instance~\cite{MR3532394, MR3707346, MR3717439}.
	
	In the setting of classical minimal surfaces, a regularity theory for almost minimizers was developed by Tamanini in the early `80s as a generalization of the De Giorgi's theory concerning regularity of minimizers of the variational perimeter (see~\cite{tamanini_almost_min}). 
	
	In~\cite{caputo_guillen}, Caputo and Guillen presented a notion of non-local almost minimal boundaries similar to the one by Tamanini, extending\footnote{ In its current form, the paper~\cite{caputo_guillen} contains certain technical gaps that prevent its direct application to the regularity theory of almost minimizers. Consequently, the results of the paper~\cite{MR3322379}, when relying on~\cite{caputo_guillen}, cannot be directly employed for this purpose either. We will provide further details on these issues in the forthcoming Remark~\ref{ILREDEL3440}.	}
	methods developed more recently for non-local minimal surfaces. In~\cite{MR3322379}, the authors gave a stronger definition of almost-minimality for bounded Borel sets of~$\R^n$ and used the regularity theory to prove a uniform isoperimetric inequality for the~$s$-perimeter. Motivated by the problem of sets with prescribed non-local mean curvature, we adopt here the definition introduced in~\cite{MR3322379}.
	\medskip
	
	To present the fractional analog of Tamanini's theory, we start by recalling the definition of non-local interaction between disjoint sets introduced in~\cite{caffarelli_roquejoffre_savin_nonlocal}. Whenever needed, all subsets of~$\R^n$ will be implicitly assumed to be Lebesgue-measurable.
	
	Given~$s\in(0,1)$, for every disjoint sets~$A$ and~$B$, we define
	$$ \mathcal{L}(A,B) := \int_{\R^n}\int_{\R^n} \frac{\chi_A(x)\chi_B(y)}{|x-y|^{n+s}}\,dx\,dy = \int_B\int_A \frac{dx\,dy}{|x-y|^{n+s}}.$$
	In the wake of~\cite{caffarelli_roquejoffre_savin_nonlocal},
	we consider the~$s$-perimeter of a set~$E$ localized in an open Lipschitz domain~$\Omega$, defined as
	$$ \Per_s(E,\Omega) := \mathcal{L}(E\cap\Omega,E^c) + \mathcal{L}(E\cap\Omega^c,E^c\cap\Omega).$$
	
	\begin{definition}[Non-local almost minimizers] \label{def::almost_min}
		Given an open bounded Lipschitz domain~$\Omega\subseteq\R^n$, 
		and a parameter~$\Lambda\geq0$, we say that a set~$E$ is a~$\Lambda$-minimal set (or is an almost minimal set with respect to~$\Lambda$) for the~$s$-perimeter in~$\Omega$ if~$\Per_s(E,\Omega)<+\infty$ and
		\begin{equation} \label{eq::almost_min} 
			\Per_s(E,\Omega) \leq \Per_s(F,\Omega) + \Lambda|E\Delta F|,
		\end{equation}
		for every set~$F$ such that~$E\Delta F\subseteq \Omega$.
		
		If~$\Omega$ is unbounded, we say that~$E$ is~$\Lambda$-minimal in~$\Omega$ if it is~$\Lambda$-minimal in every bounded Lipschitz set~$K \subseteq\Omega$.
	\end{definition}
	
	Observe that, in this notation, local minimizers of the~$s$-perimeter are~$0$-minimal sets.
	
	Also, Definition~\ref{def::almost_min} is slightly stronger than a direct generalization of Tamanini's definition of almost minimality
	(see~\cite[Definition~1.5]{tamanini_almost_min}).
	
	Inspired by~\cite[Definition~2.3]{caffarelli_roquejoffre_savin_nonlocal}, we also provide a characterization of almost minimality through the notions of~	$\Lambda$-super-solution and~$\Lambda$-sub-solution properties.
	
	\begin{definition} \label{def::super_sub_sol}
		We say that a set~$E$ satisfies the~$\Lambda$-super-solution property if, for every~$A\subseteq E^c\cap \Omega$,
		\begin{equation}  \label{eq::super_sol_prop}
			\mathcal{L}(A,E)-\mathcal{L}(A,E^c\setminus A)\leq \Lambda|A|.
		\end{equation}
		Similarly, we say that~$E$ satisfies the~$\Lambda$-sub-solution property if, for every~$A\subseteq E\cap \Omega$,
		\begin{equation} \label{eq::sub_sol_prop}
			\mathcal{L}(A,E\setminus A)-\mathcal{L}(A,E^c)\geq -\Lambda|A|.
		\end{equation}
	\end{definition}
	
	\begin{lemma} \label{lemma::almost_minimal_subsupersol}
		$E$ is an almost minimal set with respect to~$\Lambda$ if and only if it satisfies both the~$\Lambda$-super-solution property and the~$\Lambda$-sub-solution property.
	\end{lemma}
	
	We point out that, in light of Definition~\ref{def::super_sub_sol} and Lemma~\ref{lemma::almost_minimal_subsupersol},
	if~$E$ is an almost minimal set then so is~$E^c$.	
	
	We will prove Lemma~\ref{lemma::almost_minimal_subsupersol} in Section~\ref{sec::holder_reg} and now we present the main results of this paper.
	
	\subsection{Regularity of almost minimal sets}
	In terms of regularity theory, we will show that the boundary of any almost minimal set	is continuously differentiable. 
	The precise statement goes as follows:
	
	\begin{theorem} \label{th::holder_reg_almost_min}
		Let~$s\in(0,1)$. Given~$\Lambda\geq0$, let~$E$ be a~$\Lambda$-minimizer for~$\Per_s(\cdot,B_1)$.
		
		Then, there exists~$\epsilon_0>0$, depending only on~$n$, $s$ and~$\Lambda$, such that if
		$$\partial E\cap B_1 \subseteq \{|x_n|\leq \epsilon_0\},$$
		then~$\partial E\cap B_{1/2}$ is a~$\cont^{1,\alpha}$-surface, for any~$\alpha\in(0,s)$.
	\end{theorem}
	
	We point out that
	Theorem~\ref{th::holder_reg_almost_min} is a generalization to~$\Lambda$-minimal sets of~\cite[Theorem~6.1]{caffarelli_roquejoffre_savin_nonlocal}.
	
	The details of the regularity theory of almost minimal sets
	will be presented in Section~\ref{sec::holder_reg} by suitably adapting some
	ideas that have been developed in~\cite{caffarelli_roquejoffre_savin_nonlocal} for~$s$-minimal surfaces. In particular, we will articulate our strategy in three steps:
	\renewcommand{\theenumi}{\roman{enumi}}
	\begin{enumerate} 
		\item we prove some volume density estimates which ensure that almost minimal sets have no high density cusps (see Theorem~\ref{th::unif_dens_estimates});
		\item we prove that almost minimal sets satisfy some Euler-Lagrange Inequalities in the viscosity sense (see Theorem~\ref{th::ELeq});
		\item we exploit the improvement of flatness technique (see Theorem~\ref{th::improv_flat}) to deduce~$\cont^{1,\alpha}$-regularity.
	\end{enumerate}
	
	\begin{rem}\label{ILREDEL3440}
		We mention that, in our setting, the previous results in the literature about almost minimizers in the fractional setting are not directly applicable. Specifically, concerning the Euler-Lagrange inequality for almost-minimizers, the argument given in Step~3 of the proof of~\cite[Theorem~5.3]{caputo_guillen} would not work as it is (see in particular line~11 from below on page~18 there, where both integrands vanish at~$r=2\epsilon^*$). This technical point has a significant impact also on the corresponding monotonicity formula and ultimately on the regularity theory, making~\cite[Theorem~1.1]{caputo_guillen} also not directly applicable to our case. The use of~\cite[Theorem~3.4]{MR3322379} would not solve this technical issue, since Step~1 of its proof also explicitly relies on~\cite[Theorem~1.1]{caputo_guillen}, and therefore a careful revision of all these arguments and suitable adaptations are necessary in our framework.\end{rem}
	
	\subsection{Monotonicity formula and size of the singular set}
	As a next step, we show that the size of the singular set of an almost minimizer~$E$, that we denote by~$\Sigma_E$, is somewhat negligible,
	according to the following statement:
	
	\begin{theorem} \label{th::haus_dim_singular_almost}
		Let~$E$ be a~$\Lambda$-minimal set in a open bounded Lipschitz domain~$\Omega\subseteq\R^n$. 
		
		Then, $\dim_\haus\ (\Sigma_E)\leq n-3$.
	\end{theorem}
	
	It would be interesting to investigate if Theorem~\ref{th::haus_dim_singular_almost} is sharp.
	In this respect, we recall that the sharpness of the analogous result for~$s$-minimal surfaces (see~\cite[Corollary~2]{savin_valdinoci_regularity}) is still an open problem.	
	
	The strategy to establish Theorem~\ref{th::haus_dim_singular_almost}
	relies on a blow-up argument and a monotonicity formula, that
	lead us to the classification of the tangent cones and the desired 
	estimate of the Hausdorff dimension of~$\Sigma_E$.
	
	The analysis of the singular set will be carried through in Section~\ref{sec::haus_dim}. For this,
	in Section~\ref{sec::monotonicity} we use the extension technique presented in~\cite{caffarelli_silvestre_extension} to characterize almost minimality (in turn, this characterization
	plays a crucial role in proving the monotonicity formula in the forthcoming Theorem~\ref{th::Phi_monotone}).
	
	To deal with such an extension problem, we consider the~$(n+1)$-dimensional half-space
	$$\R^{n+1}_+:=\big\{(x,z)\in\R^{n+1}\;{\mbox{ s.t. }}\; x\in\R^n,\ z\in[0,+\infty)\big\} .$$ 
	Given a bounded Lipschitz open domain~$\Omega\subseteq\R^{n+1}$, we define the sets
	$$ \Omega_0:=\Omega\cap\{z=0\} \qquad{\mbox{and}}\qquad \Omega_+:=\Omega\cap\{z>0\} .$$
	
	Moreover, if~$E\subseteq\R^n\times\{0\}$, we define the~$(n+1)$-dimensional extension of~$u:=\chi_E-\chi_{E^c}$ as the convolution
	\begin{equation}\label{TILDEUDEF}\widetilde{u}(x,z):=u*P(x,z),\end{equation}
	where
	\begin{align*}
		&P(x,z):=C_P \frac{z^s}{(x^2+z^2)^{\frac{n+s}{2}}},\\
		\text{and}\quad&C_P:=\left(\int_{\R^n}\frac{dx}{(1+x^2)^{(n+2)/2}}\right)^{-1}.
	\end{align*}
	For more details about this~$(n+1)$-dimensional extension, we refer to~\cite{caffarelli_silvestre_extension} and~\cite[Definition~7.1]{caffarelli_roquejoffre_savin_nonlocal}.
	
	Then, we have the following:
	
	\begin{prop} \label{prop::char_almost_min_ext}
		$E$ is a~$\Lambda$-minimal set in~$B_1$ if and only if there exists a positive constant~$\widetilde{c}_{n,s}$ such that the extension~$\widetilde{u}$ of~$u=\chi_E-\chi_{E^c}$ to~$\R^{n+1}_+$ satisfies
		\begin{equation} \label{eq::char_almos_min_ext}
			\int_{\Omega_+} z^{1-s}|\nabla\widetilde{u}|^2\,dx\,dz \leq \int_{\Omega_+} z^{1-s}|\nabla\overline{v}|^2\,dx\,dz + 8\widetilde{c}_{n,s}\Lambda\left|E\Delta \{\overline{v}(x,0)=1\}\right|,
		\end{equation}
		for every bounded Lipschitz open set~$\Omega\subseteq\R^{n+1}$ such that~$\Omega_0\comp B_1\times\{0\}$, and for every function~$\overline{v}:\R^{n+1}_+\to\R$ such that $\supp(\overline{v}-\widetilde{u})\comp\Omega$ and~$|\overline{v}|=1$ on~$\Omega_0$.
	\end{prop}
	
	\begin{rem} \label{rem::hat_Lambda}
		We point out that the constant~$\widetilde{c}_{n,s}$ in Proposition~\ref{prop::char_almost_min_ext} here is the same as in~\cite[Lemma~7.2]{caffarelli_roquejoffre_savin_nonlocal}. 
		Since it will appear several times in the forthcoming discussion, given~$\Lambda\geq0$, we define the constant
		$$ \hat{\Lambda} := 8\widetilde{c}_{n,s}\Lambda.$$
	\end{rem}
	
	Now we consider a set~$E\subseteq\R^n$ and~$R>0$. For every~$r\in(0,R)$, we define
	\begin{equation}\label{defphimono}
		\Phi_E(r) := r^{s-n}\int_{B_r^+} z^{1-s}|\nabla\widetilde{u}|^2\,dx\,dz + \frac{n-s}{s}\omega_n\,\hat{\Lambda}r^s,
	\end{equation}
	where~$\widetilde{u}$ is the extension to~$\R^{n+1}_+$ of~$u=\chi_E-\chi_{E^c}$ (defined as in~\eqref{TILDEUDEF}), $ \omega_n$
	is the Lebesgue measure of the~$n$-dimensional unit ball, and~$\hat{\Lambda}$ is as in Remark~\ref{rem::hat_Lambda}.
	
	We show that whenever~$E$ is almost minimal in a domain~$\Omega\supset B_R$, the function~$\Phi_E(r)$ is monotone in~$r$,
	according to the following result:
	
	\begin{theorem}[Monotonicity formula] \label{th::Phi_monotone}
		If~$E$ is a~$\Lambda$-minimal set in a open bounded Lipschitz domain~$\Omega\subseteq\R^n$ such that~$0\in\partial E$, then~$\Phi_E(r)$ is increasing for any~$r$ small enough. 
	\end{theorem}
	
	\subsection{Stickiness for almost minimal sets}
	
	The notion of stickiness was firstly introduced in~\cite{dipierro_savin_valdinoci_boundary_behavior} as an unexpected boundary behavior typical of~$s$-minimal surfaces. This behavior is easier to define rigorously when we consider~$s$-minimal graphs in dimension~$2$, since the stickiness phenomenon can be read as a discontinuity at the boundary.
	
	\begin{definition} \label{def::stickiness}
		Let~$E$ be a~$2$-dimensional~$s$-minimal subgraph in~$\Omega=\Omega'\times\R$, with~$\Omega'\subseteq\R$ a bounded open interval, i.e.~$E$ is a local minimizer of the~$s$-perimeter such that 
		$$E=\subgr(u):=\big\{(x_1,x_2)\in\R^2\;{\mbox{ s.t. }}\; x_2<u(x_1)\big\},$$ 
		for some function~$u:\R\to\R$. 
		
		We say that~$E$ sticks to the boundary of~$\Omega$, or that~$E$ shows a sticking phenomenon, if~$u$ is discontinuous at~$\partial\Omega'$.
	\end{definition}
	
	For further details about stickiness, we refer to~\cite{dipierro_savin_valdinoci_graphs_properties, MR3926519, dipierro_savin_valdinoci_generality_stickiness, MR4548844}.
	\smallskip
	
	In this paper, we show that stickiness does not characterize almost-minimal surfaces. Namely, we show that every time we have a non-sticking~$s$-minimal surface~$A$ in~$\Omega$, we can find a sticking almost minimizer for~$\Per_s(\cdot,\Omega)$ with external datum~$A$, and, vice versa, every time we have a sticking~$s$-minimal surface, we can find a non-sticking almost minimizer. 
	
	To do this, we first prove that the property of being almost minimal is preserved under suitable external perturbations, in the following sense:
	
	\begin{theorem}[Almost minimality does not see external perturbations] \label{th::perturbed_almost_min}
		
		Consider the cylindrical domain~$\Omega= B^{n-1}_1\times\R\subseteq\R^{n-1}\times\R$. Let~$E$ be a subset of~$\R^n$, and let~$\Sigma\subseteq\Omega^c\cap E^c$ be a bounded set.
		
		Assume that there exist a small~$\delta>0$, a function~$\psi\in\cont^{1,\alpha}(\R^{n-1})$, and a vertical translation~$\tau_n$ such that
		\begin{equation} \label{eq::E_assumption}
			E=\subgr(\psi), \quad\mbox{ in }\Omega_\delta\setminus\Omega, 
		\end{equation}
		with~$ \alpha\in(0,1)$, and
		\begin{equation}\label{eq::Sigma_assumption}
			\tau_n(\Sigma)\cap \left[\left(B^{n-1}_{1+\delta}\setminus\overline{B^{ n-1}_1}\right)\times\R\right] \subseteq G_{C,\xi} ,
		\end{equation}
		where 
		$$G_{C,\xi}:=\big\{(x',x_n)\in\R^n\;{\mbox{ s.t. }}\; 1<|x'|<1+\delta,\; \psi(x')<x_n<\psi(x')+C(|x'|-1)^{n+s-1+\xi}\big\},$$ 
		for some~$C>0$ and~$\xi>0$.
		
		Then, $E$ is a~$\Lambda$-minimizer for~$\Per_s(\cdot,\Omega)$ with external datum~$E\setminus\Omega$, for some~$\Lambda\geq0$, if and only if~$\widetilde{E}:=E\cup\Sigma$ is a~$\widetilde{\Lambda}$-minimizer for~$\Per_s(\cdot,\Omega)$ with external datum~$\widetilde{E}\setminus\Omega$, for some~$\widetilde{\Lambda}\geq0$.
	\end{theorem}
	
	The proof of Theorem~\ref{th::perturbed_almost_min} is contained in Section~\ref{sec::perturbed_almost_min}.
	This result will be recalled frequently when discussing the stickiness phenomenon in Section~\ref{sec::stickiness_almost_min}. 
	In particular, as a byproduct of Theorem~\ref{th::perturbed_almost_min}, we infer both the existence of sticking almost minimal sets and the existence of non-sticking almost minimal sets (see Propositions~\ref{prop::sticking_almost_min} and~\ref{prop::non_sticking_almost_min} in Section~\ref{sec::stickiness_almost_min}).
	
	Besides, we point out that in~\cite{dipierro_savin_valdinoci_generality_stickiness} the authors establish that~$2$-dimensional~$s$-minimal graphs can only either stick to the boundary or be globally~$\cont^1$. In contrast, we will show in Section~\ref{sec::stickiness_almost_min} that almost minimal surfaces can manifest intermediate behaviors. In particular, we exhibit an example of a non-sticking continuous, but not continuously differentiable almost minimal set (see Example~\ref{ex::nonstiking_C0_nonC1}). In order to do this, we will exploit Theorem~\ref{th::perturbed_almost_min} to suitably perturb the half-space.
	
	\subsection{The non-local Massari's Problem}
	As already mentioned at the beginning of this paper, one of the natural settings in which the regularity theory of
	almost minimizers finds a convenient application is 
	the non-local version of the Massari's Problem about sets with prescribed mean curvature.
	
	To this end, let us consider an open bounded Lipschitz domain~$\Omega\subseteq\R^n$. For an assigned function~$H\in L^1(\Omega)$, and for every set~$F\subseteq\R^n$, we define the non-local Massari's functional
	$$ \mathscr{J}^H_{\Omega,s}(F) := \Per_s(F,\Omega) + \int_{F\cap\Omega} H(x)\,dx .$$ 
	
	\begin{definition} \label{def::var_frac_min_curv}
		Let~$E$ be a set of finite fractional perimeter in~$\Omega$. We say that~$H$ is a variational non-local mean curvature for~$E$ if~$E$ is a minimizer of the functional~$\mathscr{J}^H_{\Omega,s}$ with a given datum outside~$\Omega$, i.e.
		\begin{equation*} 
			\mathscr{J}^H_{\Omega,s}(E) \leq \mathscr{J}^H_{\Omega,s}(F)
			\quad {\mbox{ for all~$F$ such that }} F\setminus\Omega=E\setminus\Omega .
		\end{equation*}
	\end{definition}
	
	Whenever the assigned mean curvature~$H$ is an essentially bounded function in~$\Omega$, as a direct consequence of the regularity theory for non-local almost minimal sets, we establish the existence and regularity of minimizers of~$\mathscr{J}^H_{\Omega,s}$ in the family of sets
	\begin{equation*} \label{eq::family_F}
		\mathscr{F} := \big\{F \;{\mbox{ s.t. }}\; \Per_s(F,\Omega)<+\infty {\mbox{ and }} F\setminus\Omega=M\setminus\Omega\big\} ,
	\end{equation*} 
	where~$M\subseteq\R^n$ is a set of finite~$s$-perimeter that plays the role of an external datum.
	
	More precisely, we prove that:
	
	\begin{theorem}[Existence of a minimizer] \label{th::existence}
		The family~$\mathscr{F}$ admits an element of variational non-local mean curvature~$H$, i.e. there exists~$E\in\mathscr{F}$ such that
		$$ \mathscr{J}^H_{\Omega,s}(E)=\inf\big\{\mathscr{J}^H_{\Omega,s}(F)\;{\mbox{ s.t. }}\; F\in\mathscr{F}\big\} .$$
	\end{theorem}
	
	Moreover, we show that every minimizer of~$\mathscr{J}^H_{\Omega,s}$ is a~$\norm{H}_{L^\infty(\Omega)}$-minimal set in the sense of Definition~\ref{def::almost_min}, see Section~\ref{gfyeui65748123456wfgh007686}. Accordingly, thanks to the regularity theory for almost minimizers, we deduce the following:
	
	\begin{theorem} \label{th::regularity}
		Let~$H\in L^\infty(B_1)$.
		Let~$E$ be a set with variational non-local mean curvature~$H$ in~$B_1$, that is 
		$$ \mathscr{J}^H_{B_1,s}(E)\leq \mathscr{J}^H_{B_1,s}(F)\quad {\mbox{ for all~$ F$ such that }} F\setminus B_1 = E\setminus B_1 .$$
		
		Then, there exists some~$\epsilon_0>0$, depending only on~$n$, $s$, and~$\|H\|_{L^\infty(B_1)}$, such that if
		$$\partial E\cap B_1 \subseteq \{|x_n|\leq \epsilon_0\},$$  
		then~$\partial E\cap B_{1/2}$ is a~$\cont^{1,\alpha}$-surface, for every~$\alpha\in(0,s)$.
	\end{theorem}
	
	Theorems~\ref{th::existence} and~\ref{th::regularity} are proved
	in Section~\ref{sec::massari}.
	
	
	\section{H\"older regularity for almost minimal sets and proofs of
		Lemma~\ref{lemma::almost_minimal_subsupersol} and Theorem~\ref{th::holder_reg_almost_min}} \label{sec::holder_reg}
	
	This section is devoted to the proof of Theorem~\ref{th::holder_reg_almost_min},
	via three main steps: uniform density estimates,
	Euler-Lagrange Inequalities in the viscosity sense and the improvement of flatness technique.
	
	\subsection{Preliminary facts}
	Let~$E$ be a set in~$\R^n$. Since the fractional perimeter is invariant under zero-measure modifications of the set under consideration, we replace~$E$ with one of its normalizations. Namely, we assume that~$\partial E$ coincides with the boundary of the measure theoretic interior of~$E$, i.e.
	$$ \partial E=\partial E(1) := \big\{x\in\R^n\;
	{\mbox{ s.t. }}
	\; 0<|E\cap B_r(x)|<\omega_n r^n \;{\mbox{ for all~$r>0$}} \big\}.$$
	For full details about the convenient choice of the representative of an~$s$-minimal set, see~\cite[Appendix~A]{dipierro_savin_valdinoci_graphs_properties}.
	\smallskip
	
	To start with, we provide the proof of Lemma~\ref{lemma::almost_minimal_subsupersol}.
	
	\begin{proof}[Proof of Lemma~\ref{lemma::almost_minimal_subsupersol}]
		Given sets~$E$ and~$F$ such that~$E\setminus\Omega=F\setminus\Omega$, let us define
		$$ A^-:=E\setminus F\subseteq E\cap\Omega\qquad
		\text{and}\qquad  A^+:=F\setminus E\subseteq E^c\cap\Omega . $$
		Notice that, by definition of localized~$s$-perimeter, we have that
		\begin{equation} \label{eq::PerF_PerE}
			\begin{split}
				&\Per_s(F,\Omega)-\Per_s(E,\Omega) \\
				=&\left[\mathcal{L}(A^-,E\setminus A^-)-\mathcal{L}(A^-,E^c)\right]
				-\left[\mathcal{L}(A^+,E)-\mathcal{L}(A^+,E^c\setminus A^+)\right]
				+2\mathcal{L}(A^-,A^+).
			\end{split}
		\end{equation}
		
		If~$E$ satisfies both the~$\Lambda$-sub-solution and the~$\Lambda$-super-solution properties, we also have that
		\begin{align}
			&\mathcal{L}(A^-,E\setminus A^-)-\mathcal{L}(A^-,E^c)\geq -\Lambda|A^-|\label{eq::A_subsol}\\
			{\mbox{and }} \quad &	\mathcal{L}(A^+,E)-\mathcal{L}(A^+,E^c\setminus A^+) \leq \Lambda|A^+|. \label{eq::A_supersol}
		\end{align}
		Therefore, subtracting~\eqref{eq::A_supersol} to~\eqref{eq::A_subsol}, we obtain that
		$$ \Per_s(F,\Omega)-\Per_s(E,\Omega) \geq -\Lambda|E\setminus F|-\Lambda|F\setminus E| = -\Lambda|E\Delta F|,$$
		which proves that~$E$ is an almost minimal set with respect to~$\Lambda$.
		
		Conversely, suppose that~$E$ is a~$\Lambda$-minimal set. Then, for every sets~$A^-\subseteq E\cap\Omega$ and~$A^+\subseteq E^c\cap\Omega$, we use~\eqref{eq::PerF_PerE} and the almost minimality of~$E$ with competitors 
		$$ F^+=E\cup A^+\qquad
		\text{and}\qquad F^-=E\setminus A^-$$
		to deduce that
		\begin{eqnarray*}
			&&		-\Lambda|A^-|=-\Lambda|E\Delta F^-|\leq \Per_s(F^-,\Omega)-\Per_s(E,\Omega)
			= \mathcal{L}(A^-,E\setminus A^-)-\mathcal{L}(A^-,E^c)
		\end{eqnarray*}
		and
		\begin{eqnarray*}
			&&		\Lambda|A^+|=\Lambda|E\Delta F^+|\leq \Per_s(E,\Omega)-\Per_s(F^+,\Omega)
			= \mathcal{L}(A^+,E\setminus A^+)-\mathcal{L}(A^+,E^c),
		\end{eqnarray*}
		which show
		that~$E$ satisfies both the~$\Lambda$-sub-solution and the~$\Lambda$-super-solution properties, as desired.
	\end{proof}
	\smallskip 
	
	We also establish a general result which guarantees that the limit of a convergent sequence of almost minimal sets is almost minimal, as well. Moreover, we show that the~$s$-perimeters is sequentially continuous in the family of almost minimal sets with respect to the~$L^1_{\loc}$-convergence.
	
	\begin{prop}  \label{prop::conv_almost_min}
		Let~$\{E_k\}_k$ be a sequence of~$\Lambda_k$-minimal sets such that~$E_k\to E$ in~$L^1_{\loc}(\R^n)$, for some set~$E$ of finite fractional perimeter, and~$\Lambda_k\to\Lambda\in[0,+\infty)$, as~$k\to+\infty$. 
		
		Then, the limit set~$E$ is~$\Lambda$-minimal.
		
		Moreover,
		\begin{equation}\label{lapers} 
			\Per_s(E,\Omega) = \lim_{k\to+\infty} \Per_s(E_k,\Omega).
		\end{equation}
	\end{prop} 
	
	This result is a readjustment to the almost minimizers case of~\cite[Theorem~3.3]{caffarelli_roquejoffre_savin_nonlocal}. For the sake of completeness, a proof of Proposition~\ref{prop::conv_almost_min} can be found in Appendix~\ref{sec::prop_conv_almost_min}.
	
	\subsection{Uniform Density Estimates}
	We now discuss a version for almost minimal boundaries of the uniform density estimates introduced in~\cite{caffarelli_roquejoffre_savin_nonlocal}. Intuitively, the result claims that if~$x\in\partial E$, every ball centered at~$x$ determines two regions of the space of comparable measure.
	
	\begin{theorem}[Uniform density estimates] \label{th::unif_dens_estimates}
		Let~$E$ be a~$\Lambda$-almost minimizer for~$\Per_s$ in~$\Omega$.
		Then, there exist constants~$r_0>0$ and~$c_0\in(0,1)$, depending only on~$n$, $s$, and~$\Lambda$, such that, for any~$x_0\in(\partial E)\cap\Omega$ and~$r\in(0,\min\{r_0,\dist(x_0,\partial\Omega)\})$,
		\begin{equation}\label{eq::unif_dens_estimates}
			c_0r^n\leq |E\cap B_r(x_0)|\leq (1-c_0)r^n.
		\end{equation}
	\end{theorem}
	The proof of this result will be presented here below.
	\smallskip
	
	As an interesting byproduct of the uniform density estimates
	in Theorem~\ref{th::unif_dens_estimates}, we improve Proposition~\ref{prop::conv_almost_min}. In particular, we show that uniform density estimates and local convergence are sufficient to prove uniform convergence, i.e. with respect to Hausdorff distance.

	We recall that a sequence of sets~$\{E_k\}_k$ converges to some set~$E$ with respect to the Hausdorff distance, and we write
	$$ \lim_{k\to+\infty} \mbox{d}_\haus{}(\partial E_k,\partial E)=0 ,$$		
	if, for every compact set~$K$ of~$\Omega$ and for all~$\epsilon>0$, there exists~$k_0\in\N$ such that, for all~$k\ge k_0$,
	\begin{equation*}
		\begin{split}
			&\partial E_k\cap K \subseteq \mathscr{U}_\epsilon(\partial E)\cap K\\
			\mbox{and}\quad& \partial E\cap K \subseteq \mathscr{U}_\epsilon(\partial E_k)\cap K
		\end{split}
	\end{equation*}
	where, for any set~$A$, 
	\begin{equation*} \label{eq::fat_boundary}
		\mathscr{U}_\epsilon(\partial A):=\{x\in\R^n\,{\mbox{ s.t. }}\, \dist(x,\partial A)<\epsilon\}.
	\end{equation*}
	
	Then, we have the following:
	\begin{cor}[Uniform convergence for almost minimal boundaries] \label{cor::boundary_conv}
		Let~$\{E_k\}_k$ be a sequence of~$\Lambda_k$-minimal sets such that~$E_k\to E$ in~$L^1_{\loc}(\R^n)$, for some set~$E$ of finite~$s$-perimeter, and~$\Lambda_k\to\Lambda\in[0,+\infty)$, as~$k\to+\infty$.
		
		Then, 
		$$ \lim_{k\to+\infty} \mbox{d}_\haus{}(\partial E_k,\partial E)=0 .$$		
	\end{cor}
	
	Corollary~\ref{cor::boundary_conv} is a special case of a more general result in which we assume that the
	sets~$E_k$ and~$E$ satisfy the uniform density estimate instead of the almost minimality condition. We refer to Appendix~\ref{sec::unif_density_estimates_cor}
	for the general statement with proof. 
	We provide the proof of Corollary~\ref{cor::boundary_conv} here below.
	\smallskip
	
	Another consequence of the uniform density estimates in Theorem~\ref{th::unif_dens_estimates} is the so-called ``clean ball condition'': 
	
	\begin{cor}[Clean ball condition for almost minimal sets] \label{cor::clean_ball}
		Let~$E$ be an almost minimizer for~$\Per_s$ in~$\Omega$.
		Let~$x\in\partial E$ and~$r>0$ such that~$B_r(x)\subseteq\Omega$.
		
		Then, there exist a constant~$c>0$, depending only on~$n$, $s$, and~$\Lambda$, and points~$y_1\in E$ and~$y_2\in\Omega\setminus E$ such that
		$$ B_{cr}(y_1) \subseteq E\cap B_r(x)\qquad {\mbox{and}} \qquad B_{cr}(y_2) \subseteq E^c\cap B_r(x).$$
	\end{cor}
	
	The proof of Corollary~\ref{cor::clean_ball} can be found below.		
	Now, we give a proof of Theorem~\ref{th::unif_dens_estimates}, for which we need the following technical lemma. 
	
	\begin{lemma}[Lemma 7.1, \cite{MR1707291}] \label{lemma::iteration}
		Let~$\beta\in(0,1)$, $N>1$, and~$M>0$. Let~$\{x_k\}_k$ be a decreasing sequence in~$\R$ such that
		$$ x_{k+1}^{1-\beta}\leq N^kMx_k.$$
		If~$x_0\leq N^{\frac{1}{\beta}-\frac{1}{\beta^2}}M^{-\frac{1}{\beta}}$, then~$x_k\to0$ as~$k\to+\infty$.
	\end{lemma}
	
	\begin{proof}[Proof of Theorem~\ref{th::unif_dens_estimates}]
		If~$E$ is a~$\Lambda$-minimal set, for some~$\Lambda\geq0$, we have that
		\begin{equation} \label{eq::almost_min2}
			\Per_s(E,\Omega)\leq\Per_s(F,\Omega)+\Lambda|E\Delta F|,
		\end{equation}
		for every~$F $ such that~$ F\setminus\Omega=E\setminus\Omega$.
		
		Let~$x_0\in( \partial E)\cap\Omega$. Up to a translation,
		we can suppose that~$x_0$ coincides with the origin. Moreover, up a dilation,
		we can also assume that~$B_1\subseteq\Omega$.
		
		Define~$A_r:=E\cap B_r$, with~$r\in(0,1)$ and observe
		that~$A_r\subseteq\Omega$.
		Also, let~$\mu(r):=|A_r|$ and notice that,
		by the co-area formula, $\mu'(r)=\haus{n-1}(E\cap\partial B_r)$.
		
		Our strategy now is to provide an estimate of~$\mu^{1-\frac{s}{n}}(r)$ in terms of~$\mu(r)$.
		For this, set~$q:=\frac{2n}{n-s}$ and observe that, thanks to the Sobolev embeddings,  
		\begin{equation} \label{eq::ude1}
			\norm{\chi_{A_r}}_{L^q(\R^n)}\leq C\norm{\chi_{A_r}}_{H^{\frac{s}{2}}(\R^n)} = C\left(\mathcal{L}(A_r,A_r^c)\right)^{\frac12},
		\end{equation}
		up to renaming~$C$, that depends only on~$n$ and~$s$.
		
		Now, since
		\begin{equation*}
			\mathcal{L}(A_r,A_r^c) = \mathcal{L}(A_r,E\cap A_r^c)+\mathcal{L}(A_r,E^c),
		\end{equation*}
		it follows from~\eqref{eq::sub_sol_prop} that
		\begin{equation} \label{eq::ude2}
			\mathcal{L}(A_r,A_r^c)\leq 2\mathcal{L}(A_r,E\cap A_r^c) + \Lambda\mu(r)\leq 2\mathcal{L}(A_r,B_r^c) + \Lambda\mu(r).
		\end{equation}
		Moreover, by Fubini-Tonelli's Theorem, we have that
		\begin{equation*}
			\begin{split}
				&	\mathcal{L}(A_r,B_r^c)
				=\int_{A_r}\int_{B_r^c} \frac{dy\,dx}{|x-y|^{n+s}} 
				\leq C\int_{A_r}\left(\int_{r-|x|}^{+\infty}\frac{dz}{z^{s+1}}\right)\,dx \\
				&\qquad\quad\leq C\int_{A_r} \frac{dx}{(r-|x|)^{s}}
				\leq C\int_0^r\frac{\mu'(\rho)}{(r-\rho)^s}\, d\rho ,
			\end{split}
		\end{equation*}
		for some~$C>0$, depending on~$n$ and~$s$ and possibly changing from line to line.
		
		Plugging this into~\eqref{eq::ude2}, we obtain that
		\begin{equation*} 
			\mathcal{L}(A_r,A_r^c)\le C\int_0^r\frac{\mu'(\rho)}{(r-\rho)^s}\, d\rho + \Lambda\mu(r).
		\end{equation*}
		{F}rom the last inequality and~\eqref{eq::ude1}, we deduce that
		\begin{equation}\label{eq::ude3}
			\mu^{\frac{n-s}{n}}(r)=\norm{\chi_{A_r}}_{L^{\frac{2n}{n-s}}(\R^n)}^2 \leq 
			C \left(\int_0^r\frac{\mu'(\rho)}{(r-\rho)^s}\, d\rho + \Lambda\mu(r)\right) ,
		\end{equation} up to renaming~$C$.
		
		Furthermore, since~$\mu(r)\leq c_nr^n$, we choose~$r_0>0$ such that, for every~$r\in(0,r_0]$,  
		\begin{equation*}
			C\Lambda\mu(r) \leq C\Lambda\mu^{1-\frac{s}{n}} (r)
			( c_nr^n)^{\frac{s}{n}}\leq\frac{1}{2}\mu^{1-\frac{s}{n}}(r).
		\end{equation*}
		Using this information into~\eqref{eq::ude3}
		we thus conclude that
		\begin{equation*}
			\mu^{1-\frac{s}{n}} (r)\leq C\int_0^r\frac{\mu'(\rho)}{(r-\rho)^s}\, d\rho .
		\end{equation*}
		Hence,
		integrating in~$r\in(0,t)$, we deduce that, for all~$t\in(0, r_0]$,
		\begin{equation} \label{eq::ude7}
			\begin{split}
				\int_0^t \mu^{1-\frac{s}{n}}(r)\, dr
				&\leq \int_0^t\left(C
				\int_0^r\frac{\mu'(\rho)}{(r-\rho)^s}\, d\rho  \right)\,dr\\
				&= C\int_0^t \left(\mu'(\rho)\int_\rho^t\frac{dr}{(r-\rho)^s}\right)\, d\rho \\
				&= C\int_0^t \mu'(\rho) (t-\rho)^{1-s}\,d\rho \\&
				\leq C \mu(t) \,t^{1-s}.
			\end{split}
		\end{equation}
		
		Now we set
		$$ c_0:=2^{\frac{n}s-\frac{n^2}{s^2}}(4C)^{-\frac{n}s}$$
		and we claim that
		\begin{equation}\label{y9564rzxcvbnasdfghjkqwertyui1234567}
			\mu(t)\ge c_0t^n \quad {\mbox{for all }} t\in(0,r_0].\end{equation}
		To prove this,
		we argue by contradiction and assume that
		there exists~$t_0\in(0, r_0]$ such that
		\begin{equation}\label{wt436y95687rfdcghvdsj}
			\mu(t_0)< c_0t_0^n.\end{equation}
		
		We define the sequence~$\{t_k\}_k$ as~$t_k:=\frac{t_0}{2}+\frac{t_0}{2^{k+1}}$. Then, using~\eqref{eq::ude7}, we have that
		$$ \frac{t_0}{2^{k+2}} \mu^{1-\frac{s}{n}}(t_{k+1}) = (t_k-t_{k+1})  \mu^{1-\frac{s}{n}}(t_{k+1})  \leq \int_{t_{k+1}}^{t_k}  \mu^{1-\frac{s}{n}}(r)\,dr\leq C \mu(t_k) t_k^{1-s}\leq C \mu(t_k) t_0^{1-s} .$$
		Notice that, by continuity, 
		$$\lim_{k\to+\infty}\mu(t_{k})=\mu(t_0/2)=|E\cap B_{t_0/2}|>0.$$
		Therefore, using Lemma~\ref{lemma::iteration}
		with~$x_k:=\mu(t_k)$,
		$\beta:=s/n$, $M:=4Ct_0^{-s}$, and~$N:=2$, we find that 
		$$\mu(t_0)>N^{\frac{1}{\beta}-\frac{1}{\beta^2}}M^{-\frac{1}{\beta}}=2^{\frac{n}s-\frac{n^2}{s^2}}\big(4Ct_0^{-s}\big)^{-\frac{n}s}.$$ 
		Thus, thanks to~\eqref{wt436y95687rfdcghvdsj} we deduce that
		$$ c_0t_0^n> \mu(t_0)>
		2^{\frac{n}s-\frac{n^2}{s^2}}\big(4Ct_0^{-s}\big)^{-\frac{n}s}=
		2^{\frac{n}s-\frac{n^2}{s^2}}(4C)^{-\frac{n}s} t_0^n
		= c_0t_0^n ,$$
		which gives the desired contradiction
		and proves~\eqref{y9564rzxcvbnasdfghjkqwertyui1234567}. 
		
		{F}rom~\eqref{y9564rzxcvbnasdfghjkqwertyui1234567}
		we deduce that
		$$ |E\cap B_r|\geq c_0r^n \quad \text{ for all }r\in(0,r_0],$$
		which proves the first inequality in~\eqref{eq::unif_dens_estimates}.
		
		Moreover, since also~$E^c$ is an almost minimal set
		(recall Definition~\ref{def::super_sub_sol} and
		Lemma~\ref{lemma::almost_minimal_subsupersol}),
		we exploit~\eqref{y9564rzxcvbnasdfghjkqwertyui1234567}
		with~$E$ replaced by~$E^c$, thus obtaining that, for~$r$ sufficiently small,
		$$ |E^c\cap B_r|\geq c_0r^n,$$
		from which we infer the second inequality in~\eqref{eq::unif_dens_estimates}.
	\end{proof}	
	
	\begin{proof}[Proof of Corollary~\ref{cor::boundary_conv}]
		Since the sequence of nonnegative real numbers~$\Lambda_k$ is convergence, we have that there exists~$\Lambda_0\in[0,+\infty)$ such that~$\Lambda_k\le\Lambda_0$ for all~$k\in\N$. 
		
		As a consequence, the sets~$E_k$
		are almost minimizers with respect to~$\Lambda_0$ in~$\Omega$.
		Thus, we deduce from Theorem~\ref{th::unif_dens_estimates}
		that uniform density estimates hold true with constants that do not depend on~$k$. 
		
		This says that we are in the position of exploiting Corollary~\ref{cor::boundary_conv_general}, thus obtaining the desired claim
		in Corollary~\ref{cor::boundary_conv}.
	\end{proof}
	
	\begin{proof}[Proof of Corollary~\ref{cor::clean_ball}] 
		Up to rescalings and translations, we assume that~$x=0$ and~$r=1$. Let us decompose~$\R^n$ in disjoint
		hypercubes of size~$\delta$ and define~$N_\delta:=\#\text{\Fontauri{Q}}_{\ \delta}$, namely the number of elements of~$\text{\Fontauri{Q}}_{\ \delta}$, where 
		$$ \text{\Fontauri{Q}}_{\ \delta}:=\big\{Q_\delta: Q_\delta \text{ is a cube of size~$\delta$ such that }Q_\delta \subseteq Q_{3\delta}\subseteq B_1,\ Q_\delta\cap\partial E\neq\varnothing\big\}.$$
		Notice that Corollary~\ref{cor::clean_ball} is proved if we show that at least one cube~$Q_\delta$ is completely contained
		in~$E\cap B_1$.
		
		For this, we observe that the uniform density estimates in Theorem~\ref{th::unif_dens_estimates} give that~$|E\cap B_1| \geq c_0$. Therefore,
		if we prove that 
		\begin{equation}\label{eq::clean_ball1}
			N_\delta\leq C\delta^{s-n},
		\end{equation}
		for some~$C>0$ depending only on~$n$, $s$, and~$\Lambda$, then we deduce that at least~$c_0\delta^{-n}$ cubes intersect~$E\cap B_1$, provided that~$\delta$ is sufficiently small. 
		
		Thus, arguing by contradiction, if none of these cubes is completely contained in~$E\cap B_1$, also using~\eqref{eq::clean_ball1}, we have that
		\begin{equation*}
			C\delta^{s-n} \geq N_\delta \geq c_0\delta^{-n}.
		\end{equation*}
		It thereby follows that~$C\delta^s\geq c_0$, which is a contradiction whenever we choose~$\delta$ small enough. 
		
		Hence, from now on, we focus on the proof of~\eqref{eq::clean_ball1}. For this, consider a cube~$Q_\delta\subseteq B_1$ such that~$Q_\delta\cap\partial E\neq\varnothing$. 
		Let~$x_0\in Q_\delta\cap\partial E$. Then, for any~$\delta$ small enough, we have that~$B_\delta(x_0)\subseteq Q_{3\delta}\subseteq \Omega$. Thus, thanks to
		Theorem~\ref{th::unif_dens_estimates}, we obtain that
		$$|E\cap Q_{3\delta}|\geq c_0\delta^n\qquad{\mbox{and}}\qquad |E^c\cap Q_{3\delta}|\geq c_0\delta^n.$$
		
		{F}rom this, using also that~$|x-y|\leq 3\sqrt{n}\delta$, for every~$x$, $y\in Q_\delta$, we deduce that
		\begin{equation}\label{eq::clean_ball2}
			\begin{split}
				\mathcal{L}(E\cap Q_{3\delta},E^c\cap Q_{3\delta}) 
				& = \int_{E\cap Q_{3\delta}}\int_{E^c\cap Q_{3\delta}} \frac{dx\,dy}{|x-y|^{n+s}} \\
				& \geq \int_{E\cap Q_{3\delta}}\int_{E^c\cap Q_{3\delta}} \big(3\sqrt{n}\delta\big)^{-n-s}\,dx\,dy\\
				& = \big(3\sqrt{n}\delta\big)^{-n-s}|E\cap Q_{3\delta}|\,|E^c\cap Q_{3\delta}| \\
				& \geq C_1\delta^{n-s},
			\end{split}
		\end{equation}
		for some~$C_1>0$, depending on~$n$, $s$, and~$\Lambda$.
		
		Set~$\text{\Fontskrivan{R}}_\delta$ to be the family of all cubes $Q_\delta$ such that $Q_{3\delta}$ is contained in~$B_1$, and notice that~$\text{\Fontauri{Q}}_{\ \delta}\subseteq\text{\Fontskrivan{R}}_\delta$. Thus, using~\eqref{eq::clean_ball2}, we see that
		\begin{equation}\label{eq::clean_ball3}
			\begin{split}
				\mathcal{L}(E\cap B_1,E^c\cap B_1) 
				& = \sum_{Q_\delta,Q_\delta'\in\text{\Fontskrivan{R}}_\delta} \mathcal{L}(E\cap Q_\delta,E^c\cap Q_{\delta}') \\
				& \geq 3^{-2n}\sum_{Q_\delta,Q_\delta'\in\text{\Fontskrivan{R}}_\delta} \mathcal{L}(E\cap Q_{3\delta},E^c\cap Q_{3\delta}') \\
				& \geq 3^{-2n}\sum_{Q_\delta\in\text{\Fontskrivan{R}}_\delta} \mathcal{L}(E\cap Q_{3\delta},E^c\cap Q_{3\delta}) \\
				& \geq 3^{-2n}\sum_{Q_\delta\in\text{\Fontauri{Q}}_{\ \delta}} \mathcal{L}(E\cap Q_{3\delta},E^c\cap Q_{3\delta}) \\
				& \geq 3^{-2n}C_1N_\delta\delta^{n-s}.
			\end{split}
		\end{equation}
		
		Now, let us consider the set 
		\begin{equation*}
			A:=
			\begin{cases}
				B_1\quad&\text{in }B_1,\\
				E\quad&\text{in }B_1^c.
			\end{cases}
		\end{equation*}
		Then, the almost minimality of~$E$ gives that
		\begin{equation*}
			\begin{split}
				&\mathcal{L}(E\cap B_1,E^c\cap B_1)
				\leq\Per_s(E,B_1) \leq\Per_s(A,B_1) + \Lambda\omega_n \\
				&\qquad
				= \mathcal{L}(B_1, E^c\cap B_1^c)+ \Lambda\omega_n 
				\leq\Per_s(B_1,\R^n)+\Lambda\omega_n=: C_2(n,s,\Lambda).
			\end{split}
		\end{equation*}
		{F}rom this and~\eqref{eq::clean_ball3}, we obtain that
		$$ 3^{-2n}C_1N_\delta\delta^{n-s}\leq C_2,$$
		and this concludes the proof of~\eqref{eq::clean_ball1}, as desired.
	\end{proof}
	
	\subsection{Euler-Lagrange Inequalities}\label{sec::EL_ineq}
	
	As established in~\cite{caffarelli_roquejoffre_savin_nonlocal}, if~$E$ is a set of minimal~$s$-perimeter, then a weak-formulation of the Euler-Lagrange Equation for the associated minimization problem (in the viscosity sense) takes the form
	\begin{equation}\label{meancurvt5749}
		H_s[E](x):=p.v.\int_{\R^n} \frac{\chi_E(y)-\chi_{E^c}(y)}{|x-y|^{n+s}}\, dy = 0 ,\quad \text{ for a.e. }x\in(\partial E)\cap\Omega,
	\end{equation}
	and~$H_s[E]$ is called non-local mean curvature of~$\partial E$.
	
	We now show that when the minimality requirement is relaxed to
	almost minimality, we are still able to deduce Euler-Lagrange Inequalities, according to the following statements.
	
	\begin{theorem}[Euler-Lagrange Inequality]\label{th::ELeq}
		Let~$E$ be a set satisfying the~$\Lambda$-super-solution property in~$\Omega$, for some~$\Lambda\geq0$. Suppose that~$x_0\in\partial E$ and that~$E\cap\Omega$ has an interior tangent ball~$B$ at~$x_0$.
		
		Then,
		$$ \limsup_{\delta\to0} \int_{\R^n\setminus B_\delta(x_0)} \frac{\chi_E(x)-\chi_{E^c}(x)}{|x-x_0|^{n+s}} dx \leq \Lambda.$$ 
	\end{theorem} 
	
	Observe that if~$E$ satisfies the~$\Lambda$-sub-solution property and~$E\cap\Omega$ has an exterior tangent ball, then~$E^c$ satisfies the assumptions of Theorem~\ref{th::ELeq}. Therefore, using also the properties of~$\limsup$, we deduce the following:
	
	\begin{cor}[Reverse Euler-Lagrange Inequality]\label{cor::reverseELeq}
		Let~$E$ be a set satisfying the~$\Lambda$-sub-solution property in~$\Omega$, for some~$\Lambda\geq0$. Suppose that~$x_0\in\partial E$ and that~$E\cap\Omega$ has an exterior tangent ball~$B$ at~$x_0$.
		
		Then,
		$$ \liminf_{\delta\to0} \int_{\R^n\setminus B_\delta(x_0)} \frac{\chi_E(x)-\chi_{E^c}(x)}{|x-x_0|^{n+s}} dx \geq -\Lambda.$$ 
	\end{cor}

	\begin{rem}
		Up to a translation, one can always suppose that~$x_0=0$ in Theorem~\ref{th::ELeq} and Corollary~\ref{cor::reverseELeq}. Thus, for the sake of simplicity, in this section we will assume that~$x_0=0$.
	\end{rem}
	
	Before diving into the details of the proof of Theorem~\ref{th::ELeq}, we revisit a perturbation argument developed in~\cite{caputo_guillen}.
	Suppose that~$B$ is an interior tangent ball in~$E$. Up to a rescaling and a translation, we suppose that~$B=B_{2R}(-2Re_n)$, for some~$R\geq1$, and 
	that~$B$ touches~$\partial E$ at the origin. Our goal is to define a function~$T_\epsilon$ as the ``reflection outside a slightly deformed ball'', for every~$\epsilon\in(0,R)$. 
	To this end, let
	\begin{equation} \label{eq::def_V_rho_epsilon}
		V_{\rho,\epsilon}:=\big\{x\in\R^n\;{\mbox{ s.t. }}\; |x+Re_n|\leq \rho+\mbox{d}_\epsilon(x)\big\},
	\end{equation} 
	where~$\mbox{d}(x):=R^{-1}(1-|x'|^2)_+$, and~$\mbox{d}_\epsilon(x):=\epsilon^2\mbox{d}(x/\epsilon)$, see Figure~\ref{fig::deformed_balls}. 
	Here we are also using the notation~$x=(x',x_n)\in\R^{n-1}\times\R$.
	
	We define~$T_\epsilon(x)$ as the only point on the line through~$x$ and~$-Re_n$ such that~$\frac{x+T_\epsilon(x)}{2}\in \partial V_{R,\epsilon}$, i.e. 
	$$ T_\epsilon(x):=-x-2Re_n+2\left(R+\mbox{d}_\epsilon(x)\right)\frac{x+Re_n}{|x+Re_n|} .$$
	Then, we define the perturbed set~$A_\epsilon$ as 
	\begin{equation} \label{eq::def_A_epsilon}
		A_\epsilon := A_\epsilon^-\cup A_\epsilon^+,\qquad \text{ where }\quad A_\epsilon^-:=V_{R,\epsilon}\setminus E \quad \text{ and }\quad A_\epsilon^+:=T_\epsilon A_\epsilon^-\setminus E.
	\end{equation}
	Notice that one can also decompose~$A_\epsilon$ as~$S_\epsilon\dot{\cup} D_\epsilon$,
	with~$T_\epsilon S_\epsilon=S_\epsilon$ and~$D_\epsilon\subseteq V_{R,\epsilon}\setminus E$, where~$S_\epsilon:=A_\epsilon^+\cup T_\epsilon A_\epsilon^+$ and~$D_\epsilon:=A_\epsilon\setminus S_\epsilon$. Here, $\dot{\cup}$ denotes the union of disjoint sets.
	
	Moreover, denoting by~$\mathcal{R}_x$ the reflection with respect to the line
	$$\left\{t\frac{x+Re_n}{|x+Re_n|} \mbox{ s.t. }t\in\R\right\}$$ (see~\eqref{eq::def_reflex} below for an explicit expression) and setting~$r_{x,\epsilon} := \dist(x, \partial V_{R,\epsilon})$, for every~$\epsilon\in(0,{1}/(3n))$,
	we obtain the estimates
	\begin{equation} \label{eq::T_prop1}
		\|DT_\epsilon(x)-\mathcal{R}_x\|\leq\frac{2}{R}\big(
		3nr_{x,\epsilon}+|x'|\big),
	\end{equation}
	for every~$x\in V_{2R,\epsilon}\setminus V_{0,\epsilon}$,
	and
	\begin{equation} \label{eq::T_prop2}
		\left| \frac{|T_\epsilon(x)-T_\epsilon(y)|}{|x-y|}-1\right|\leq \frac{2}{R}\max\big\{3nr_{x,\epsilon}+|x'|,\,3nr_{y,\epsilon}+|y'|\big\} ,
	\end{equation}
	for every~$x$, $y\in V_{2R,\epsilon}\setminus V_{0,\epsilon}$.
	
	Furthermore, we have the inclusions
	\begin{equation} \label{eq::inclusion_A}
		A_\epsilon^-\subseteq B_{2\epsilon}\qquad {\mbox{and}}\qquad B_{\epsilon^2/(2R)}\setminus E\subseteq A_\epsilon\subseteq B_{8\epsilon} .
	\end{equation}
	The proof of these properties is quite technical and it is presented in~\cite{caputo_guillen}. For the sake of completeness, we also include the arguments necessary for our purposes
	in Appendix~\ref{sec::perturbation_properties}.
	
	For the proof of Theorem~\ref{th::ELeq}, we need the following auxiliary results. Throughout
	this section, the notation presented so far is assumed.
	
	\begin{figure}
		\centering
		\includegraphics[width=.7\linewidth]{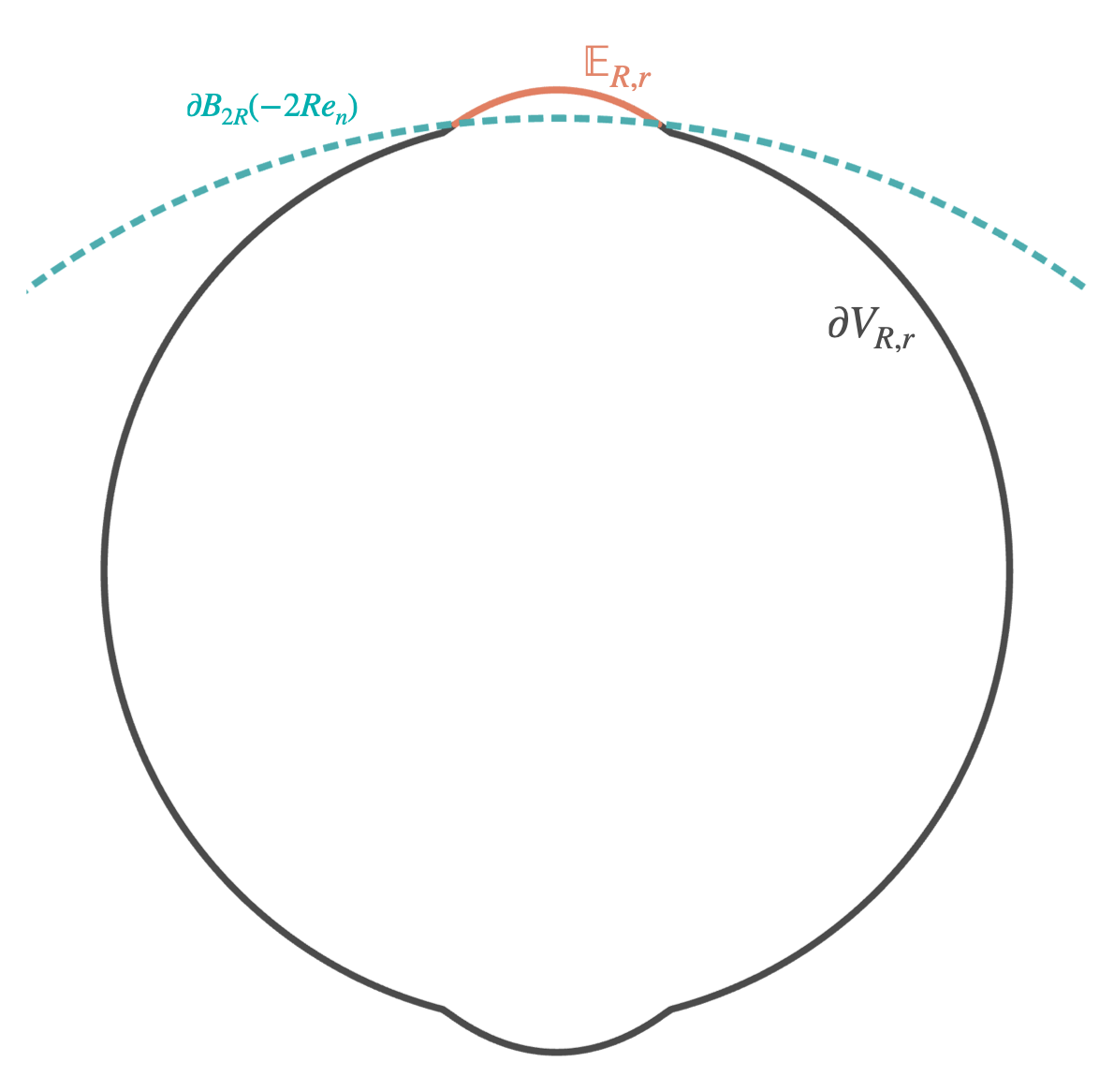}
		\caption{The deformed ball~$V_{R,r}$ and the set~$\mathbb{E}_{R,r}$ (used in the proof of Lemma~\ref{lemma::aux_EL}).}\label{fig::deformed_balls}
	\end{figure}
	
	\begin{lemma}\label{lemma::estimate_dist_V}
		Let~$\epsilon$ and~$\delta$ be such that~$0<16\epsilon<\delta$, and let~$R>1$.  
		
		Then, 
		\begin{equation*}
			\dist(A_\epsilon^-,\partial V_{R,\epsilon})\leq3\epsilon\qquad
			\text{and}\qquad\dist(A_\epsilon,\partial V_{R,\epsilon})\leq9\epsilon<\delta.
		\end{equation*}
	\end{lemma}
	
	\begin{proof}
		Recall that, by the properties of the reflection~$T_\epsilon$, we have that~$A_\epsilon^-\subseteq B_{2\epsilon}$ and~$A_\epsilon\subseteq B_{8\epsilon}$ (see~\eqref{eq::inclusion_A}). Moreover, notice that~$z:=\frac{\epsilon^2}{R}e_n\in B_\epsilon\cap \partial V_{R,\epsilon}$. 
		
		Thus, if~$x\in A_\epsilon$,
		$$ \dist(x,\partial V_{R,\epsilon})\leq |x-z| \leq 9\epsilon < \delta.$$
		In particular, if~$x\in A_\epsilon^-$, 
		$$ \dist(x,\partial V_{R,\epsilon})\leq |x-z| \leq 3\epsilon.$$
		Since~$x$ is arbitrary, the lemma is proved.
	\end{proof}

	\begin{lemma} \label{lemma::EL_aux2}
		Let~$E$ be a set with an interior tangent ball at~$0\in\partial E\cap\Omega$. 
		Let~$\epsilon$ and~$ \delta$ be such that~$0<16\epsilon<\delta$.
		
		Then, there exists~$C>0$, depending only on~$n$ and~$s$, such that
		\begin{equation*} \label{eq::EL_aux2_claim}
			\Bigg|\frac{\mathcal{L}(A_\epsilon,E\setminus B_\delta) - \mathcal{L}(A_\epsilon,E^c\setminus B_\delta)}{|A_\epsilon|}-\int_{B_\delta^c}\frac{\chi_E(x)-\chi_{E^c}(x)}{|x}|^{n+s}\,dx \Bigg| \leq C\epsilon\delta^{-1-s}.
		\end{equation*}
	\end{lemma}
	
	\begin{proof}
		Since~$|z|^{-n-s}$ is Lipschitz-continuous in~$\{|z|\geq\eta\}$, for every~$\eta>0$, it satisfies the assumptions of the Mean Value Theorem. Thus,
		\begin{equation}\label{eq::ELcor_1}
			\begin{split}
				&\left|\frac{\mathcal{L}(A_\epsilon,E\setminus B_\delta) - \mathcal{L}(A_\epsilon,E^c\setminus B_\delta)}{|A_\epsilon|}
				-\int_{B_\delta^c}\frac{\chi_E(x)-\chi_{E^c}(x)}{|x|^{n+s}}\,dx \right|\\
				&\qquad=\left| \frac{1}{|A_\epsilon|} \int_{A_\epsilon} \int_{B_\delta^c} \left(\chi_E(x)-\chi_{E^c}(x)\right)\left(
				\frac{1}{|x-y|^{n+s}}-\frac{1}{|x|^{n+s}}\right)\,dx\,dy \right|\\
				&\qquad\leq \frac{1}{|A_\epsilon|} \int_{B_\delta^c} \int_{A_\epsilon} \left|\frac{1}{|x-y|^{n+s}}-\frac{1}{|x|^{n+s}}\right|\,dy\,dx \\
				&\qquad= \frac{1}{|A_\epsilon|} \int_{B_\delta^c} \int_{A_\epsilon} c\frac{\big||x-y|-|x|\big|}{|\xi|^{n+s+1}}\,dy\,dx ,
			\end{split}
		\end{equation}
		for some positive constant~$c$ depending on~$n$ and~$s$, and for some~$\xi$ lying on the segment joining~$x$ and~$x-y$ (in particular, $|\xi|\geq\min\{|x|,|x-y|\}$). 
		Since~$A_\epsilon\subseteq B_{8\epsilon}$, we also have the estimate~$||x-y|-|x||\leq|y|\leq8\epsilon$. 
		
		Now, for every~$x\in B_\delta^c$, we decompose~$A_\epsilon$ in two sets:
		\begin{equation*}
			\begin{split}
				&A_\epsilon^1:=\big\{y\in A_\epsilon\;{\mbox{ s.t. }}\; |x-y|\geq|x|\big\}\\
				{\mbox{and}}\quad& A_\epsilon^2:=\big\{y\in A_\epsilon\;{\mbox{ s.t. }}\; |x-y|<|x|\big\} .
			\end{split}
		\end{equation*}
		We observe that
		\begin{equation}\label{eq::ELcor_2}
			\begin{split}
				&\frac{1}{|A_\epsilon|} \int_{B_\delta^c} \int_{A_\epsilon^1} c\frac{\big||x-y|-|x|\big|}{|\xi|^{n+s+1}}\,dy\,dx  
				\leq  \frac{1}{|A_\epsilon|} \int_{B_\delta^c} \int_{A_\epsilon^1} \frac{c\, \epsilon}{|x|^{n+s+1}}\,dy\,dx
				\\				&\qquad\qquad				 \leq \int_{B_\delta^c} \frac{c\,\epsilon}{|x|^{n+s+1}} \,dx
				= \frac{\omega_n}{1+s}c\,\epsilon\delta^{-1-s},
			\end{split}
		\end{equation} up to renaming~$c$.
		
		Furthermore,
		\begin{equation*}
			\begin{split}
				\frac{1}{|A_\epsilon|} \int_{B_\delta^c} \int_{A_\epsilon^2} c_{n,s}\frac{\big||x-y|-|x|\big|}{|\xi|^{n+s+1}}\,dy\,dx 
				\leq&\frac{1}{|A_\epsilon|} \int_{B_\delta^c} \int_{A_\epsilon^2} \frac{c\,\epsilon}{|x-y|^{n+s+1}}\,dy\,dx \\
				\leq&\frac{1}{|A_\epsilon|}\int_{A_\epsilon} \int_{B_\delta^c} \frac{c\,\epsilon}{|x-y|^{n+s+1}}\,dx\,dy .
			\end{split}
		\end{equation*}
		Since~$A_\epsilon\subseteq B_{8\epsilon}\subseteq B_{\delta/2}$, it follows that
		\begin{equation*}
			\begin{split}
				&\frac{1}{|A_\epsilon|} \int_{B_\delta^c} \int_{A_\epsilon^2} c_{n,s}\frac{\big||x-y|-|x|\big|}{|\xi|^{n+s+1}}\,dy\,dx 
				\leq\frac{1}{|A_\epsilon|}\int_{A_\epsilon} \int_{B_{\delta/2}^c} \frac{c\,\epsilon}{|z|^{n+s+1}}\,dz\,dy \\
				&\qquad\qquad=\int_{B_{\delta/2}^c} \frac{c\,\epsilon}{|z|^{n+s+1}}\,dz
				=\frac{\omega_n}{1+s}2^{1+s}c\,\epsilon\delta^{-1-s}.
			\end{split}
		\end{equation*}
		{F}rom this and~\eqref{eq::ELcor_2}, together with~\eqref{eq::ELcor_1}, we obtain the desired result.
	\end{proof}
	
	\begin{lemma}\label{lemma::aux_EL}
		Let~$E$ be a set with an interior tangent ball at~$0\in\partial E\cap\Omega$. Let~$V_{R,\epsilon}$
		and~$A_\epsilon^-$ be
		as in~\eqref{eq::def_V_rho_epsilon} and~\eqref{eq::def_A_epsilon} respectively.
		
		Then, there exists an infinitesimal sequence~$\epsilon_k$ such that 
		$$ \mathcal{L}\big(A_{\epsilon_k}^-,V_{R,{\epsilon_k}}^c\big) \leq c\epsilon_k^{-(1+s)/2}|A_{\epsilon_k}^-|,$$
		for some positive constant~$c$ depending only on~$n$ and~$s$.
	\end{lemma}

	\begin{proof}
		Let~$\epsilon\in(0,1)$ and define~$r_{x,\epsilon}:=\dist(x,\partial V_{R,\epsilon})$. 
		For every~$r\in(0,\epsilon)$, we define the function~$m(r):=\haus{n-1}(A_\varepsilon^- \cap \partial V_{R,r})$.
		
		We point out that~$m(r)=\haus{n-1}(E^c \cap \partial V_{R,r})$, since, for all~$r\in(0,\epsilon)$,
		\begin{equation*}
			A_\varepsilon^- \cap \partial V_{R,r}=V_{R,\varepsilon}\cap E^c \cap \partial V_{R,r}=E^c \cap \partial V_{R,r}.
		\end{equation*}
		This says that the function~$m$ does not depend on~$\epsilon$.
		
		Now, if~$x\in A_\epsilon^-\subseteq V_{R,\epsilon}$, we have that
		$$ \int_{V_{R,\epsilon}^c} \frac{dy}{|x-y|^{n+s}}
		\le \int_{\R^n\setminus B_{r_{x,\epsilon}}(x)}\frac{dy}{|x-y|^{n+s}}
		= C\int_{r_{x,\epsilon}}^{+\infty}\frac{dr}{r^{1+s}}=\frac{C}{(r_{x,\epsilon})^s}.$$
		Thus, integrating over~$A_\epsilon^-$, we infer that
		\begin{equation} \label{eq::lemmaaux_1}
			\mathcal{L}(A_\epsilon^-,V_{R,\epsilon}^c) \leq C \int_{A_\epsilon^-} \frac{dx}{(r_{x,\epsilon})^s}.
		\end{equation}
		
		Notice that, by construction, 
		$$A_\epsilon^-=V_{R,\epsilon}\setminus E\subseteq V_{R,\epsilon}\setminus B_{2R}(-2Re_n).$$
		Therefore, if~$x\in A_\epsilon^-\cap \partial V_{R,r}$, for some~$r\in(0,\epsilon)$, we have that
		\begin{equation}\label{eq::r_x_epsilon_lower_bound}
			r_{x,\epsilon} \geq \dist(\mathbb{E}_{R,\epsilon},\mathbb{E}_{R,r}) \geq \frac{1}{1+c}\frac{\epsilon^2-r^2}{R},
		\end{equation}
		where~$c>0$ is a constant depending only on the dimension~$n$, and
		$$\mathbb{E}_{R,r}:=\left\{x=(x',x_n)\;{\mbox{ s.t. }}\; |x+Re_n|= R+\frac{r^2}{R}\left(1-\frac{|x'|^2}{r^2}\right)\right\}\setminus B_{2R}(-2Re_n).$$ 
		To have an idea of the situation described here, see Figure~\ref{fig::deformed_balls}. Full details to obtain the second inequality in~\eqref{eq::r_x_epsilon_lower_bound} are contained in Section~\ref{sec::r_x_epsilon_lower_bound}.
		
		Now, let us define the function~$\psi$ as
		\begin{equation*}
			\psi(x',x_n):=\left(|x'|^2-R^2+R|x+Re_n|\right)^{1/2},
		\end{equation*}
		so that~$ \psi^{-1}(r)=\partial V_{R,r}\cap A_\epsilon^-$, for every~$r\in(0,\epsilon)$.
		
		Moreover, recalling that~$A_\epsilon^-\subseteq B_{2\epsilon}$, for every~$ (x',x_n)\in A_\epsilon^-\cap\partial V_{R,r}$ we have that
		\begin{equation}\label{eq::gradient_level_set}
			\begin{split}
				|\nabla\psi(x',x_n)| &\geq|\partial_n\psi(x',x_n)|=\frac{R}{2\psi(x',x_n)}\frac{|x_n+R|}{|x+Re_n|}=\frac{R}{2r}\frac{|x_n+R|}{|x+Re_n|}\\
				&\geq\frac{R}{2r}\frac{R-2\epsilon}{R+2\epsilon}
				\geq \frac{R}{4r}
				\geq \frac{1}{4r},
			\end{split}
		\end{equation}
		as long as~$\epsilon$ is small enough.
		
		Therefore, using~\eqref{eq::r_x_epsilon_lower_bound} and~\eqref{eq::gradient_level_set} and the co-area formula, we can rewrite~\eqref{eq::lemmaaux_1} as
		\begin{equation} \label{eq::lemmaaux_2}
			\begin{split}
				\mathcal{L}(A_\epsilon^-,V_{R,\epsilon}^c) &\leq C \int_{A_\epsilon^-} \frac{|\nabla \psi(x)|}{(r_{x,\epsilon})^s}\frac{dx}{|\nabla \psi(x)|} = 
				C\int_{0}^\epsilon \int_{A_\epsilon^-\cap\partial V_{R,r}} \frac{1}{(r_{x,\epsilon})^s}\frac{d\haus{n-1}(x)\,dr}{|\nabla \psi(x)|}\\
				&\qquad\leq C\int_{0}^\epsilon \int_{A_\epsilon^-\cap\partial V_{R,r}}\frac{1}{(\epsilon^2-r^2)^s}\frac{d\haus{n-1}(x)\,dr}{|\nabla \psi(x)|}\\
				&\qquad\leq C\int_{0}^\epsilon \int_{A_\epsilon^-\cap\partial V_{R,r}}\frac{r^{1-s}}{(\epsilon-r)^s}d\haus{n-1}(x)\,dr\\
				&\qquad\leq C\int_0^\epsilon \frac{m(r)}{(\epsilon-r)^s} \,dr,
			\end{split}
		\end{equation}
		up to renaming~$C>0$ line after line.
		
		Accordingly, the desired result is proved if we show that	there exists in infinitesimal sequence~$\epsilon_k$ such that
		\begin{equation}\label{eq::lemmaaux_3} 
			\int_0^{\epsilon_k} \frac{m(r)}{(\epsilon_k-r)^s} \,dr \leq \epsilon_k^{-(1+s)/2} \int_0^{\epsilon_k} m(r) \,dr,
		\end{equation}
		since this, together with~\eqref{eq::lemmaaux_2}, will imply that
		$$ \mathcal{L}(A_{\epsilon_k}^-,V_{R,\epsilon_k}^c) \leq	C\epsilon_k^{-(1+s)/2} \int_0^{\epsilon_k} m(r) \,dr=C\epsilon_k^{-(1+s)/2} |A_{\epsilon_k}^-|.$$
		
		Hence, from now on, we focus on the proof of~\eqref{eq::lemmaaux_3}. To this end, we argue by contradiction and assume that
		there exists~$\epsilon^*\in(0,1)$ such that, for all~$\epsilon\in(0,\epsilon^*)$,
		\begin{equation*}
			\int_0^\epsilon \frac{m(r)}{(\epsilon-r)^s} \,dr > \epsilon^{-(1+s)/2} \int_0^\epsilon m(r) \,dr  .
		\end{equation*}
		Thus, integrating in~$\epsilon\in(0,\epsilon^*)$,
		\begin{equation*}
			\int_{0}^{\epsilon^*}\left(\int_0^\epsilon \frac{m(r)}{(\epsilon-r)^s} \,dr\right)\,d\epsilon > \int_{0}^{\epsilon^*}
			\left(\epsilon^{-(1+s)/2} \int_0^\epsilon m(r) \,dr\right) \,d\epsilon .
		\end{equation*}
		Using the Fubini-Tonelli's Theorem, it follows that
		\begin{align*}  
			\int_0^{\epsilon^*}\left( \int_r^{\epsilon^*}\frac{m(r)}{(\epsilon-r)^s} \,d\epsilon\right)\,dr
			&> \int_0^{\epsilon^*}m(r)\left( \int_r^{\epsilon^*} \epsilon^{-(1+s)/{2}}\, d\epsilon\right) \,dr\\
			&= \int_0^{\epsilon^*} \frac{2m(r)}{1-s}\Big((\epsilon^*)^{(1-s)/{2}} -r^{(1-s)/{2}}\Big) \,dr .
		\end{align*}
		Since the integrand is positive, we have that
		\begin{equation*}  
			\begin{split}
				\int_0^{\epsilon^*}\left(\int_r^{\epsilon^*}\frac{m(r)}{(\epsilon-r)^s} \,d\epsilon\right)\,dr 
				& >  \frac{2}{1-s}\int_{0}^{\epsilon^*/2} m(r)\Big((\epsilon^*)^{(1-s)/{2}}-(\epsilon^*/2)^{(1-s)/{2}}\Big)\,dr\\
				&= \frac{2}{1-s} \big(1-2^{(s-1)/2}\big)\,(\epsilon^*)^{(1-s)/{2}}\int_0^{\epsilon^*/2} m(r)\,dr .
			\end{split}
		\end{equation*}
		
		We also observe that
		\begin{equation*}
			\int_0^{\epsilon^*}\left( \int_r^{\epsilon^*}\frac{m(r)}{(\epsilon-r)^s} \,d\epsilon\right)\,dr
			= \frac{1}{1-s} \int_0^{\epsilon^*} m(r) (\epsilon^*-r)^{1-s} \,dr
			\leq \frac{(\epsilon^*)^{1-s}}{1-s} \int_0^{\epsilon^*} m(r) \,dr.
		\end{equation*}
		{F}rom the last two displays, we thus obtain that
		\begin{equation}\label{eq::lemmaaux_6}
			\int_0^{\epsilon^*} m(r)\,dr > 2\big(1-2^{(s-1)/2}\big)\,(\epsilon^*)^{-\frac{1-s}{2}} \int_0^{\epsilon^*/2} m(r)\,dr \geq M\int_{0}^{\epsilon^*/2}m(r)\,dr,
		\end{equation}
		for every~$M \leq 2(1-2^{(s-1)/2})(\epsilon^*)^{-\frac{1-s}{2}}$.
		
		Now, we take~$k_0\in\N$ sufficiently large and, for any~$k\ge k_0$, we use~\eqref{eq::lemmaaux_6}
		with~$\epsilon^*:=2^{1-k}$. In this way, we obtain that
		\begin{equation}\label{226bis}
			\int_{0}^{2^{-k}}m(r)\,dr < M^{-1} \int_{0}^{2^{-k+1}}m(r)\,dr < \dots < M^{k_0-k} \int_{0}^{2^{-k_0}}m(r)\,dr .
		\end{equation}
		By the co-area formula, for every~$k\geq k_0$,
		$$\int_{0}^{2^{-k}}m(r) \,dr = |V_{R,2^{-k}}\cap E^c|.$$
		In particular, 
		\begin{equation} \label{eq::lemmaaux_7}
			\int_{0}^{2^{-k_0}}m(r) \,dr = |E^c\cap V_{R,2^{-k_0}}|\leq |B_{2R}(-Re_n)\cap B_R^c(-Re_n)|=:c_{n,R} .
		\end{equation}
		
		Also, by the uniform density estimates in Theorem~\ref{th::unif_dens_estimates}, we have that
		\begin{equation*}
			|V_{R,2^{-k}}\cap E^c| \geq |E^c\cap B_{2^{-k}}(-Re_n)|\geq c_02^{-nk} .
		\end{equation*}
		This, \eqref{226bis} and~\eqref{eq::lemmaaux_7} lead us to 
		$$ c_02^{-nk} < c_{n,R} M^{k_0-k}.$$
		
		Now, if~$k_0$ is large enough,
		it holds that~$2^{n+1} \leq 2(1-2^{(s-1)/2}) 2^{-\frac{(1-k)(1-s)}{2}}$.
		As a result, we can take~$M:=2^{n+1}$ and we have that
		$$ c_02^{k} < c_{n,R}2^{k_0(n+1)}.$$
		Taking
		$$ k:=\log_2\left(\frac{2c_{n,R}2^{k_0(n+1)}}{c_0}
		\right),$$ we obtain the desired contradiction.
	\end{proof}
	
	With this preliminary work, we can now complete the proof of Theorem~\ref{th::ELeq}.
	
	\begin{proof}[Proof of Theorem~\ref{th::ELeq}]
		
		Let~$\epsilon$ and~$\delta$ be such that~$0<16\epsilon<\delta<\min\{R,\diam(\Omega)\}$. Therefore, we have that~$A_\epsilon
		\subseteq B_{8\epsilon}\subseteq B_{\delta/2}\subseteq B_\delta\subseteq V_{2R,\epsilon}\setminus V_{0,\epsilon}$.
		
		Using the~$\Lambda$-super-solution property of~$E$ with~$A_\epsilon$ (see~\eqref{eq::super_sol_prop}), we obtain that
		\begin{equation} \label{eq::EL_suppprop_estimate1}
			\mathcal{L}(A_\epsilon,E)-\mathcal{L}(A_\epsilon, E^c\setminus A_\epsilon) \leq \Lambda |A_\epsilon|.
		\end{equation}
		
		Now, we define the set
		$$ F_{\epsilon,\delta} := T_\epsilon\left((E^c\setminus A_\epsilon)\cap B_\delta\right)$$
		and we point out that
		\begin{equation}\label{eq::TB_subset_B}
			T_\epsilon(B_\delta\setminus V_{R,\epsilon})\subseteq B_\delta.
		\end{equation}
		Indeed, for every~$x\in B_\delta\setminus V_{R,\epsilon}$ consider the hyperplane 
		$$\Pi_x := \left\{z\in\R^n\;{\mbox{ s.t. }}\; \left\langle \left(z-Re_n+\big(R+\mbox{d}_\epsilon(x)\big)\frac{x+Re_n}{|x+Re_n|}\right),(x+Re_n)\right\rangle=0 \right\} .$$
		Then, the reflection of~$x$ with respect to~$\Pi_x$ is~$T_\epsilon x$ and we have that~$\Pi_x$ separates~$x$ and the origin. Hence, $|T_\epsilon x|\leq|x|\leq\delta$, which gives~\eqref{eq::TB_subset_B}.
		
		In particular, since 
		\begin{equation*}
			\begin{split}
				(E^c\setminus A_\epsilon)\cap B_\delta &= E^c\cap V_{R,\epsilon}^c\cap(T_\epsilon A_\epsilon^-)^c \cap B_\delta\\
				&= E^c\cap V_{R,\epsilon}^c\cap(T_\epsilon E^c)^c\cap B_\delta,
			\end{split}
		\end{equation*}
		we have that
		\begin{equation*}
			(E^c\setminus A_\epsilon )\cap B_\delta \cap T_\epsilon E^c =\varnothing,
		\end{equation*}
		which yields that 
		\begin{equation*}
			F_{\epsilon,\delta}\cap E^c=\varnothing.
		\end{equation*}
		Hence, from this and~\eqref{eq::TB_subset_B}, we conclude that~$F_{\epsilon,\delta}\subseteq E\cap B_\delta$. 
		
		Therefore, we rewrite the left-hand side in~\eqref{eq::EL_suppprop_estimate1} as
		\begin{equation*}
			\begin{split}
				&\mathcal{L}(A_\epsilon,E)-\mathcal{L}(A_\epsilon, E^c\setminus A_\epsilon) \\
				&\qquad\qquad= \mathcal{L}(A_\epsilon,E\setminus B_\delta)-\mathcal{L}(A_\epsilon, E^c\setminus B_\delta) \\
				&\qquad\qquad\quad+ \mathcal{L}(A_\epsilon, F_{\epsilon,\delta})-\mathcal{L}(A_\epsilon, (E^c\setminus A_\epsilon)\cap B_\delta))+ \mathcal{L}(A_\epsilon, (E\cap B_\delta)\setminus  F_{\epsilon,\delta})\\
				&\qquad\qquad = \mathcal{L}(A_\epsilon,E\setminus B_\delta)-\mathcal{L}(A_\epsilon, E^c\setminus B_\delta) \\
				&\qquad\qquad\quad+ \mathcal{L}(A_\epsilon, F_{\epsilon,\delta})-\mathcal{L}(A_\epsilon, T_\epsilon F_{\epsilon,\delta})+ \mathcal{L}(A_\epsilon, (E\cap B_\delta)\setminus  F_{\epsilon,\delta})\\
				&\qquad\qquad \geq \mathcal{I}_1 + \mathcal{I}_2,
			\end{split}
		\end{equation*}
		where
		\begin{eqnarray*}
			\mathcal{I}_1&:=& \mathcal{L}(A_\epsilon,E\setminus B_\delta)-\mathcal{L}(A_\epsilon, E^c\setminus B_\delta)\\
			{\mbox{and }}\quad  \mathcal{I}_2&:=& \mathcal{L}(A_\epsilon, F_{\epsilon,\delta})-\mathcal{L}(A_\epsilon, T_\epsilon F_{\epsilon,\delta}).		
		\end{eqnarray*}
		We deduce from~\eqref{eq::EL_suppprop_estimate1} that
		\begin{equation} \label{eq::EL_suppprop_estimate2}
			\mathcal{I}_1 + \mathcal{I}_2 \leq \Lambda|A_\epsilon|.
		\end{equation}
		
		Now, thanks to Lemma~\ref{lemma::EL_aux2},
		\begin{equation*} 
			\left|\frac{\mathcal{I}_1}{|A_\epsilon|}-\int_{B_\delta^c}\frac{\chi_E(x)-\chi_{E^c}(x)}{|x|^{n+s}}\,dx \right| \leq C\epsilon\delta^{-1-s},
		\end{equation*}
		hence, from~\eqref{eq::EL_suppprop_estimate2} we get that
		\begin{equation} \label{eq::EL_aux2_1}
			\int_{B_\delta^c} \frac{\chi_E(x)-\chi_{E^c}(x)}{|x|^{n+s}}dx 
			\leq C\epsilon\delta^{-1-s}-\frac{\mathcal{I}_2}{|A_\epsilon|}+\Lambda .
		\end{equation}
		
		Now, our goal is to estimate~$\mathcal{I}_2$. To do this, recall that~$A_\epsilon=S_\epsilon\dot{\cup} D_\epsilon$. Then,
		\begin{equation}\label{eq::energy_Fed_estimate}
			\mathcal{I}_2 = \left[\mathcal{L}(S_\epsilon,F_{\epsilon,\delta}) - \mathcal{L}(S_\epsilon,T_\epsilon F_{\epsilon,\delta})\right]+\left[\mathcal{L}(D_\epsilon,F_{\epsilon,\delta}) - \mathcal{L}(D_\epsilon,T_\epsilon F_{\epsilon,\delta})\right] .
		\end{equation}
		Recalling that~$S_\epsilon=T_\epsilon S_\epsilon$ and
		changing variable with the reflection~$T_\epsilon$, we obtain that
		\begin{equation}\label{eq::EL_2}
			\begin{split}
				&	\mathcal{L}(S_\epsilon,T_\epsilon F_{\epsilon,\delta})=\mathcal{L}(T_\epsilon S_\epsilon,T_\epsilon F_{\epsilon,\delta})
				=\int_{T_\epsilon F_{\epsilon,\delta}}\int_{T_\epsilon S_\epsilon} \frac{1}{|x-y|^{n+s}}\,dx\,dy\\ 
				&\qquad\qquad=\int_{F_{\epsilon,\delta}}\int_{S_\epsilon} \frac{\left| \det\big(DT_\epsilon(x)\big) \right|\, \left|\det\big(DT_\epsilon(y)\big)\right| }{|T_\epsilon x-T_\epsilon y|^{n+s}}\,dx\,dy .
			\end{split}
		\end{equation}
		
		Thanks to the properties of~$T_\epsilon$ (see formula~\eqref{eq::perturb_prop1} in Appendix~\ref{sec::perturbation_properties}), we have that
		\begin{equation}\label{eq::EL_3}
			\|DT_\epsilon(x)-\mathcal{R}_x\|\leq\frac{2}{R}\big(3nr_{x,\epsilon}+2|x'|\big),\quad \mbox{for every }x\in S_\epsilon,
		\end{equation}
		where~$\mathcal{R}_x$ is the reflection with the respect to the line
		$$ \left\{t\frac{x+Re_n}{|x+Re_n|} \mbox{ such that }t\in\R\right\},$$
		and it is defined by components as
		\begin{equation} \label{eq::def_reflex}
			\left(\mathcal{R}_x\right)_i(y):= -y_i+2\frac{(x+Re_n)_i}{|x+Re_n|^2}\sum_{j=1}^{n}(x+Re_n)_jy_j,\quad\mbox{for all }i\in\{1,\dots,n\}.
		\end{equation}
		
		Notice that the term in the right-hand side of~\eqref{eq::EL_3} is at most of order~$\delta$ for every~$x\in S_\epsilon\subseteq B_\delta$ (indeed, thanks to Lemma~\ref{lemma::estimate_dist_V}, we have that~$r_{x,\epsilon}\leq\delta$).
		
		In particular, if~$\delta$ is small enough, we deduce from~\eqref{eq::EL_3} that~$\|DT_\epsilon(x)-\mathcal{R}_x\|\leq2$. 
		Hence, 
		\begin{equation*} 
			\begin{split}
				&\left|\det\left(DT_\epsilon(x)\right)\right| \leq\big(1+\|DT_\epsilon-\mathcal{R}_x\|\big)^n \leq1+n!\,2^n\|DT_\epsilon(x)-\mathcal{R}_x\|\\
				&\qquad\qquad\qquad \qquad \leq1+\frac{c_n}{R}\big(3nr_{x,\epsilon}+2|x'|\big) .
			\end{split}
		\end{equation*}
		Similarly, let~$y\in F_{\epsilon,\delta}$. Recall that~$ F_{\epsilon,\delta} \subseteq B_\delta$ and notice that~$z:=\frac{\epsilon^2}{R}e_n\in B_\delta\cap \partial V_{R,\epsilon}$. Thus, we have that~$r_{y,\epsilon}\leq|y-z|\leq2\delta$. 	Thus, we deduce that also~$\frac{2}{R}(3nr_{y,\epsilon}+2|y'|)\leq C\delta$.
		
		In particular, $\|DT_\epsilon(y)-\mathcal{R}_y\|\leq2$, hence
		$$ |\det\big(DT_\epsilon(y)\big)|\leq1+\frac{c_n}{R}\big(
		3nr_{y,\epsilon}+2|y'|\big) .$$
		Therefore, from~\eqref{eq::EL_2} we obtain that 
		$$ \mathcal{L}(S_\epsilon,T_\epsilon F_{\epsilon,\delta}) \leq \int_{F_{\epsilon,\delta}}\int_{S_\epsilon} \frac{1+\frac{c_n}{R}\max\{ 3nr_{x,\epsilon}+2|x'|,3nr_{y,\epsilon}+2|y'|\}}{|T_\epsilon x-T_\epsilon y|^{n+s}}\,dx\,dy .$$
		Using~\eqref{eq::T_prop2}, we also have
		$$ |T_\epsilon x-T_\epsilon y|^{-n-s}\leq \left[1-\frac{2}{R}\max\{ 3nr_{x,\epsilon}+2|x'|,3nr_{y,\epsilon}+2|y'|\}\right]^{-n-s}|x-y|^{-n-s} .$$
		Since~$\frac{2}{R}(3nr_{x,\epsilon}+2|x'|)\leq C\delta$ and~$\frac{2}{R}(3nr_{y,\epsilon}+2|y'|)\leq C\delta$, the function
		$$\left[1-\frac{2}{R}\max\{ 3nr_{x,\epsilon}+2|x'|,3nr_{y,\epsilon}+2|y'|\}\right]^{-n-s}$$ is uniformly bounded in~$x$ and~$y$, up to choosing~$\delta$ small enough. 
		
		Now we observe that the function~$\frac{1+\eta}{(1-\eta)^{n+s}}$ is bounded for~$\eta$ sufficiently small. As a result, there exists a positive constant~$c$ such that~$\frac{1+\eta}{(1-\eta)^{n+s}}\leq 1+c\eta$. Thus, using this with~$\eta:=\max\{ 3nr_{x,\epsilon}+2|x'|,3nr_{y,\epsilon}+2|y'|\}$, we deduce that
		\begin{equation}\label{eq::EL_4}
			\begin{split}
				&\mathcal{L}(S_\epsilon,TF_{\epsilon,\delta})
				\leq\int_{F_{\epsilon,\delta}}\int_{S_\epsilon} \frac{1+\frac{c_n}{R}\max\{ 3nr_{x,\epsilon}+2|x'|,3nr_{y,\epsilon}+2|y'|\}}{|x-y|^{n+s}}\,dx\,dy\\
				&\qquad=\mathcal{L}(S_\epsilon,F_{\epsilon,\delta})+\int_{F_{\epsilon,\delta}}\int_{S_\epsilon} \frac{\frac{c_n}{R}\max\{ 3nr_{x,\epsilon}+2|x'|,3nr_{y,\epsilon}+2|y'|\}}{|x-y|^{n+s}}\,dx\,dy\  .
			\end{split}
		\end{equation}
		
		A similar argument can be applied for the terms involving~$D_\epsilon$ by proceeding as follows. Since~$D_\epsilon\subseteq A_\epsilon\subseteq B_\delta$, considering the reflection with respect to the hyperplane~$\Pi_y$, we have that~$T_\epsilon y$ is the reflection of~$y$ with respect to~$\Pi_y$ and~$\Pi_y$ separates~$x$ and~$T_\epsilon y$. Thus, we obtain the estimate~$|x-T_\epsilon y|\geq|x-y|$, whence 
		\begin{equation*}
			\begin{split}
				&\mathcal{L}(D_\epsilon,T_\epsilon F_{\epsilon,\delta})
				=\int_{D_\epsilon}\int_{F_{\epsilon,\delta}} \frac{\left|\det\big(DT_\epsilon(y)\big)\right|}{|x-T_\epsilon y|^{n+s}}\,dx\,dy\\
				&\qquad\leq\int_{D_\epsilon}\int_{F_{\epsilon,\delta}} \frac{1+\frac{c_n}{R}\max\{ 3nr_{x,\epsilon}+2|x'|,3nr_{y,\epsilon}+2y'|\}}{|x-y|^{n+s}}\,dx\,dy\\
				&\qquad=\mathcal{L}(D_\epsilon, F_{\epsilon,\delta})+\int_{D_\epsilon}\int_{F_{\epsilon,\delta}} \frac{\frac{c_n}{R}\max\{ 3nr_{x,\epsilon}+2|x'|,3nr_{y,\epsilon}+2|y'|\}}{|x-y|^{n+s}}\,dx\,dy .
			\end{split}
		\end{equation*}
		Using this and~\eqref{eq::EL_4}, together with~\eqref{eq::energy_Fed_estimate}, we conclude that
		\begin{equation} \label{eq::energy_Fed_estimate2}
			-\mathcal{I}_2 \leq \int_{A_\epsilon}\int_{F_{\epsilon,\delta}} \frac{\frac{c_n}{R}\max\{ 3nr_{x,\epsilon}+2|x'|,3nr_{y,\epsilon}+2|y'|\}}{|x-y|^{n+s}}\,dx\,dy .
		\end{equation}
		
		Now, to get to the desired result, we need sharper estimates for the term~$\max\{ 3nr_{x,\epsilon}+|x'|,3nr_{y,\epsilon}+|y'|\}$. To this end, notice that if~$x\in A_\epsilon$, then~$|x|\leq8\epsilon$. Moreover, by Lemma~\ref{lemma::estimate_dist_V}, we have that~$r_{x,\epsilon}\leq9\epsilon$. Using also the simple estimates
		$$ |y'|\leq |x'|+|x-y| \qquad{\mbox{and}}\qquad r_{y,\epsilon}\leq r_{x,\epsilon}+|x-y|,$$
		we find that
		\begin{equation*}
			\max\{ 3nr_{x,\epsilon}+2|x'|,3nr_{y,\epsilon}+2|y'|\}\leq
			\begin{cases}
				(6n+2)|x-y|+16\epsilon, &\text{if }|x-y|\geq2r_{x,\epsilon},\\
				Cn\epsilon, &\text{if }|x-y|<2r_{x,\epsilon}.
			\end{cases}
		\end{equation*}
		{F}rom this, we obtain that
		\begin{equation*} 
			\begin{split}
				&\int_{F_{\epsilon,\delta}\setminus B_{2r_{x,\epsilon}}(x)} \frac{\max\{ 3nr_{x,\epsilon}+2|x'|,3nr_{y,\epsilon}+2|y'|\}}{|x-y|^{n+s}}\,dy
				\leq C_1\int_{F_{\epsilon,\delta}\setminus B_{2r_{x,\epsilon}}(x)} \frac{|x-y|+\epsilon}{|x-y|^{n+s}}\,dy \\
				&\qquad\qquad \leq C_1 \left(  \int_{B_\delta} \frac{dy}{|x-y|^{n+s-1}}  +  \int_{F_{\epsilon,\delta}\setminus B_{2r_{x,\epsilon}}(x)} \frac{\epsilon}{|x-y|^{n+s}}\,dy \right) \\
				&\qquad\qquad \leq C_1 \left( \delta^{1-s} + \int_{F_{\epsilon,\delta}\setminus B_{2r_{x,\epsilon}}(x)} \frac{\epsilon}{|x-y|^{n+s}}\,dy \right),
			\end{split}
		\end{equation*}
		and
		\begin{equation*} 
			\int_{F_{\epsilon,\delta}\cap B_{2r_{x,\epsilon}}(x)} \frac{\max\{ 3nr_{x,\epsilon}+2|x'|,3nr_{y,\epsilon}+2|y'|\}}{|x-y|^{n+s}}\,dy \leq C_1 \int_{F_{\epsilon,\delta}\cap B_{2r_{x,\epsilon}}(x)} \frac{\epsilon}{|x-y|^{n+s}}\,dy  ,
		\end{equation*}
		up to renaming~$C_1>0$, that depends only on~$n$ and~$s$.
		
		Therefore, we deduce that
		\begin{equation} \label{eq::EL_pre6}
			\int_{F_{\epsilon,\delta}} \frac{\max\{ 3nr_{x,\epsilon}+2|x'|,3nr_{y,\epsilon}+2|y'|\}}{|x-y|^{n+s}}\,dy 
			\leq C_1\left( \delta^{1-s}+\int_{F_{\epsilon,\delta}} \frac{\epsilon}{|x-y|^{n+s}}\,dy \right).
		\end{equation}
		Integrating~\eqref{eq::EL_pre6} over~$A_\epsilon$, we arrive at
		\begin{equation} \label{eq::EL_6} 
			\begin{split}
				&\int_{A_\epsilon} \int_{F_{\epsilon,\delta}} \frac{\max\{ 3nr_{x,\epsilon}+2|x'|,3nr_{y,\epsilon}+2|y'|\}}{|x-y|^{n+s}}\,dx\,dy \\
				&\qquad\qquad\leq C_1\big(\delta^{1-s}|A_\epsilon| +\epsilon\mathcal{L}(A_\epsilon^-,F_{\epsilon,\delta})+\epsilon\mathcal{L}(A_\epsilon^+,F_{\epsilon,\delta})\big).
			\end{split}
		\end{equation}
		
		The next step is estimating~$\mathcal{L}(A_\epsilon^+, F_{\epsilon,\delta})$ in terms of~$\mathcal{L}(A_\epsilon^-, F_{\epsilon,\delta})$. Since~$T_\epsilon A_\epsilon^+\subseteq A_\epsilon^-$, changing variable with~$T_\epsilon$, we arrive at
		\begin{equation}\label{eq::EL_7}
			\begin{split}
				\mathcal{L}(A_\epsilon^+, F_{\epsilon,\delta})
				&=\int_{T_\epsilon A_\epsilon^+}\int_{F_{\epsilon,\delta}} \frac{|\det\left(DT_\epsilon(x)\right)|}{|T_\epsilon x- y|^{n+s}}\,dx\,dy\\
				&\leq\int_{A_\epsilon^-}\int_{F_{\epsilon,\delta}} \frac{1+\frac{c_n}{R}\max\{ 3nr_{x,\epsilon}+2|x'|,3nr_{y,\epsilon}+2|y'|\}}{|x-y|^{n+s}}\,dx\,dy\\
				&\leq (1+C_2\epsilon)\mathcal{L}(A_\epsilon^-, F_{\epsilon,\delta})+C_2\delta^{1-s}|A_\epsilon^-|,
			\end{split}
		\end{equation}
		for some positive~$C_2$ depending on~$n$ and~$s$.
		
		Using~\eqref{eq::EL_6} and~\eqref{eq::EL_7} in~\eqref{eq::energy_Fed_estimate2}, we are led to
		\begin{equation}\label{eq::energy_Fed_estimate3}
			-\mathcal{I}_2 \leq \frac{C_3}{R}\big(\delta^{1-s}|A_\epsilon|+\epsilon\mathcal{L}(A_\epsilon^-, F_{\epsilon,\delta})\big),
		\end{equation}
		for some~$C_3>0$ depending only on~$n$ and~$s$.
		
		We now show that~$\mathcal{L}(A_\epsilon^-,F_{\epsilon,\delta})$ is controlled by~$\epsilon^{-(s+1)/2}$.
		Since~$F_{\epsilon,\delta}\subseteq E$, using also the~$\Lambda$-super-solution property of~$E$ with~$A_\epsilon^-$, we obtain that
		$$\mathcal{L}(A_\epsilon^-,F_{\epsilon,\delta})\leq\mathcal{L}(A_\epsilon^-,E)\leq\mathcal{L}(A_\epsilon^-,E^c\setminus A_\epsilon^-)+\Lambda|A_\epsilon^-|.$$
		We also see that~$E^c\setminus A_\epsilon^-\subseteq V_{R,\epsilon}$, thus
		$$\mathcal{L}(A_\epsilon^-,F_{\epsilon,\delta})\leq\mathcal{L}(A_\epsilon^-,V_{R,\epsilon})+\Lambda|A_\epsilon^-|.$$
		Thanks to Lemma~\ref{lemma::aux_EL}, we know that there
		exists an infinitesimal sequence~$\epsilon_k$ such that
		$$ \mathcal{L}(A_{\epsilon_k}^-,V_{R,\epsilon_k}^c)\leq c|A_{\epsilon_k}^-|\epsilon_k^{-(s+1)/2}.$$
		
		This with~\eqref{eq::energy_Fed_estimate3} gives that
		\begin{equation} \label{eq::EL_8}
			\begin{split}
				-\mathcal{I}_2 
				\leq & \frac{C_3}{R} \Big(\delta^{1-s}|A_{\epsilon_k}|+c\epsilon_k^{\frac{1-s}{2}}|A_{\epsilon_k}^-|+\epsilon_k \Lambda|A_{\epsilon_k}^-|\Big)
				\\	\leq& \frac{C_3}{R} |A_{\epsilon_k}| \Big(\delta^{1-s}+c\epsilon_k^{\frac{1-s}{2}}+\epsilon_k \Lambda\Big).
			\end{split}
		\end{equation}
		{F}rom this and~\eqref{eq::EL_aux2_1}, we deduce that
		\begin{equation*}
			\int_{B_\delta^c} \frac{\chi_E(x)-\chi_{E^c}(x)}{|x|^{n+s}}dx 
			\leq C\epsilon_k\delta^{-1-s}+ \frac{C_3}{R} \Big(\delta^{1-s}+c\epsilon_k^{\frac{1-s}{2}}+\epsilon_k\Lambda\Big) + \Lambda.
		\end{equation*}
		Taking the limit as~$k\to+\infty$, we conclude that
		\begin{equation*}
			\int_{B_\delta^c} \frac{\chi_E(x)-\chi_{E^c}(x)}{|x|^{n+s}}dx \leq \frac{C_3}{R} \delta^{1-s} + \Lambda.
		\end{equation*}
		Taking now the~$\limsup$ as~$\delta\to0$, we conclude the proof of Theorem~\ref{th::ELeq}.
	\end{proof}
	
	As a consequence of the estimates in~\eqref{eq::EL_suppprop_estimate2} and~\eqref{eq::EL_8}, we deduce the following:
	
	\begin{cor}[Alternative Euler-Lagrange Inequalities] \label{cor::EL_cor}
		There exist sequences~$\{\epsilon_k\}_k$ and~$\{\delta_k\}_k$, with~$0<16\epsilon_k<\delta_k$, such that~$\epsilon_k$, $\delta_k\to0$ as~$k\to+\infty$, and
		\begin{equation*}
			\left| \mathcal{L}(A_{\epsilon_k},E\setminus B_{\delta_k})-\mathcal{L}(A_{\epsilon_k},E^c\setminus B_{\delta_k}) \right| \leq C |A_{\epsilon_k}| \Big(\delta_k^{1-s}+\epsilon_k^{\frac{1-s}{2}}+\Lambda\Big),
		\end{equation*}
		for some positive constant~$C$ depending only on~$n$ and~$s$.
	\end{cor}
	
	\subsection{Harnack's Inequality} \label{sec::harnack}
	Now, we use Theorem~\ref{th::ELeq} to prove a Harnack's type result. To this purpose, we introduce the notion of flatness for a set. For later convenience, we now define, for any $x_0\in\R^n$, $\nu\in \S^{n-1}$, and $r\in\R$, the cylinder
	\begin{equation*}
		C_r^\nu(x_0) := \{x\in\R^n \mbox{ s.t. } |\pi_\nu(x-x_0)|<r \},
	\end{equation*}
	where $\pi_\nu(x):= x-(x\cdot\nu)\nu$ denotes the projection onto the hyperplane $\{x\cdot\nu=0\}$.
	
	When $\nu=e_n$, we will denote the vertical cylinder of radius $r$ centered at $x_0$ simply with $C_r(x_0)$.
	Also, when the cylinders are centered at the origin, we use the short notations~$C_r^\nu:=C_r^\nu(0)$
	and~$C_r:=C_r(0)$.
	
	\begin{definition} \label{def::flatness}
		Let~$h$, $r>0$ and~$x_0\in\partial E$. We say that a set~$E$ is~$\left(\frac{h}{r}\right)$-flat in $C_r^\nu(x_0)$ if 
		$$ \partial E\cap C_r^\nu(x_0)\subseteq\big\{|(x-x_0)\cdot\nu|<h\big\}.$$
	\end{definition}
	Heuristically, we say that a set is~$\left(\frac{h}{r}\right)$-flat in the cylinder $C_r^\nu(x_0)$ if it can be trapped in a slab of width~$2h$ in direction $\nu$. 
	
	\begin{theorem}[Harnack's Inequalities] \label{th::harnack_ineq}
		Let~$E\subseteq\R^n$ and~$\Lambda\geq0$. 
		
		Then, for any~$s\in(0,1)$ and~$\alpha\in(0,s)$, there exist positive constants $\delta_0$ and~$k_0$, depending only on~$n$, $s$, $\alpha$, and~$\Lambda$, such that, setting~$a:=2^{-\alpha k}$ for some~$k\geq k_0$, if~$E$ is~$(2^{-ks}\Lambda)$-minimal for~$\Per_s$ in~$B_{2^k}$, and there exists a collection $\{\nu_j\}_j\subseteq \S^{n-1}$ such that 
		\begin{equation}  \label{eq::flatness_harnack}
			\partial E\cap C^{\nu_j}_{2^j}\subseteq\{|x\cdot\nu_j|\leq a2^{j(1+\alpha)}\},\qquad {\mbox{for all }} j=0,\dots,k ,
		\end{equation}
		with~$\nu_0=e_n$, then we have the inclusions
		\begin{eqnarray}
			\label{eq::harnack_ineq1}\text{either} \quad&\partial E\cap C_{\delta_0} \subseteq \big\{x_n\leq a(1-\delta_0^2)\big\}\\
			\label{eq::harnack_ineq2}\text{or} \quad &\partial E\cap C_{\delta_0} \subseteq \big\{x_n\geq a(\delta_0^2-1)\big\}.
		\end{eqnarray}
	\end{theorem}
	
	We point out that
	Theorem~\ref{th::harnack_ineq} is a generalization to~$\Lambda$-minimizers of~\cite[Theorem~6.9]{caffarelli_roquejoffre_savin_nonlocal}.
	
	We will deduce Harnack's Inequalities in Theorem~\ref{th::harnack_ineq}
	from the following partial result.
	
	\begin{lemma}[Partial Harnack's Inequality]\label{lemma::partial_harnack}
		Let~$E\subseteq\R^n$ and~$\Lambda\geq0$.
		
		Then, for any~$s\in(0,1)$ and~$\alpha\in(0,s)$, there exist positive constants $\delta_0$ and~$k_0$, depending only on~$n$, $s$, $\alpha$, and~$\Lambda$, such that, setting~$a:=2^{-\alpha k}$ for some~$k\geq k_0$, if~$E$ is a~$(2^{-ks}\Lambda)$-minimal set in~$B_{2^k}$ such that
		\begin{equation}  \label{eq::flatness_harnack_lemma}
			\partial E\cap C^{\nu_j}_{2^j}\subseteq\big\{|x\cdot\nu_j|\leq a2^{j(1+\alpha)}\big\},\qquad{\mbox{for all }} j=0,\dots,k ,
		\end{equation}
		with~$\nu_0=e_n$, and
		\begin{equation} \label{eq::half_measure_cylinder}
			|E\cap C_{\delta_0,a}| \geq \frac{1}{2}|C_{\delta_0,a}| = \omega_{n-1}\delta_0^{n-1}a,
		\end{equation}where
		$$C_{\delta_0,a}:=B^{n-1}_{\delta_0}\times(-a,a),$$
		then we have that
		\begin{eqnarray} 
			\label{eq::inclusion_harnack1}\mbox{either}\quad&C_{\delta_0}\cap\big\{x_n<-a(1-\delta_0^2)\big\}\subseteq E\\
			\label{eq::inclusion_harnack2}\mbox{or}\quad&C_{\delta_0}\cap\big\{x_n>a(1-\delta_0^2)\big\}\subseteq E.
		\end{eqnarray}
		
	\end{lemma}  
	
	\begin{rem}
		We stress that, in Theorem~\ref{th::harnack_ineq} and Lemma~\ref{lemma::partial_harnack}, we are not requiring that~$0\in\partial E$.
	\end{rem}
	
	The idea of the proof of Lemma~\ref{lemma::partial_harnack} is to argue by contradiction in order to show that if neither~\eqref{eq::inclusion_harnack1} nor~\eqref{eq::inclusion_harnack2} hold true, then~$E$ stretches in such a way that it contradicts the Euler-Lagrange Inequality in Theorem~\ref{th::ELeq}.
	
	\begin{proof}[Proof of Lemma~\ref{lemma::partial_harnack}]	
		In what follows, we consider~$k_0=k_0(\delta_0)$ such that~$a\leq 2^{-k_0\alpha}<\delta_0$.
		
		By assumptions, $\partial E\cap C_1 \subseteq \{|x_n|<a\}$. We claim that if~$\{x_n<-a\}\cap C_1 \subseteq E$, then~\eqref{eq::inclusion_harnack1} holds true.
		
		Arguing by contradiction, suppose not, hence there is a portion of~$\partial E \cap C_{\delta_0}$ that is trapped in the strip~$\{-a<x_n<-(1-\delta_0^2)a\}$.
		
		Thanks to the flatness assumptions~\eqref{eq::flatness_harnack_lemma}, there exists a parabola of opening~$-\frac{a}{2}$ tangent to~$\partial E$ from below in~$C_{\delta_0}$. Denoting by~$y$ the tangent point, we have that 
		$$ y\in C_{\delta_0} \qquad\text{and}\qquad -a\leq y_n \leq -(1-\delta_0^2)a.$$
		
		Choosing~$\delta_0<1/2$, we estimate the far-away contributions of the fractional mean curvature of~$E$ at~$y$ as
		\begin{equation} \label{eq::harnack_proof1} 
			\begin{split}
				&	\left| \int_{B_{1/2}^c(y)}  \frac{\chi_E(x)-\chi_{E^c}(x)}{|x-y|^{n+s}}\,dx  \right|
				\\	\leq& \sum_{j=0}^{k} 	\left| \int_{B_{2^j}(y)\setminus B_{2^{j-1}}(y)}  \frac{\chi_E(x)-\chi_{E^c}(x)}{|x-y|^{n+s}}\,dx  \right| 
				+ \left| \int_{B_{2^k}^c(y)}  \frac{\chi_E(x)-\chi_{E^c}(x)}{|x-y|^{n+s}}\,dx  \right| \\
				\leq& \sum_{j=0}^{k} 	\left| \int_{B_{2^j}(y)\setminus B_{2^{j-1}}(y)}  \frac{\chi_E(x)-\chi_{E^c}(x)}{|x-y|^{n+s}}\,dx  \right| 
				+ \omega_n\int_{2^k}^{+\infty}  \frac{dr}{r^{1+s}} . 
			\end{split}
		\end{equation}
		By the flatness assumptions~\eqref{eq::flatness_harnack_lemma}, we have that
		\begin{equation*}
			\begin{split}
				&\partial E\cap B_{2^j}(y) \subseteq \partial E\cap C^{\nu_{j+1}}_{2^{j+1}} \subseteq \big\{|x\cdot\nu_{j+1}| \leq a 2^{(j+1)(\alpha+1)}\big\}\\
				\text{and }\quad  &|y\cdot\nu_{j+1}|  \leq a 2^{(j+1)(\alpha+1)},
			\end{split}
		\end{equation*}
		hence
		\begin{equation} \label{eq::flatness_harnack2}
			\partial E\cap B_{2^j}(y) \subseteq \big\{|(x-y)\cdot\nu_{j+1}| \leq 2 a 2^{(j+1)(\alpha+1)}\big\} .
		\end{equation}
		So, if we define 
		$$ D := \left(B_{2^j}(y)\setminus B_{2^{j-1}}(y)\right) \setminus \big\{ |(x-y)\cdot \nu_{j+1}| \leq 2a2^{(j+1)(\alpha+1)} \big\} ,$$
		by symmetry, we have that
		\begin{equation*}
			\int_{D} \frac{\chi_E(x)-\chi_{E^c}(x)}{|x-y|^{n+s}} dx = 0.
		\end{equation*}
		
		Thus, we infer that
		\begin{equation} \label{eq::harnack_proof2}
			\begin{split}
				&	\left| \int_{B_{2^j}(y)\setminus B_{2^{j-1}}(y)}  \frac{\chi_E(x)-\chi_{E^c}(x)}{|x-y|^{n+s}}\,dx  \right| 
				\leq \int_{B_{2^j}(y)\setminus B_{2^{j-1}}(y)}  \frac{\chi_{\left\{|(x-y)\cdot\nu_{j+1}|\leq2 a2^{(j+1)(\alpha+1)}\right\}}(x)}{|x-y|^{n+s}}\,dx \\
				&\qquad\qquad \leq \omega_n \int_{2^{j-1}}^{2^j} \frac{4 a 2^{(j+1)(\alpha+1)}r^{n-2}}{r^{n+s}} \,dr 
				= C a \int_{2^{j-1}}^{2^j} \frac{dr}{r^{1+s-\alpha}} ,
			\end{split}
		\end{equation}
		for some~$C>0$, depending only on~$n$ and~$\alpha$.
		
		Therefore, from~\eqref{eq::harnack_proof1}, \eqref{eq::harnack_proof2}, and the fact that~$\alpha\in(0,s)$, we deduce that
		\begin{equation} \label{eq::harnack_proof3}
			\begin{split}
				&\left| \int_{B_{1/2}^c(y)}  \frac{\chi_E(x)-\chi_{E^c}(x)}{|x-y|^{n+s}}\,dx  \right|
				\leq \sum_{j=0}^{k} 	C a \int_{2^{j-1}}^{2^j} \frac{dr}{r^{1+s-\alpha}} + \omega_n\int_{2^k}^{+\infty}  \frac{dr}{r^{1+s}}    \\
				&\qquad\leq C a \int_{1/2}^{+\infty} \frac{dr}{r^{1+s-\alpha}}+ \frac{\omega_n}{s}2^{-sk} 
				\leq c_1 \big(a+2^{-\alpha k}\big) \leq c_1 a,
			\end{split}
		\end{equation}
		for some~$c_1>0$ depending only on~$n$, $s$, and~$\alpha$.
		
		Now, let us denote by~$P$ the subgraph of the tangent parabola~$x_n=y_n-\frac{a}{2}|x'-y'|^2$. Then,
		\begin{equation} \label{eq::harnack_proof4}
			\begin{split}
				&\limsup_{\rho\to0} \int_{B_{1/2}(y)\setminus B_{\rho}(y)} \frac{\chi_E(x)-\chi_{E^c}(x)}{|x-y|^{n+s}}\,dx  \\
				&\qquad\qquad\geq \int_{B_{1/2}(y)} \frac{\chi_{E\setminus P}(x)}{|x-y|^{n+s}}\,dx
				+ \int_{B_{1/2}(y)} \frac{\chi_{P}(x)-\chi_{P^c}(x)}{|x-y|^{n+s}}\,dx ,
			\end{split}
		\end{equation}
		Notice that, since~$P$ is smooth, both terms of the right-hand side of~\eqref{eq::harnack_proof4} are well-defined. Now, we estimate these terms.

		Since~$y\in C_{\delta_0}\cap\left\{-a\leq x_n\leq -(1-\delta_0^2)a\right\}$, up to taking~$\delta_0$ (and hence~$a$) small enough, we have that~$C_{\delta_0,a}\subseteq B_{1/2}(y)$, and
		\begin{equation} \label{eq::harnack_proof5}
			|x-y|^{-n-s}\geq (3\delta_0)^{-n-s},
		\end{equation}
		for every~$x\in C_{\delta_0,a}$.
		
		By~\eqref{eq::half_measure_cylinder}, we also have that
		\begin{equation*}
			\begin{split}
				&|(E\setminus P)\cap C_{\delta_0,a}| = |E\cap C_{\delta_0,a}|-|P\cap C_{\delta_0,a}| \\
				&\qquad\geq \frac{1}{2}|C_{\delta_0,a}|-\left|B^{n-1}_{\delta_0}\times\{-a\leq x_n\leq -(1-\delta_0^2)a\}\right|\\
				&\qquad =\omega_{n-1}\delta_0^{n-1}a(1-\delta_0^2).
			\end{split}
		\end{equation*}
		Hence, $|(E\setminus P)\cap C_{\delta_0,a}| \geq \frac{1}{2}\omega_{n-1}\delta_0^{n-1}a$, whenever~$\delta_0^2<1/2$.
		Consequently, using this together with~\eqref{eq::harnack_proof5}, we obtain that
		\begin{equation} \label{eq::harnack_proof6}
			\int_{B_{1/2}(y)}\frac{\chi_{E\setminus P}(x)}{|x-y|^{n+s}}\,dx
			\geq \int_{C_{\delta_0,a}}\frac{dx}{|x-y|^{n+s}} 
			\geq \frac{|(E\setminus P)\cap C_{\delta_0,a}|}{(3\delta_0)^{n+s}}
			\geq c_2a\delta_0^{-1-s},
		\end{equation}
		for some~$c_2>0$ depending on~$n$ and~$s$.
		
		In addition, since~$P$ is smooth,
		\begin{equation} \label{eq::harnack_proof7}
			\begin{split}
				&\int_{B_{1/2}(y)} \frac{\chi_{P}(x)-\chi_{P^c}(x)}{|x-y|^{n+s}}\,dx 
				\geq c_3 \int_{0}^{1/2} r^{n-2} \left(\int_{0}^{-\frac{a}{2}r^2}\frac{dt}{(r^2+t^2)^{\frac{n+s}{2}}}\right) \,dr \\
				& \qquad\qquad\geq -c_3\frac{a}{2}\int_{0}^{1/2}\frac{r^n}{r^{n+s}}\,dr
				= -c_3a ,
			\end{split}
		\end{equation}
		for some~$c_3>0$ depending on~$n$ and~$s$, and possibly changing from line to line.
		To check this, see for instance~\cite[Section~4]{abatangelo_valdinoci_nonlocal_curv}.
		
		Then, using~\eqref{eq::harnack_proof6} and~\eqref{eq::harnack_proof7} in~\eqref{eq::harnack_proof4}, we come to
		\begin{equation}\label{eq::harnack_proof8}
			\limsup_{\rho\to0} \int_{B_{1/2}(y)\setminus B_{\rho}(y)}  \frac{\chi_E(x)-\chi_{E^c}(x)}{|x-y|^{n+s}}\,dx \geq (c_2-c_3\delta_0^{1+s})a\delta_0^{-1-s} \geq c_4 a\delta_0^{-1-s},
		\end{equation}
		for some~$c_4>0$ depending on~$n$, $s$, and~$\alpha$.
		
		Therefore, thanks to~\eqref{eq::harnack_proof3}, \eqref{eq::harnack_proof8}, and Theorem~\ref{th::ELeq}, we infer that
		\begin{equation*} \label{eq::harnack_proof9}
			2^{-sk}\Lambda \geq \limsup_{\rho\to0} \int_{B_\rho^c(y)} \frac{\chi_E(x)-\chi_{E^c}(x)}{|x-y|^{n+s}}\,dx \geq \big(c_4-c_1\delta_0^{1+s}\big)a\delta_0^{-1-s}.
		\end{equation*}
		However, this leads to a contradiction whenever we take~$a$ and~$\delta_0$ small enough.
		
		Furthermore, if~$\{x_n>a\}\cap C_1 \subseteq E$, up to considering the
		supergraph~$\widetilde{P}$ of the parabola~$x_n=y_n+\frac{a}{2}|x'-y'|^2$ tangent to~$E$ at some point~$y$ from above, the foregoing computations show that~\eqref{eq::inclusion_harnack2} holds true, concluding the proof.
	\end{proof}
	
	\begin{proof}[Proof of Theorem~\ref{th::harnack_ineq}]
		Suppose that~$|E\cap C_{\delta_0,a}|\geq \frac{1}{2}|C_{\delta_0,a}|$. Since~$E$ is~$(2^{-ks}\Lambda)$-minimal in~$B_{2^k}$, thanks to Lemma~\ref{lemma::partial_harnack}, we have either $C_{\delta_0}\cap\{x_n<-a(1-\delta_0^2)\}\subseteq E$ or~$C_{\delta_0}\cap\{x_n>a(1-\delta_0^2)\}\subseteq E$, from which it follows 
		
		\begin{align*}
			\mbox{either}\quad& \partial E\cap C_{\delta_0} \subseteq \big\{x_n\geq a(\delta_0^2-1)\big\}\\
			\mbox{or}\quad& \partial E\cap C_{\delta_0} \subseteq \big\{x_n\leq a(1-\delta_0^2)\big\},
		\end{align*}
		respectively.
		
		Additionally, if~$|E\cap C_{\delta_0,a}|< \frac{1}{2}|C_{\delta_0,a}|$, then necessarily~$|E^c\cap C_{\delta_0,a}|\geq \frac{1}{2}|C_{\delta_0,a}|$.
		
		Moreover, recalling Definition~\ref{def::super_sub_sol} and
		Lemma~\ref{lemma::almost_minimal_subsupersol}, we have that also~$E^c$ is an almost minimal set. Thus, using Lemma~\ref{lemma::partial_harnack} with~$E^c$, we deduce again that
		
		\begin{align*}
			\mbox{either}\quad& \partial E\cap C_{\delta_0} \subseteq \big\{x_n\geq a(\delta_0^2-1)\big\}\\
			\mbox{or}\quad& \partial E\cap C_{\delta_0} \subseteq \big\{x_n\leq a(1-\delta_0^2)\big\},
		\end{align*}
		as desired.
	\end{proof}
	
	As shown in~\cite{caffarelli_roquejoffre_savin_nonlocal}, when~$k$ is much larger than~$k_0$, we are able to apply iteratively the Harnack's Inequalities. Since it will be useful in the following section, here we retrace the argument.
	
	\begin{cor} \label{cor::iterative_arg}	
		Let~$E$ be~$(2^{-ks}\Lambda)$-minimizer for~$\Per_s$ in~$B_{2^k}$, for some~$\Lambda>0$ and~$k\geq k_0$ from Theorem~\ref{th::harnack_ineq}, and set~$a:=2^{-\alpha k}$. Suppose that 
		\begin{equation}  \label{eq::tail_control}
			\partial E\cap C^{\nu_j}_{2^j}\subseteq\big\{|x\cdot\nu_j|\leq a2^{j(1+\alpha)}\big\},\qquad {\mbox{for all }} j=0,\dots,k ,
		\end{equation}
		with~$\nu_0=e_n$.
		
		Then, there exists $\zeta>0$, depending only on~$n$, $s$, $\alpha$ , $\Lambda$, $k$, $k_0$, and $\delta_0$ such that, for every $x_0\in\partial E\cap C_{1/2}$ and for all~$j\leq\zeta$, there exists~$z_j(x_0)\in\R$ such that
		\begin{equation} \label{eq::z_j_slab}
			\partial E\cap C_{(\delta_0/2)^j}(x_0) \subseteq \left\{|x_n-z_j(x_0)|\leq a\left(1-\frac{\delta_0^2}{2}\right)^j\right\}.
		\end{equation}
	\end{cor}
	
	\begin{proof}
		Let us start from the case~$x_0=0$ for simplicity. 
		Let also $\delta_{0}$ and $k_0$ be as in Theorem~\ref{th::harnack_ineq}, and take $k>k_0$. Notice that the conclusion in Theorem~\ref{th::harnack_ineq} is still verified for any $\delta<\delta_0$. Thus, we can
		assume that $\delta_0=2^{1-m_0}$, for some $m_0\in\N$. Now we define
		\begin{equation} \label{eq::zeta_choice}
			\widetilde{\zeta} := \left(\log \frac{2-\delta_0^2}{\delta_0}\right)^{-1}\log\left(\frac{a_0}{a}\right) =\Big(\log
			\big(2^{m_0}(1-2^{1-2m_0})\big)\Big)^{-1}\log\left(\frac{a_0}{a}\right),
		\end{equation}
		where $a_0:=2^{-k_0\alpha}$ and $a:=2^{-k\alpha}$.
		
		Let us make an explicit remark on our choice of parameters. On the one hand, we want to apply iteratively the Harnack inequality in Theorem~\ref{th::harnack_ineq}. This is possible as long as the assumption in~\eqref{eq::harnack_ineq1} is satisfied, which occurs exactly as long as $j\leq \widetilde{\zeta}$. On the other hand, we pick $a$, $a_0$, and $\delta_0$ so small that the setting under consideration here is coherent with the almost-minimality. In particular, a dilation of a $\Lambda$-minimal set by a parameter greater or equal than $1$ is again $\Lambda$-minimal in the same domain (see \eqref{eq::still_almost_min} below for a detailed discussion of this property). 
		
		Now, we claim that there exists a collection $\{z_j\}_j\subseteq[-a,a]$, with $j\in\{0,\dots,\widetilde{\zeta}\}$, such that~$z_0=0$,
		\begin{equation} \label{eq::z_j_slab_0}
			\partial E\cap C_{2^{-m_0j}} \subseteq \left\{|x_n-z_j|\leq a\left(1-2^{1-2m_0}\right)^j\right\},
		\end{equation}
		and, for all~$j\in\{0,\dots,\widetilde{\zeta}-1\}$,
		the following inequalities hold true
		\begin{equation} \label{eq::z_j_nested}
			\begin{split}
				&z_j-a\left(1-2^{1-2m_0}\right)^j\leq z_{j+1}-a\left(1-2^{1-2m_0}\right)^{j+1} \\
				&\qquad\qquad\leq z_{j+1}+a\left(1-2^{1-2m_0}\right)^{j+1} \leq z_j+a\left(1-2^{1-2m_0}\right)^j.
			\end{split}
		\end{equation}
		
		In fact, to prove~\eqref{eq::z_j_slab_0} and~\eqref{eq::z_j_nested}, it is convenient first
		to check an auxiliary implication.
		For short, we denote 
		by~${\mathcal{P}}(j)$
		the validity of~\eqref{eq::z_j_slab_0} for the index~$j$
		and by~${\mathcal{Q}}(j)$
		the validity of~\eqref{eq::z_j_nested} for the index~$j$, and we claim that, given~$j\in\{0,\dots,\widetilde\zeta-1\}$,
		\begin{equation}\label{AUSXOLjd9o438m5yu}\begin{split}&
				{\mbox{if~${\mathcal{P}}(i)$ for all indices~$i\in\N\cap[0,j]$}}
				\\&{\mbox{and~${\mathcal{Q}}(i)$ for all indices~$i\in\N\cap[0,j-1]$,}}
				\\&{\mbox{then~${\mathcal{Q}}(j)$ and ${\mathcal{P}}(j+1)$,}}\end{split}\end{equation}
		for a suitable choice of~$z_{j+1}$.
		
		We now prove~\eqref{AUSXOLjd9o438m5yu}.
		Later on, we will use~\eqref{AUSXOLjd9o438m5yu}
		to check~\eqref{eq::z_j_slab_0}
		and~\eqref{eq::z_j_nested}.
		
		The proof of~\eqref{AUSXOLjd9o438m5yu} relies on the 
		Harnack type result presented in Theorem~\ref{th::harnack_ineq} and goes as follows.
		Assume~${\mathcal{P}}(i)$ for all indices~$i\in\N\cap[0,j]$
		and~${\mathcal{Q}}(i)$ for all indices~$i\in\N\cap[0,j-1]$ (the latter, in particular,
		establishes the existence of~$z_j$).
		
		We define \begin{equation}\label{OIJHs0wpkt-957mn8iu}
			\widetilde{E}:=2^{m_0j} (E-z_{j}e_n)\end{equation} and we show that, when~$k$ is as in the statement of Corollary~\ref{cor::iterative_arg},
		\begin{equation}\label{eq::still_almost_min}
			{\mbox{$\widetilde{E}$ is a $(2^{-ks}\Lambda)$-minimizer in $B_{2^k}$.}}\end{equation} 
		Indeed, let $\widetilde{F}\subseteq\R^n$ be such that $\widetilde{F}\setminus B_{2^k}=\widetilde{E}\setminus B_{2^k}$, and define $F:=z_{j}e_n+2^{-m_0j}\widetilde{F}$. Then, we have
		\begin{equation}\label{98nc7vbk7Xui}
			F\setminus B_{2^{k-m_0j}}(z_{j}e_n) = E\setminus B_{2^{k-m_0j}}(z_{j}e_n) .
		\end{equation}
		Since $|z_{j}|\leq a$ by construction, we also have $B_{2^{k-m_0j}}(z_{j}e_n) \subseteq B_{2^k}$, so that~\eqref{98nc7vbk7Xui} yields that
		\begin{equation*}
			F\setminus B_{2^k} = E \setminus B_{2^k}.
		\end{equation*}
		Hence, taking advantage of the the scaling properties of the $s$-perimeter and the almost minimality of $E$, we obtain
		\begin{equation*} 
			\begin{split}
				&\Per_s(\widetilde{E},B_{2^k})-\Per_s(\widetilde{F},B_{2^k}) = 2^{m_0j(n-s)}\Big(\Per_s(E,B_{2^{k-m_0j}}(z_{j}e_n))-\Per_s(F,B_{2^{k-m_0j}}(z_{j}e_n))\Big)\\
				&\qquad =2^{m_0j(n-s)}\Big(\Per_s(E,B_{2^k})-\Per_s(F,B_{2^k})\Big) \leq 2^{m_0j(n-s)}2^{-ks}\Lambda |E\Delta F| \\
				&\qquad = 2^{-m_0 j s}2^{-ks}\Lambda |\widetilde{E}\Delta \widetilde{F}| \leq 2^{-ks}\Lambda |\widetilde{E}\Delta \widetilde{F}|,			
			\end{split} 
		\end{equation*}
		which establishes~\eqref{eq::still_almost_min}.
		
		Now we define $\widetilde{a}:=a2^{m_0 j}(1-2^{1-2m_0})^j$ and use again the notation for~$\widetilde{E}$ in~\eqref{OIJHs0wpkt-957mn8iu}.
		We claim that there exists a collection $\{\nu'_i\}_i\subseteq\S^{n-1}$ such that,
		for every~$i\in\{0,\dots,k\}$,
		\begin{equation} \label{eq::improved_trap}
			\partial\widetilde{E}\cap C^{\nu'_i}_{2^i}  \subseteq \big\{|x \cdot \nu'_i|\leq \widetilde{a} 2^{i(1+\alpha)}\big\}.
		\end{equation}
		To show \eqref{eq::improved_trap}, 
		we first notice that when~$j=0$ we have that~$\widetilde a=a$ and~$\widetilde{E}=E$
		and therefore~\eqref{eq::improved_trap} in this case boils down to~\eqref{eq::tail_control}.
		
		Hence, to prove~\eqref{eq::improved_trap}, one can assume that
		\begin{equation}\label{9qmcv90b34t5bn7.0-1}
			j\ge1.\end{equation}
		We distinguish the cases $i=0$, $1\leq i\leq j$, and $j<i\leq k$. 
		
		First, notice that, after rescaling and translating~${\mathcal{P}}(h)$ for all indices~$h\in\N\cap[0,j]$, we obtain
		\begin{equation} \label{eq::j_leq_ell}
			\partial \widetilde{E}\cap C_{2^{m_0(j-h)}} \subseteq \left\{|x_n-2^{m_0j}(z_h-z_j)|\leq a2^{m_0j}\left(1-2^{1-2m_0}\right)^h\right\}.
		\end{equation}
		Therefore, taking $h:=j$ in~\eqref{eq::j_leq_ell}, it follows that
		\begin{equation*} 
			\partial \widetilde{E}\cap C_1 \subseteq \left\{|x_n| \leq a2^{m_0j}\left(1-2^{1-2m_0}\right)^j\right\} = \left\{|x_n| \leq \widetilde{a}\right\}.
		\end{equation*}
		This shows~\eqref{eq::improved_trap}
		when~$i=0$.
		
		We now check the validity of~\eqref{eq::improved_trap}
		when~$1\leq i\leq j$. By~${\mathcal{Q}}(h)$ for all indices~$h\in\N\cap[0,j-1]$, we have that
		\begin{equation*}
			\begin{split}
				&z_h-a\left(1-2^{1-2m_0}\right)^h\leq z_{j}-a\left(1-2^{1-2m_0}\right)^{j} \\
				&\qquad\qquad\leq z_{j}+a\left(1-2^{1-2m_0}\right)^{j} \leq z_h+a\left(1-2^{1-2m_0}\right)^h.
			\end{split}
		\end{equation*}
		This entails that
		\begin{equation*}
			\begin{split}
				&z_h 
				\leq z_{j}-a\left(1-2^{1-2m_0}\right)^{j}
				+a\left(1-2^{1-2m_0}\right)^h	
				\leq z_j +a\left(1-2^{1-2m_0}\right)^{h}, \\
				\mbox{and}\quad& z_j 
				\leq z_h+a\left(1-2^{1-2m_0}\right)^h-a\left(1-2^{1-2m_0}\right)^{j} 
				\leq z_h +a\left(1-2^{1-2m_0}\right)^{h}.
			\end{split}
		\end{equation*}
		Therefore,
		\begin{equation*}
			|z_h-z_j| \leq a\left(1-2^{1-2m_0}\right)^{h}
		\end{equation*}
		and hence
		\begin{equation*} 
			2^{m_0j}|z_h-z_j| + a2^{m_0j}\left(1-2^{1-2m_0}\right)^{h}\leq a2^{1+m_0j}\left(1-2^{1-2m_0}\right)^{h} .
		\end{equation*}
		
		Combining this with \eqref{eq::j_leq_ell}, and using the notation $i:=j-h$, we come to
		\begin{equation*}
			\partial\widetilde{E} \cap C_{2^i} \subseteq \partial\widetilde{E} \cap C_{2^{m_0 i}} \subseteq \left\{ |x_n|\leq a2^{1+m_0j}(1-2^{1-2m_0})^{j-i}\right\} = \left\{|x_n| \leq \widetilde{a}2(1-2^{1-2m_0})^{-i}\right\},
		\end{equation*}
		for every $i\in\N\cap[1,j]$.
		
		Now, let $\sigma>0$ be such that $1-2^{1-2m_0} = 2^{-\sigma}$. Moreover, up to taking $m_0$ large enough, we assume \begin{equation}\label{9qmcv90b34t5bn7.0-2}\sigma<\alpha .\end{equation} Then, we have
		\begin{equation*}
			2(1-2^{1-2m_0})^{-i} = 2^{1+i\sigma} \leq 2^{1+i\alpha} \leq 2^{i(1+\alpha)}.
		\end{equation*}
		Thus, it follows from the last two displays that, for every~$i\in\N\cap[1,j]$,
		\begin{equation*} 
			\partial\widetilde{E} \cap C_{2^i}  \subseteq \left\{|x_n|\leq \widetilde{a} 2^{i(1+\alpha)}\right\}.
		\end{equation*}
		This proves~\eqref{eq::improved_trap}
		when~$1\leq i\leq j$ with~$\nu'_i:=e_n$.
		
		To conclude the proof of \eqref{eq::improved_trap}, we are only left to deal with the case~$j+1\le i\le k$. 
		For this, we let~$h:=i-j$ and remark that~$h\in\N\cap[1,k-j]$.
		Notice that, by the flatness assumption in~\eqref{eq::tail_control} and the definition of $\widetilde{E}$, we have that
		\begin{equation}\label{eq::improved_trap_2_1}
			\partial\widetilde{E}\cap C^{\nu_h}_{2^{m_0j+h}}(2^{m_0j}z_j e_n) \subseteq \left\{|x\cdot\nu_h|\leq 2^{m_0j}|z_j| + a2^{m_0j+h(1+\alpha)}\right\}.
		\end{equation}
		
		In addition, for every $x\in C^{\nu_h}_{2^{h+j}}$, we have
		\begin{equation*}\begin{split}&
				|(x-2^{m_0 j}z_j e_n)\cdot\nu_h| \leq 
				|x\cdot\nu_h|+2^{m_0 j}|z_j |
				\le2^{h+j}+2^{m_0 j}a\\&\qquad=
				2^{h+j} + 2^{m_0 j-k\alpha} \leq 2^{h+m_0j-1}+2^{1-h-k\alpha}2^{h+m_0j-1}\leq 2^{m_0j+h},\end{split}
		\end{equation*}
		that yields $C^{\nu_h}_{2^{h+j}}\subseteq C^{\nu_h}_{2^{m_0j+h}}(2^{m_0j}z_j)$.
		
		This and~\eqref{eq::improved_trap_2_1} give that
		\begin{equation}\label{eq::improved_trap_2_2}\partial\widetilde{E}\cap C^{\nu_h}_{2^{h+j}} \subseteq \left\{|x\cdot\nu_h|\leq 2^{m_0j}|z_j| + a2^{m_0j+h(1+\alpha)}\right\}.\end{equation}
		
		Moreover, recalling that $|z_j|\leq a$, we see that 
		\begin{equation*} 
			\begin{split}
				&2^{m_0j}|z_j|+a2^{m_0j+h(1+\alpha)} \leq a2^{m_0j}(1+2^{h(1+\alpha)}) \\&\qquad\leq a2^{m_0j}2^{1+h(1+\alpha)} = \widetilde{a}2^{1+h(1+\alpha)}(1-2^{1-2m_0})^{-j} \\
				&\qquad= \widetilde{a}2^{1+h(1+\alpha)+\sigma j}=\widetilde{a}2^{(h+j)(1+\alpha)}2^{1+\sigma j-j(1+\alpha)} \leq \widetilde{a}2^{(h+j)(1+\alpha)},
			\end{split}
		\end{equation*}
		thanks to~\eqref{9qmcv90b34t5bn7.0-1} and~\eqref{9qmcv90b34t5bn7.0-2}. 
		
		Therefore, recalling that $i=j+h$ and using~\eqref{eq::improved_trap_2_2}, we obtain
		\begin{equation*} 
			\partial\widetilde{E}\cap C^{\nu'_i}_{2^{i}} \subseteq \left\{|x\cdot\nu'_i| \leq \widetilde{a}2^{i(1+\alpha)}\right\}.
		\end{equation*}
		for every $i\in\{j+1,\dots,k\}$, where $\nu'_i:=\nu_h$.
		This completes the proof of~\eqref{eq::improved_trap}.
		
		By virtue of~\eqref{eq::still_almost_min} and~\eqref{eq::improved_trap}, we can employ
		Theorem~\ref{th::harnack_ineq} and conclude that
		\begin{align*}
			\mbox{either}\quad&\partial\widetilde{E}\cap C_{2^{-m_0}}\subseteq\partial\widetilde{E}\cap C_{2^{1-m_0}} \subseteq \big\{x_n\leq\widetilde{a} (1-2^{2-2m_0})\big\}\\
			\mbox{or}\quad&\partial\widetilde{E}\cap C_{2^{-m_0}}\subseteq\partial\widetilde{E}\cap C_{2^{1-m_0}} \subseteq \big\{x_n\geq -\widetilde{a}(1-2^{2-2m_0})\big\}.
		\end{align*}

		Rescaling and translating back, and recalling the definition of $\widetilde{a}$, we infer that
		\begin{align}
			\label{eq::harnack_ell_1}\mbox{either}\quad&\partial E\cap C_{2^{-m_0(j+1)}} \subseteq \big\{a(1-2^{1-2m_0})^{j} \leq x_n- z_j\leq a(1-2^{2-2m_0})(1-2^{1-2m_0})^{j}\big\}\\
			\label{eq::harnack_ell_2}\mbox{or}\quad&\partial E\cap C_{2^{-m_0(j+1)}} \subseteq \big\{a(1-2^{2-2m_0})(1-2^{1-2m_0})^{j} \leq x_n- z_j\leq a(1-2^{1-2m_0})^{j}\big\}.
		\end{align}
		
		Suppose that \eqref{eq::harnack_ell_1} holds true. Then, setting $z_{j+1} := z_j- a2^{1-2m_0}(1-2^{1-2m_0})^{j}$, we deduce that
		\begin{equation} \label{eq::z_ell+1_nested}
			\begin{split}
				&z_j-a\left(1-2^{1-2m_0}\right)^j\leq z_{j+1}-a\left(1-2^{1-2m_0}\right)^{j+1} \\
				&\qquad\qquad\leq z_{j+1}+a\left(1-2^{1-2m_0}\right)^{j+1} \leq z_j+a\left(1-2^{1-2m_0}\right)^j
			\end{split}
		\end{equation}
		and
		\begin{equation} \label{eq::improved_trap_ell+1}
			\partial E\cap C_{2^{-m_0(j+1)}} \subseteq \left\{|x_n-z_{j+1}| \leq a(1-2^{1-2m_0})^{j+1}\right\}.
		\end{equation}
		Analogously, if \eqref{eq::harnack_ell_2} is verified, then~$z_{j+1} := z_j+ a2^{1-2m_0}(1-2^{1-2m_0})^{j}$ gives again~\eqref{eq::z_ell+1_nested} and~\eqref{eq::improved_trap_ell+1}.
		In any case, we have established~\eqref{eq::z_ell+1_nested}, which is precisely~${\mathcal{Q}}(j)$,
		as well as~\eqref{eq::improved_trap_ell+1}, which is precisely~${\mathcal{P}}(j+1)$: the proof of~\eqref{AUSXOLjd9o438m5yu} is thereby complete.
		
		We have now to check the claims in~\eqref{eq::z_j_slab_0} and~\eqref{eq::z_j_nested},
		which correspond, respectively, to~${\mathcal{Q}}(j)$ 
		with~$j\in\{0,\dots,\widetilde\zeta-1\}$
		and to~${\mathcal{P}}(j)$
		with~$j\in\{0,\dots,\widetilde\zeta\}$. We argue recursively.
		To start with, notice that~\eqref{eq::tail_control} gives~${\mathcal{P}}(0)$ with $z_j=0$. 
		This and~\eqref{AUSXOLjd9o438m5yu}
		imply~${\mathcal{Q}}(0)$ and~${\mathcal{P}}(1)$.
		Then, we use~${\mathcal{P}}(0)$, ${\mathcal{P}}(1)$,
		and~${\mathcal{Q}}(0)$, together with~\eqref{AUSXOLjd9o438m5yu},
		to deduce~${\mathcal{Q}}(1)$ and~${\mathcal{P}}(2)$.
		Then, we use~${\mathcal{P}}(0)$, ${\mathcal{P}}(1)$, ${\mathcal{P}}(2)$,
		${\mathcal{Q}}(0)$
		and~${\mathcal{Q}}(1)$, together with~\eqref{AUSXOLjd9o438m5yu},
		to deduce~${\mathcal{Q}}(2)$ and~${\mathcal{P}}(3)$.
		Then we keep iterating till we deduce~${\mathcal{Q}}(\widetilde\zeta-1)$
		and~${\mathcal{P}}(\widetilde\zeta)$, which proves~\eqref{eq::z_j_slab_0}
		and~\eqref{eq::z_j_nested}, as desired.
		
		This establishes the inclusion in~\eqref{eq::z_j_slab}
		when~$x_0=0$.
		\smallskip
		
		A similar argument works more generally for every~$x_0\in\partial E\cap C_{1/2}\setminus\{0\}$. Indeed, we have
		$$ \partial E\cap C^{\nu_j}_{2^{j-1}}(x_0) \subseteq \partial E\cap C^{\nu_j}_{2^j} \subseteq \big\{|x\cdot \nu_j|\leq a2^{j(1+\alpha)}\big\},\qquad{\mbox{for all }} j=0,\dots,k ,$$
		thus
		$$  \partial E\cap C^{\nu_j}_{2^{j-1}}(x_0) \subseteq \big\{|(x-x_0)\cdot\nu_j|\leq 2^{2+\alpha}a2^{(j-1)(1+\alpha)}\big\},\qquad {\mbox{for all }} j=0,\dots,k .$$
		Moreover, if~$E$ is almost minimal in~$B_{2^k}$, then the translated set~$E-x_0$ is almost minimal in~$B_{2^{k-1}}$. Therefore, setting $a':=2^{2+\alpha}a$ and applying the argument discussed above to~$E-x_0$, we obtain
		
		\begin{equation*}
			\partial (E-x_0)\cap C_{(\delta_0/2)^j} \subseteq \left\{|x_n-z_j|\leq a'\left(1-\frac{\delta_0^2}{2}\right)^j\right\}
		\end{equation*}
		as long as $j\leq\zeta$, where
		\begin{equation*}
			\zeta:=\Big(\log\big(2^{m_0}(1-2^{1-2m_0})\big)\Big)^{-1}\log\left(\frac{a_0}{a'}\right),
		\end{equation*}
		Namely,		
		\begin{equation*}
			\partial E\cap C_{(\delta_0/2)^j}(x_0) \subseteq \left\{|x_n-\widetilde{z}_j|\leq a'\left(1-\frac{\delta_0^2}{2}\right)^j\right\},
		\end{equation*}
		where $\widetilde{z}_j:=z_j+|x_0|$, as long as~$j\leq\zeta$.
	\end{proof}
	
	\subsection{Improvement of flatness} \label{sec::improv_flat}
	
	The Harnack type result discussed in Theorem~\ref{th::harnack_ineq} constitutes our main tool to prove that if~$E$ is an almost minimal set with controlled flatness in every ball of radius~$2^{-j}$, for every~$j\leq k_0$, then its flatness improves in every ball of smaller radius. We refer to this result as improvement of flatness.
	
	\begin{theorem}[Improvement of flatness] \label{th::improv_flat}
		Let~$E$ be~$\Lambda$-minimal in~$B_1$, for some~$\Lambda\geq0$, and assume that~$0\in\partial E$. 
		
		Then, there exists~$k_0\in\N$, depending only on~$n$, $s$, $\alpha$ and~$\Lambda$, such that if 
		\begin{equation} \label{eq::flat_improve_hp}
			\partial E\cap B_{2^{-j}}\subseteq\big\{|x\cdot\nu_j|\leq 2^{-j(1+\alpha)}\big\},\qquad {\mbox{for all~$j\leq k$ and for some~$k\geq k_0$,}}
		\end{equation}
		where~$\{\nu_j\}_{j=0}^k$ is a family of unit vectors with~$\nu_0=e_n$, then there exist unit vectors~$\{\nu_j\}_{j\geq k+1}$ such that 
		\begin{equation*}
			\partial E\cap B_{2^{-j}}\subseteq\big\{|x\cdot\nu_j|\leq2^{-j(1+\alpha)}\big\},\qquad {\mbox{for all }} j\geq k+1.
		\end{equation*}
	\end{theorem}
	
	Notice that, thanks to the scaling properties of the~$s$-perimeter, if~$E$ is a~$\Lambda$-minimal set in~$B_1$, then~$2^k E$ is a~$(2^{-ks}\Lambda)$-minimal set in~$B_{2^k}$. Moreover, the flatness assumptions in~\eqref{eq::flat_improve_hp} become
	\begin{equation*}
		\partial(2^k E) \cap B_{2^{k-j}} \subseteq \big\{|x\cdot\nu_j| \leq 2^k2^{-j(1+\alpha)}\big\} = \big\{|x\cdot\nu_j| \leq 2^{-k\alpha}2^{(k-j)(1+\alpha)}\big\} .
	\end{equation*}
	Thus, we deal with a rescaled version of the improvement of flatness, from which Theorem~\ref{th::improv_flat} follows with an iterative argument. 
	
	\begin{theorem}[Rescaled improvement of flatness] \label{th::improv_flat_rescaled}
		Let~$E$ be~$(2^{-ks}\Lambda)$-minimal in~$B_{2^k}$, for some~$\Lambda\geq0$, and assume that~$0\in\partial E$.
		
		Then, there exists~$k_0\in\N$,
		depending only on~$n$, $s$, $\alpha$, and~$\Lambda$,
		such that if, for some~$k\geq k_0$, it holds that 
		\begin{equation} \label{eq::rescaled_flat_improve_hp}
			\partial E\cap C_{2^j}\subseteq \big\{|x\cdot\nu_j|\leq 2^{-k\alpha}2^{j(1+\alpha)}\big\},\quad {\mbox{ for all }} j\in\{ 0,\dots,k\},
		\end{equation}
		then there exists a unit normal~$\nu_{-1}$ such that 
		\begin{equation*}
			\partial E\cap C_{1/2}\subseteq \big\{|x\cdot\nu_{-1}|\leq 2^{-k\alpha-\alpha-1}\big\}.
		\end{equation*}
	\end{theorem}
	
	The strategy to prove Theorem~\ref{th::improv_flat_rescaled} relies on a contradiction argument. In particular, we consider a sequence of almost minimal sets~$\{E_k\}_k$ such that~$E_k$ is~$2^{-(k-j)\alpha}$-flat in~$C_{2^j}$, but~$\partial E_k$ cannot fit in any cylinder of height~$2^{-(k+1)\alpha}$ in~$C_{1/2}$. Our goal is to show that there exists a subsequence~$\{E_{k_l}\}_l$ uniformly convergent to a subgraph of some linear function~$u$, which leads to a contradiction.
	
	With this intent, let us consider a sequence of~$(2^{-ks}\Lambda)$-minimal sets~$\{E_k\}_k$, for some~$\Lambda\geq0$. For every~$k\in\N$, define~$a_k:=2^{-k\alpha}$, and assume that there exists~$d_k\in\N$ such that
	\begin{equation}\label{eq::flatness_Ek}
		\partial E_k\cap C_{2^h} \subseteq \big\{|x\cdot\nu_h^k|\leq a_k2^{h(1+\alpha)}\big\},\quad{\mbox{for all }}h\in\{ 0,\dots, h_k\} ,
	\end{equation} 
	with~$\nu_0^k=e_n$ for every~$k$.
	
	Moreover, let us denote by~$E_k^*$ the vertical dilation of~$E_k$ of a factor~$a_k$, i.e.
	$$ E_k^*:=\big\{(x',x_n)\;{\mbox{ s.t. }}\; (x',a_kx_n)\in E_k\big\}.$$
	To proceed with this argument, we need the following technical lemma that provides H\"older estimates for every set~$E_k^*$.
	
	\begin{lemma}[H\"older estimates] \label{lemma::holder_est}
		Let~$\{E_k\}_k$ be a sequence of~$(2^{-ks}\Lambda)$-minimal sets, for some~$\Lambda\geq0$, satisfying~\eqref{eq::flatness_Ek}.
		
		For any~$r\geq1$, let
		$$ A_k^r:=\big\{(x',x_n/a_k)\;{\mbox{ s.t. }}\; (x',x_n)\in\partial E_k\cap C_r\big\} \quad\text{ for every } k\in\N.$$
		Assume that there exists a positive constant~$C$ independent of~$k$ such that~$A_k^r\subseteq \{|x_n|\leq C\}$, for all~$k\in\N$.
		
		Then, there exists~$k_0\in\N$, depending only on~$n$, $s$, $\alpha$, and~$\Lambda$, such that, for all~$ k\geq k_0$, 
		\begin{equation}\label{eq::holder_est}
			A_k^r\cap\{|x'|<r/2\}\subseteq \subgr(g_k^+)\setminus \subgr(g_k^-),
		\end{equation}
		where~$g_k^-$ and~$g_k^+$ are H\"older continuous functions with modulus of continuity~$\theta t^\eta$, for some~$\eta\in(0,1)$ and~$\theta>0$ depending only on~$C$ above and~$\delta_0$ from Theorem~\ref{th::harnack_ineq}.
	\end{lemma}
	
	\begin{proof}
		Let~$k\in\N$, and take~$x_0\in A_k^r\cap\{|x'|<r/2\}$. Define 
		$$ \Sigma_{k,j}:=A_k^r \cap \left\{|x'-x_0'|<\frac{1}{2}\left(\frac{\delta_0}{2}\right)^j\right\}.$$
		Notice that, since 
		\begin{equation*}
			\partial E^*_k \cap C_{2^h}\subseteq \big\{|x\cdot\nu_h^k|\leq 2^{h(1+\alpha)}\big\},\quad{\mbox{for all }}h\in\{ 0,\dots,h_k\},
		\end{equation*}
		then also
		\begin{equation*}
			\Sigma_{k,j} = \Sigma_{k,j}\cap C_{2^h} \subseteq \big\{|x\cdot\nu_h^k|\leq 2^{h(1+\alpha)}\big\},\quad{\mbox{for all }}h\in\{ 0,\dots,h_k\}.
		\end{equation*}

		Accordingly, as we have shown in Corollary~\ref{cor::iterative_arg}, applying iteratively the Harnack's Inequality (Theorem~\ref{th::harnack_ineq}), we obtain that
		\begin{equation} \label{eq::Sigma_trap}
			\Sigma_{k,j} \subseteq \left\{|x_n-z_j|< 2^{2+\alpha}\left(1-\frac{\delta_0^2}{2}\right)^j\right\},\qquad{\mbox{for all }}j\in\{ 0,\dots, j_k\} ,
		\end{equation}
		for some $z_j=z_j(x_0)$, with $|z_j|\leq1$ for all $j$ (and all $k$).
		Notice also that~$j_k\to+\infty$ as~$k\to+\infty$. 
		
		Now, we show that, for every~$x\in A_k^r\cap\{|x'|<r/2\}$,
		\begin{equation}\label{261BIS00}
			|x_n-z_{j_k}|<\theta\max\{|x'-x_0'|,b_k\}^\eta ,
		\end{equation}
		where~$b_k$ is an infinitesimal sequence of positive real numbers.
		
		For this, we observe that, given~$x \in A_k^r\cap\{|x'|<r/2\}$, we have that~$|x'-x_0'|<r$. We discern three cases depending on whether~$|x'-x_0'|$ is in between~$r$ and~$\frac{1}{2}$, in between~$\frac{1}{2}$ and~$\frac{1}{2}\left(\frac{\delta_0}{2}\right)^{j_k}$, or smaller than~$\frac{1}{2}\left(\frac{\delta_0}{2}\right)^{j_k}$. The second case will be useful to find suitable~$\eta$ and~$\theta$, the first case follows from the second one by a suitable modification, while the third case is related to the quantity~$b_k$.
		
		Let us start from the case~$\frac{1}{2}\left(\frac{\delta_0}{2}\right)^{j_k}<|x'-x_0'|<\frac{1}{2}$. In this situation, there exists~$j\leq j_k$ in~$\N$ such that
		$$ \frac{1}{2}\left(\frac{\delta_0}{2}\right)^{j+1}\leq|x'-x_0'|<\frac{1}{2}\left(\frac{\delta_0}{2}\right)^{j}.$$
		In other words,
		$$ j< \frac{\log(2|x'-x_0'|)}{\log(\delta_0/2)}\leq j+1.$$
		Moreover, from~\eqref{eq::Sigma_trap}, we have that~$|x_n-z_j|\leq 2^{2+\alpha} \left( 1-\frac{\delta_0^2}{2}\right)^j$.
		Therefore, setting
		\begin{equation}\label{yrueiow5674836547839gfhdjvfbdns}
			\eta:=\frac{\log(1-\frac{\delta_0^2}{2})}{\log(\delta_0/2)}\in(0,1),
		\end{equation} 
		we obtain that
		\begin{align*}
			&\left(1-\frac{\delta_0^2}{2}\right)^j 
			\leq\left (1-\frac{\delta_0^2}{2}\right)^{\frac{\log(2|x'-x_0'|)}{\log(\delta_0/2)} -1} 
			=\left(1-\frac{\delta_0^2}{2}\right)^{-1}\left(
			1-\frac{\delta_0^2}{2}\right)^{\eta\frac{\log(2|x'-x_0'|)}{\log(1-\delta_0^2/2)}}
			\\&\qquad\qquad\qquad
			=\left(1-\frac{\delta_0^2}{2}\right)^{-1}2^\eta|x'-x_0'|^\eta . 
		\end{align*}
		Thus, defining~$\theta:=2^\eta(1+2^{2+\alpha}) \left(1-\frac{\delta_0^2}{2}\right)^{-1}$ and also recalling \eqref{eq::z_j_nested}, we have that
		\begin{equation*} 
			|x_n-z_{j_k}| \leq  |x_n-z_j|+|z_j-z_{j_k}| \leq (1+2^{2+\alpha})\left(1-\frac{\delta_0^2}{2}\right)^j \leq \theta |x'-x_0'|^\eta,
		\end{equation*}
		which gives~\eqref{261BIS00} in this case.
		
		Now, if~$\frac{1}{2}\leq|x'-x_0'|<r$, we have that~$|x'-x_0'|^\eta\geq2^{-\eta}$, where~$\eta$ is as in~\eqref{yrueiow5674836547839gfhdjvfbdns}. Also,
		$$ |x_n-z_{j_k}|\leq |x_n|+|z_{j_k}| \leq 2C.$$
		Thus, provided that~$\theta\geq 2^{1+\eta}C$, we conclude that
		$$ |x_n-z_{j_k}|\leq  \theta |x'-x_0'|^\eta,$$
		as desired.
		
		If instead~$|x'-x_0'|\leq \frac{1}{2}\left(\frac{\delta_0}{2}\right)^{j_k}$, we observe that
		\begin{equation*}
			x\in \Sigma_{k,j_k} \subseteq  \Sigma_{k,j_k-1}\subseteq \left\{|x_n-z_{j_k-1}|<2^{2+\alpha} \left(1-\frac{\delta_0^2}{2}\right)^{j_k-1}\right\}. 
		\end{equation*}
		Accordingly, defining $b_k:=\left(1-\frac{\delta_0^2}{2}\right)^{j_k}$ and recalling \eqref{eq::z_j_nested}, we obtain $b_k\to0$ as~$k\to+\infty$, and
		
		\begin{equation*} 
			\begin{split}
				&|x_n-z_{j_k}|<2^{2+\alpha}\left(1-\frac{\delta_0^2}{2}\right)^{j_k-1} + |z_{j_k-1}-z_{j_k}| \leq (1+2^{2+\alpha})\left(1-\frac{\delta_0^2}{2}\right)^{j_k-1} \\
				&\qquad= (1+2^{2+\alpha})\left(1-\frac{\delta_0^2}{2}\right)^{-1}b_k \leq\theta b_k \leq\theta b_k^\eta,
			\end{split}
		\end{equation*}
		which completes the proof of~\eqref{261BIS00}.
		
		From \eqref{261BIS00}, we deduce that~$A_k^r\cap \left\{|x'-x_0'|<\frac{r}{2}\right\} $ lies above the graph of the H\"older continuous function 
		$$ f_k^-[x_0](x') := z_{j_k}- \theta\max\{b_k,|x'-x_0'|\}^{\eta},$$
		and below the graph of the H\"older continuous function 
		$$ f_k^+[x_0](x') := z_{j_k}+ \theta\max\{b_k,|x'-x_0'|\}^{\eta} .$$
		Taking the infimum and the supremum respectively over all~$x_0\in A_k^r\cap \left\{|x'-x_0'|<\frac{r}{2}\right\} $, we conclude that~$A_k^r\cap \left\{|x'-x_0'|<\frac{r}{2}\right\} $ lies between the graphs of the H\"older continuous functions
		\begin{align*}
			&g_k^-:=\inf_{x_0\in A_k^r\cap \left\{|x'-x_0'|<\frac{r}{2}\right\} } f_k^-[x_0]\\
			\text{and }\qquad &g_k^+:=\sup_{x_0\in A_k^r\cap \left\{|x'-x_0'|<\frac{r}{2}\right\} } f_k^+[x_0].\qedhere
		\end{align*}
	\end{proof}
	
	\begin{rem} \label{rem::unif_conv_holder}
		Notice that, given any $x\in A_k^r\cap \left\{|x'|<\frac{r}{2}\right\}$, we have
		$$ 0<g_k^+(x')-g_k^-(x')\leq z_{j_k}(x)+\theta b_k^\eta-(z_{j_k}(x)+\theta b_k^\eta) = 2\theta b_k^\eta .$$
		Therefore, since~$b_k\to0$ and~$j_k\to+\infty$, taking the limit as~$k\to+\infty$, we deduce from Ascoli-Arzel\`a's Theorem that, up to subsequences, $g_k^-$ and~$g_k^+$ uniformly converge to some~$\cont^{0,\eta}$-function~$u$. Hence, $A_k^r$ converges uniformly to~$\graph(u)$ in~$\{|x'|<r/2\}$.
	\end{rem}
	
	Now we prove that we can extract a subsequence of dilated sets~$\{E_{k_j}^*\}_j$ that is uniformly convergent to a linear function. This will lead us to a contradiction, proving Theorem~\ref{th::improv_flat_rescaled}. The procedure is carried through~\cite[Chapter~6]{caffarelli_roquejoffre_savin_nonlocal} and is articulated in the following lemmata.
	
	\begin{lemma}[Lemma 6.10, \cite{caffarelli_roquejoffre_savin_nonlocal}]\label{lemma::conv_unif}
		There exists a subsequence~$\{E_{k_j}^*\}_j$ that is uniformly convergent on every compact subset of~$\R^n$ to the subgraph of a H\"older continuous function~$u$ such that~$u(0)=0$, and for which there exists some~$c>0$ such that~$u(x')\leq c(1+|x'|^{1+\alpha})$, for every~$x'\in\R^{n-1}$.
	\end{lemma}
	
	\begin{lemma}[Lemma 6.11, \cite{caffarelli_roquejoffre_savin_nonlocal}]\label{lemma::linear_viscosity}
		Let~$u$ be a function such that the sequence of almost-minimizers~$\{E_k\}_k$ converges to~$\subgr(u)$ uniformly on every compact subset of~$\R^n$, as~$k\to+\infty$.
		
		Then, 
		$$ (-\Delta)^{\frac{1+s}{2}} u=0 \text{ in }\R^{n-1},$$
		in the viscosity sense.
	\end{lemma}
	
	\begin{lemma}[Proposition 6.7, \cite{caffarelli_roquejoffre_savin_nonlocal}] \label{lemma::CRS_harmonic_implies_linear}
		If~$(-\Delta)^{\frac{1+s}{2}} u=0$ in~$\R^{n-1}$ in the viscosity sense, and 
		$$|u(x)|\leq c\big(1+|x|^{1+\alpha}\big),$$ 
		then~$u$ is linear.
	\end{lemma}
	
	Since in our setting~$\{E_k\}_k$ are almost-minimal, but not necessarily~$s$-minimal sets, few modifications are needed in the proofs of Lemmata~\ref{lemma::conv_unif} and~\ref{lemma::linear_viscosity}. For the details, we refer to Appendices~\ref{sec::lemma_conv_unif} and~\ref{sec::lemma_linear_viscosity} respectively. 
	
	\begin{proof}[Proof of Theorem~\ref{th::improv_flat_rescaled}]
		Arguing by contradiction, suppose that result does not hold. Then, we construct a sequence~$\{E_k\}_k$ such that 
		\begin{equation}\label{iu6y54trstar463782}\begin{split}
				&\partial E_k\cap B_{2^j}\subseteq \big\{|x\cdot\nu_j|\leq2^{j+\alpha(j-k)}\big\},\quad {\mbox{for all }}j\in\{0,\dots, j_k\},\\
				\text{but}\quad &\partial E_k\cap B_{1/2}\not\subseteq \big\{|x\cdot\nu_{-1}|\leq2^{-1-\alpha-\alpha k}\big\}.
		\end{split}\end{equation}
		Thanks to Lemma~\ref{lemma::conv_unif}, we have that, up to considering a subsequence, $\{E_k\}_k$ converges uniformly in~$B_1$ to the subgraph of some H\"older-continuous function~$u$, and we set~$E:=\subgr(u)$. Moreover, by Lemma~\ref{lemma::linear_viscosity}, $u$ is linear, hence~$E$ is a half-space. Since the sequence~$\{E_k\}_k$ is uniformly convergent, $E_k$ lies in an~$\epsilon$-neighborhood of~$E$ in~$B_{1/2}$, for every~$k$ sufficiently large. Thus, up to making~$\epsilon$ sufficiently small, there exists some~$\overline{k}$ such that
		$$ \partial E_{\overline{k}}\cap B_{1/2}\subseteq \big\{|x\cdot\nu |\leq2^{-1-\alpha-\alpha \overline{k}}\big\} ,\quad\text{ for some } \nu\in \S^{n-1}. $$
		However, this contradicts~\eqref{iu6y54trstar463782}.
	\end{proof}
	
	\subsection{Proof of Theorem~\ref{th::holder_reg_almost_min}}
	Now, we show that every almost minimizer in~$B_1$ is a $\cont^{1,\alpha}$-surface in~$B_{1/2}$, for every~$\alpha\in(0,s)$.
	
	To this purpose, consider~$x_0\in\partial E\cap B_{1/2}$. Let~$k_0$ be as in Theorem~\ref{th::improv_flat} and choose~$\epsilon_0<2^{-k_0(1+\alpha)}$. In this way, the improvement of flatness result (Theorem~\ref{th::improv_flat}) holds for the translated set~$E-x_0$. In particular, there exists a sequence~$\{\nu_j\}_j$ of unit vectors such that
	\begin{equation} \label{eq::flat_inclusions}
		\partial E\cap B_{2^{-j}}(x_0) \subseteq \big\{|(x-x_0)\cdot\nu_j|\leq 2^{-j(1+\alpha)}\big\},
	\end{equation}
	for every~$j\in\{0,\dots,k_0\}$, with~$\nu_0=e_n$ . Thanks to Lemma~\ref{lemma::geom_estimate} in Appendix~\ref{sec::geom_idea}, we deduce from~\eqref{eq::flat_inclusions} that
	$$ |\nu_j-\nu_{j+1}|\leq c2^{-j\alpha}.$$
	Thus, the sequence~$\{\nu_j\}_j$ converges to some unit vector~$\nu(x_0)$ and
	\begin{equation*}
		\begin{split}
			&|\nu(x_0)-\nu_j|
			=\lim_{k\to+\infty}|\nu_{j+k}-\nu_{j}| \leq \lim_{k\to+\infty} \sum_{q=0}^{k-1}|\nu_{j+q+1}-\nu_{j+q}|\\
			&\qquad\qquad\leq \lim_{k\to+\infty} \sum_{q=0}^{k-1} 2^{-(j+q)\alpha} \leq c2^{-j\alpha}.
		\end{split}
	\end{equation*}
	In view of this, we obtain that
	$$ |(x-x_0)\cdot\nu(x_0)| \leq |(x-x_0)\cdot\nu_j|+|x-x_0|\,|\nu(x_0)-\nu_j|\leq c2^{-j(1+\alpha)} ,$$
	for every~$x\in\partial E\cap B_{2^{-j}}(x_0)$, namely
	\begin{equation}\label{eq::flat_nu0}
		\partial E\cap B_{2^{-j}}(x_0) \subseteq \big\{|(x-x_0)\cdot\nu(x_0)|<c2^{-j(1+\alpha)}\big\}.
	\end{equation}
	
	Now, for every points~$x_0$, $y_0\in\partial E\cap B_{1/2}$ such that~$x_0\in B_{2^{-j}}(y_0)\setminus B_{2^{-j-1}}(y_0)$, for some~$j\in\N$, we estimate the angle between the unit normals~$\nu(x_0)$ and~$\nu(y_0)$. Thanks to the flatness condition~\eqref{eq::flat_nu0}, we have that
	\begin{align*}
		\text{either }&\big\{(x-y_0)\cdot\nu(x_0)>-c2^{-j(1+\alpha)}\big\}\subseteq \big\{x\cdot\nu(y_0)<c2^{-j(1+\alpha)}\big\}\\
		\text{or }&\big\{x\cdot\nu(y_0)>-c2^{-j(1+\alpha)}\big\}\subseteq\big\{(x-y_0)\cdot\nu(x_0)<c2^{-j(1+\alpha)}\big\} \text{ in }B_{2^{-j}}(y_0).
	\end{align*}
	Therefore, thanks to Lemma~\ref{lemma::geom_estimate}, we have an estimate of the worst case scenario. In particular, 
	$$ |\nu(x_0)-\nu(y_0)|\leq C2^{-j\alpha}\leq 2C|x_0-y_0|^\alpha.$$
	This shows that~$\nu$ is H\"older continuous, that is sufficient to prove that~$\partial E\cap B_{1/2}$ is a~$\cont^{1,\alpha}$ surface (see for instance~\cite[Theorem~5.8]{MR0165073}).
	
	
	\section{Monotonicity formula for almost minimal sets
		and proofs of Proposition~\ref{prop::char_almost_min_ext} and Theorem~\ref{th::Phi_monotone}}
	\label{sec::monotonicity}
	
	Let us recall from the~$s$-minimal surfaces theory that 
	a monotonicity formula holds whenever a set~$E$ is a minimizer for~$\Per_s$, namely
	\begin{equation}\label{eq::monotinicty_sminimal}
		\Xi_E(r) := r^{s-n}\int_{B_r} z^{1-s}|\nabla\widetilde{u}|^2 \quad
		\text{ is increasing in }r,
	\end{equation}
	where~$\widetilde{u}$ is defined through an extension argument from~$\chi_E-\chi_{E^c}$ (see~\cite[Chapters~7-8]{caffarelli_roquejoffre_savin_nonlocal}). This result is particularly useful in classifying regular and singular points of~$s$-minimal surfaces. In order to develop an analogous argument for almost minimal sets, we need some preliminaries. 
	
	\subsection{Preliminaries for a monotonicity formula}
	In~\cite{caffarelli_silvestre_extension}, it is shown that for every function~$u:\R^n\to\R$ such that
	$$ \int_{\R^n} \frac{|u(x)|}{(1+|x|^2)^{\frac{n+s}{2}}}\,dx<+\infty,$$
	its extension to the~$(n+1)$-dimensional half-space~$\R^{n+1}_+$ 
	is a solution~$\widetilde{u}:\R^{n+1}_+\to\R$ of 
	\begin{equation}\label{eq::extension}
		\begin{cases}
			\divergence(z^{1-s}\nabla\widetilde{u})=0 \quad&\text{in }\R^{n+1}_+,\\
			\widetilde{u}=u \quad&\text{on }\{z=0\}.
		\end{cases}
	\end{equation}
	
	Moreover, as anticipated in~\eqref{TILDEUDEF}, $\widetilde{u}$ can be obtained as the convolution~$P*u$ between the trace function~$u$ and the Poisson kernel.
	
	Now, let us consider the functional~$\mathrm{J}_r$, which is the local contribution of the~$H^{s/2}$-Gagliardo seminorm in the ball~$B_r$, and is defined in~\cite{caffarelli_roquejoffre_savin_nonlocal} for every function~$u\in H^{s/2}(\R^n)$ as
	\begin{equation}\label{defjeir}\begin{split}
			\mathrm{J}_r(u) 
			&:= \int_{\R^n}\int_{\R^n} \frac{|u(x)-u(y)|^2}{|x-y|^{n+s}} \chi_{B_r}(x)\big(\chi_{B_r}(y)+2\chi_{B_r^c}(y)\big) \,dx\,dy\\
			&= \int_{B_r}\int_{B_r} \frac{|u(x)-u(y)|^2}{|x-y|^{n+s}}\,dx\,dy + 2  \int_{B_r^c}\int_{B_r} \frac{|u(x)-u(y)|^2}{|x-y|^{n+s}}\,dx\,dy\\
			&= [u]_{H^{s/2}(\R^n)}^2-
			[u]_{H^{s/2}(B_r^c)}^2.
	\end{split}\end{equation}
	A simple computation shows that if~$u=\chi_E-\chi_{E^c}$, then~$\mathrm{J}_r(u)=8\Per_s(E,B_r)$.
	
	Also, in~\cite{caffarelli_roquejoffre_savin_nonlocal}, the authors show that there exists a relationship between the local energy and the extension functional, as stated in the forthcoming lemma. 
	
	In what follows, $\Omega$ is a bounded Lipschitz domain in~$\R^{n+1}$, and we recall that
	$$ \Omega_0:=\Omega\cap\{z=0\}\qquad{\mbox{and}}\qquad\Omega_+:=\Omega\cap\{z>0\}.$$
	
	\begin{lemma}[Lemma 7.2, \cite{caffarelli_roquejoffre_savin_nonlocal}]\label{lemma::local_energy_diff}
		Let~$u$, $v:\R^n\to\R$ be functions such that~$\mathrm{J}_1(u)$ and~$\mathrm{J}_1(v)$ are finite, and~$u-v$ is compactly supported in~$B_1$.
		
		Then,
		\begin{equation} \label{lemma::rel_J_energy}
			\widetilde{c}_{n,s} \left(\mathrm{J}_1(v)-\mathrm{J}_1(u)\right) =\inf_{\Omega,\overline{v}} \int_{\Omega_+} z^{1-s}(|\nabla \overline{v}|^2-|\nabla \widetilde{u}|^2) \,dx\,dz,
		\end{equation}
		where the infimum is taken among all bounded Lipschitz domains~$\Omega$ with~$\Omega_0\subseteq B_1$ and all functions~$\overline{v}$ such that~$\supp(\overline{v}-\widetilde{u})\comp\Omega$ and~$\trace(\overline{v})=v$ on~$\{z=0\}$.
	\end{lemma}
	
	As a consequence of Lemma~\ref{lemma::local_energy_diff}, we obtain Proposition~\ref{prop::char_almost_min_ext} arguing as follows.
	
	\begin{proof}[Proof of Proposition~\ref{prop::char_almost_min_ext}]
		Let~$E\subseteq\R^n$ and~$u:=\chi_E-\chi_{E^c}$ with extension~$\widetilde{u}$. 
		
		Suppose that~\eqref{eq::char_almos_min_ext} is satisfied, i.e. 
		\begin{equation*} 
			\int_{\Omega_+} z^{1-s}|\nabla\widetilde{u}|^2\,dx\,dz \leq \int_{\Omega_+} z^{1-s}|\nabla\overline{v}|^2\,dx\,dz + \hat{\Lambda}\big|E\Delta \{\overline{v}(x,0)=1\}\big|,
		\end{equation*}
		for every bounded Lipschitz open set~$\Omega\subseteq\R^{n+1}$ such that~$\Omega_0\subseteq B_1$, and every function~$\overline{v}$ such that $\supp(\overline{v}-\widetilde{u})\comp\Omega$ and~$|\overline{v}|=1$ on~$\Omega_0$. Recall that~$\hat{\Lambda}$ here is as in Remark~\ref{rem::hat_Lambda}.
		
		Let~$F\subseteq\R^n$ be such that~$F\setminus B_1=E\setminus B_1$ and define the function~$v:=\chi_F-\chi_{F^c}$. Then, thanks to Lemma~\ref{lemma::local_energy_diff}, we have that
		\begin{align*}
			8\widetilde{c}_{n,s} \left(\Per_s(F,B_1)-\Per_s(E,B_1)\right)
			&= \widetilde{c}_{n,s} \left(\mathrm{J}_1(v)-\mathrm{J}_1(u)\right) \\
			&=\inf_{\Omega,\overline{v}} 
			\int_{\Omega_+} z^{1-s}(|\nabla \overline{v}|^2-|\nabla \widetilde{u}|^2)\,dx\,dz\\
			&\geq -\hat{\Lambda}|E\Delta F|,
		\end{align*}
		where the infimum is take among all bounded Lipschitz domains~$\Omega$ such that~$\Omega_0\comp B_1$ and all functions~$\overline{v}$ such that~$\supp(\overline{v}-\widetilde{u})\comp\Omega$ and~$\trace(\overline{v})=v$ on~$\{z=0\}$. This gives that~$E$
		is a~$\Lambda$-minimal set in~$B_1$.
		
		Conversely, assume that~$E$ is a~$\Lambda$-minimal set in~$B_1$, for some~$\Lambda\geq0$. Consider a bounded Lipschitz open set~$\Omega\subseteq\R^{n+1}$ such that~$\Omega_0\subseteq B_1$ and let~$\overline{v}$ be a function such that $\supp(\overline{v}-\widetilde{u})\comp\Omega$ and $v:=\trace(\overline{v})\in\{-1,1\}$ in~$\R^n$. Then, using the almost minimality of~$E$ with the set~$F:=\{v(x)=1\}$ and Lemma~\ref{lemma::local_energy_diff}, we obtain that
		\begin{align*}
			0&\leq \Per_s(F,B_1)-\Per_s(E,B_1)+\Lambda|E\Delta F|\\
			&= \frac{1}{8}\left(\mathrm{J}_1(v)-\mathrm{J}_1(u)\right)+\Lambda|E\Delta \{v(x)=1\}|\\
			&\leq \frac{1}{8\widetilde{c}_{n,s}}\int_{\Omega_+} z^{1-s}(|\nabla \overline{v}|^2-|\nabla \widetilde{u}|^2)\,dx\,dz +\Lambda|E\Delta \{v(x)=1\}|,
		\end{align*} which proves~\eqref{eq::char_almos_min_ext}, as desired.
	\end{proof}

	\subsection{Monotonicity formula and proof of Theorem~\ref{th::Phi_monotone}}
	
	Now, we prove that the function~$\Phi_E$, as defined in~\eqref{defphimono}, is increasing in~$r$, as stated in Theorem~\ref{th::Phi_monotone}. 
	
	To achieve this, let $R>0$
	be such that~$B_R \subseteq \Omega$. Our strategy is the following: we first show that $\Phi_E$ is absolutely continuous in every compact subset
	of $(0,R)$ (see Proposition~\ref{prop::Phi_abs_cont} below). Then, we infer that $\Phi_E$ is differentiable almost everywhere in $(0,R)$ with $\Phi_E'(r)\geq0$, for a.e. $r\in(0,R)$. 
	
	\begin{prop} \label{prop::Phi_abs_cont}
		Let $0<r_1<r_2<R$. Then, the function $\Phi_E$ is absolutely continuous in $[r_1,r_2]$.
	\end{prop}
	\begin{proof}
		We write
		\begin{equation*}
			\Phi_E(r) = r^{s-n}\Psi_E(r) + \frac{n-s}{s}\omega_n\hat{\Lambda}r^s,
		\end{equation*}
		where
		\begin{equation*}
			\Psi_E(r) := \int_{B_r^+} z^{1-s}|\nabla \widetilde{u}|^2 dxdz.
		\end{equation*}
		Since $r^{s-n}$ and $\frac{n-s}{s}\omega_n\hat{\Lambda}r^s$ are smooth functions in $[r_1,r_2]$, it is sufficient to prove that $\Psi_E$ is absolutely continuous in $[r_1,r_2]$.
		
		To this purpose, let us pick $\epsilon>0$, and let $\{[a_j,b_j]\}_j$ be
		a collection of (at most countably many) non-overlapping subintervals of $[r_1,r_2]$ such that
		\begin{equation*}
			\sum_j |b_j-a_j| < \delta,
		\end{equation*}
		for some $\delta>0$ to be chosen. 
		
		Notice that
		\begin{equation*} 
			|B_{b_j}^+\setminus B_{a_j}^+| = \frac{\omega_{n+1}}{2}(b_j^{n+1}-a_j^{n+1})\leq c|b_j-a_j|,
		\end{equation*}
		for some positive constant $c$ depending only on $n$ and $R$. This entails that
		\begin{equation}\label{eq::delta_control_sets}
			\sum_j |B_{b_j}^+\setminus B_{a_j}^+| < c\delta
		\end{equation}
		
		On the other hand, by the absolute continuity of the Lebesgue integral, there exists $\widetilde{\delta}>0$ such that if
		\begin{equation*}
			\sum_j |B_{b_j}^+\setminus B_{a_j}^+| < \widetilde{\delta},
		\end{equation*}
		then
		\begin{equation} \label{eq::lebesgue_abs_cont}
			\sum_j \int_{B_{b_j}^+\setminus B_{a_j}^+} z^{1-s}|\nabla\widetilde{u}|^2dxdz <\epsilon
			.\end{equation}
		
		Thus, choosing $\delta<\widetilde{\delta}/c$ in~\eqref{eq::delta_control_sets}, and using \eqref{eq::lebesgue_abs_cont} together with the monotonicity of $\Psi_E$, we infer that
		\begin{equation*}
			\sum_j|\Psi_E(b_j)-\Psi_E(a_j)| = \sum_j \Psi_E(b_j)-\Psi_E(a_j) = \sum_j \int_{B_{b_j}^+\setminus B_{a_j}^+} z^{1-s}|\nabla\widetilde{u}|^2dxdz <\epsilon,
		\end{equation*}
		namely $\Psi_E$ is absolutely continuous in $[r_1,r_2]$.
	\end{proof}
	
	\begin{lemma} \label{lemma::Phi_monotone_ae}
		The function $\Phi_E$ is differentiable for almost every $r\in(0,R)$, and
		\begin{equation*} \label{eq::Phi_monotone_ae}
			\Phi_E'(r)\geq0,\quad\mbox{for a.e. }r\in(0,R).
		\end{equation*}
	\end{lemma}
	
	\begin{proof}
		Thanks to Proposition~\ref{prop::Phi_abs_cont} and the Fundamental Theorem of Calculus for the Lebesgue integral (see e.g. \cite[Theorem 7.20]{rudin}), $\Phi_E$ is differentiable for almost any $r\in(0,R)$ with derivative
		\begin{align*}
			\Phi_E'(r) 
			&= r^{s-n} \int_{\partial B_r^+} z^{1-s}|\nabla\widetilde{u}(x,z)|^2\, d\haus{n}(x,z) \\
			&\qquad\qquad-(n-s)r^{s-n-1} \int_{B_r^+} z^{1-s}|\nabla\widetilde{u}(x,z)|^2 \,dx\,dz +(n-s)\,\omega_n\,\hat{\Lambda} r^{s-1} \\
			&= r^{s-n} \Bigg(\int_{\partial B_r^+} z^{1-s}|\nabla\widetilde{u}(x,z)|^2\, d\haus{n}(x,z) \\
			&\qquad\qquad-\frac{n-s}{r} \int_{B_r^+} z^{1-s}|\nabla\widetilde{u}(x,z)|^2 \,dx\,dz +(n-s)\,\omega_n\,\hat{\Lambda} r^{n-1}\Bigg).
		\end{align*}
		Therefore, in order to show that $\Phi_E'(r)\geq0$
		for almost any~$r\in(0,R)$, it is sufficient to prove that
		\begin{equation*}
			\int_{\partial B_r^+} z^{1-s}|\nabla\widetilde{u}(x,z)|^2\, d\haus{n}(x,z) -\frac{n-s}{r} \int_{B_r^+} z^{1-s}|\nabla\widetilde{u}(x,z)|^2 \,dx\,dz +(n-s)\,\omega_n\,\hat{\Lambda} r^{n-1}\geq0.
		\end{equation*}
		
		With this aim, we consider the function
		\begin{equation*}
			\widetilde{v}_r(x,z):=
			\begin{cases}
				\widetilde{u}\left(\frac{rx}{|(x,z)|},\frac{rz}{|(x,z)|}\right) ,&\quad \text{  if }0<|x|^2+z^2\leq r^2,\\
				\widetilde{u}(x,z) ,&\quad \text{ if }|x|^2+z^2> r^2.
			\end{cases}
		\end{equation*}

		The main idea of the argument that we present is now
		to construct a suitable competitor for the almost minimality condition. Given a radius~$r$, this competitor will be defined as a homogeneous rescaling inside the ball of radius~$r$, while preserving the data outside the ball. However, some technical difficulties arise from taking traces from the extended half-space~$\R^{n+1}_+$ to the original ambient space~$\R^n$. These difficulties stem from the fact that the trace is defined only up to sets of null $n$-dimensional Lebesgue
		measure, whereas the construction of the competitor requires specifying values on a sphere. Consequently, for the traced competitor, the values along a sphere in~$\R^n$ are, in principle, not properly defined, since such a sphere has null $n$-dimensional measure.
		
		To handle this issue, the argument relies on a delicate measure-theoretic selection: the radius~$r$ will be chosen from an appropriate set of full one-dimensional measure, ensuring that the corresponding sphere of radius~$r$ in~$\R^n$ admits suitable homogeneous constructions compatible with the notion of trace. The precise details of this measure-theoretic argument are presented below.
		
		We claim that 
		\begin{equation*}\begin{split}&
				{\mbox{the trace of~$\widetilde{v}_r$ on~$\{z=0\}$ is well-defined for almost every $r\in(0,R)$}}\\&{\mbox{and
						(up to sets of null $n$-dimensional Lebesgue measure)}}\\
				&{\mbox{this trace is of the form~$\chi_{F_r}-\chi_{F_r^c}$, for some set~$F_r$ such that~$F_r\setminus B_r = E\setminus B_r$. }}\end{split}\end{equation*}
		
		Indeed, by construction, $\trace(\widetilde{u})(x)=u(x)$ for almost every $x\in\R^n$, that is 
		this identity holds
		for all~$x\in{\mathcal{G}}$, for a suitable
		set~${\mathcal{G}}\subseteq\R^n$
		with $|\R^n\setminus{\mathcal{G}}|=0$ (where $|\cdot|$ denotes the $n$-dimensional Lebesgue measure). 
		
		We point out that, by the co-area formula,
		\begin{equation*}
			0=|\R^n\setminus{\mathcal{G}}| = \int_{0}^{+\infty} \haus{n-1}( \partial B_r\setminus{\mathcal{G}}) \,dr \geq 0.
		\end{equation*}
		As a result, there exists a set ${\mathcal{W}}\subseteq(0,+\infty)$ such that $\haus{1}\big((0,+\infty)\setminus
		{\mathcal{W}}\big)=0$ and 
		\begin{equation}\label{OLNSdoiewum6nbiu}{\mbox{$\haus{n-1}(\partial B_r\setminus{\mathcal{G}})=0$ for every $r\in{\mathcal{W}}$. }}\end{equation}
		
		Now, for all~$r\in{\mathcal{W}}$ we define
		$${\mathcal{G}}_r:=\left\{ x\in\R^n {\mbox{ s.t. }}\frac{rx}{|x|}\in{\mathcal{G}}\right\}$$
		and
		$$\widehat{\mathcal{G}}_r:=
		{\mathcal{G}}_r\cap\partial B_1=\left\{ x\in\partial B_1 {\mbox{ s.t. }} rx\in{\mathcal{G}}\right\}=
		\frac{ 1}{r}{\mathcal{G}}\cap\partial B_1.$$
		
		We remark that, for all~$r\in{\mathcal{W}}$,
		$$ \haus{n-1}(\widehat{\mathcal{G}}_r)=\haus{n-1}\left(\frac{ 1}{r}{\mathcal{G}}\cap\partial B_1\right)
		=\frac1{r^{n-1}}\haus{n-1}({\mathcal{G}}\cap\partial B_r)
		=\frac1{r^{n-1}}\haus{n-1}(\partial B_r)=\haus{n-1}(\partial B_1),$$
		owing to~\eqref{OLNSdoiewum6nbiu}.
		
		Notice also that~${\mathcal{G}}_r$ is a cone, namely~$x\in{\mathcal{G}}_r$ if and only if~$tx\in{\mathcal{G}}_r$ for all~$t>0$.
		As a result, for all~$r_2>r_1>0$,
		\begin{equation*}
			\begin{split}
				& |{\mathcal{G}}_r\cap (B_{r_2}\setminus B_{r_1})|=
				\int_{r_1}^{r_2} \haus{n-1}( {\mathcal{G}}_r\cap \partial B_t)\,dt=
				\int_{r_1}^{r_2} t^{n-1}\,\haus{n-1}\left( \frac{1}{t}{\mathcal{G}}_r\cap \partial B_1\right)\,dt\\
				&\qquad=
				\int_{r_1}^{r_2} t^{n-1}\,\haus{n-1}( {\mathcal{G}}_r \cap \partial B_1)\,dt=
				\int_{r_1}^{r_2} t^{n-1}\,\haus{n-1}( \widehat{\mathcal{G}}_r )\,dt=
				\int_{r_1}^{r_2} t^{n-1}\,\haus{n-1}( \partial B_1 )\,dt\\
				&\qquad=|B_{r_2}\setminus B_{r_1}|,
			\end{split}
		\end{equation*}
		showing that, for all~$r\in{\mathcal{W}}$, we have that $|\R^n\setminus {\mathcal{G}}_r|=0$ (i.e. the set~${\mathcal{G}}_r$ has full measure in~$\R^n$).
		
		Consequently, for all~$r\in{\mathcal{W}}$,
		the set~$\widetilde{\mathcal{G}}_r:=
		{\mathcal{G}}_r\cap{\mathcal{G}}$
		has full measure in~$\R^n$.
		
		Accordingly, if~$r\in{\mathcal{W}}$ and~$x\in \widetilde{{\mathcal{G}}}_r$, then both~$x$ and~$\frac{rx}{|x|}$
		lie in~${\mathcal{G}}$, and therefore
		\begin{equation*}
			\widetilde{v}_r(x,0)=
			\begin{cases}
				\widetilde{u}\left(\frac{rx}{|x|},0\right)=\chi_E\left(\frac{rx}{|x|}\right)-\chi_{E^c}\left(\frac{rx}{|x|}\right),\quad&\text{if }0<|x|\leq r,\\
				\widetilde{u}(x,0)=\chi_E(x)-\chi_{E^c}(x),\quad&\text{if }|x|> r,
			\end{cases}
		\end{equation*}
		is well-defined (and thus $\widetilde{v}_r(x,0)\in\{-1,1\}$). 
		
		In what follows, we fix $r\in(0,R)\cap{\mathcal{W}}$, and we write $\widetilde{v}:=\widetilde{v}_r$ and $F:=F_r$ for simplicity.
		
		Now, exploiting Proposition~\ref{prop::char_almost_min_ext}, we find that
		\begin{equation}\label{eq::monotone_1}
			\begin{split}
				\int_{B_r^+} z^{1-s}|\nabla\widetilde{v}(x,z)|^2\,dx\,dz
				&\geq \int_{B_r^+} z^{1-s}|\nabla\widetilde{u}(x,z)|^2\,dx\,dz - \hat{\Lambda}|E\Delta F|\\
				&\geq \int_{B_r^+} z^{1-s}|\nabla\widetilde{u}(x,z)|^2\,dx\,dz - \omega_n \hat{\Lambda}r^{n} .
			\end{split}
		\end{equation}
		
		Then, observe that~$\widetilde{v}$ is constant along the radial direction in $B_r\setminus\{0\}$. Hence, 
		\begin{equation*}
			\nabla\widetilde{v}(x,z) = \frac{r}{|(x,z)|} \nabla^T \widetilde{u}\left(\frac{r}{|(x,z)|}x,\frac{r}{|(x,z)|}z\right),\quad \mbox{in }B_r\setminus\{0\},
		\end{equation*}
		where~$\nabla^T$ is the tangential component of the gradient~$\nabla$ with respect to the unitary sphere. 
		
		Therefore, by Fubini-Tonelli's Theorem, we have that
		\begin{equation*}
			\begin{split}
				\int_{B_r^+} z^{1-s}|\nabla\widetilde{v}(x,z)|^2\,dx\,dz 
				&= \int_{0}^{r}\int_{\partial B_t^+} z^{1-s}|\nabla\widetilde{v}(x,z)|^2\,d\haus{n}(x,z)\, dt \\
				&= \int_{0}^{r}\int_{\partial B_t^+} \frac{r^2}{t^2}z^{1-s}\left|\nabla^T\widetilde{u}\left(\frac{r}{t}x,\frac{r}{t}z\right)\right|^2\,d\haus{n}(x,z)\, dt.
			\end{split}
		\end{equation*}
		Using the change of variable~$(\hat{x},\hat{z}):=\left(\frac{x}{t},\frac{z}{t}\right)$, we obtain that
		\begin{equation*}
			\begin{split}
				\int_{B_r^+} z^{1-s}|\nabla\widetilde{v}(x,z)|^2\,dx\,dz 
				& = \int_{0}^{r}\int_{\partial B_1^+} r^2t^{n-2}(t\hat{z})^{1-s}|\nabla^T\widetilde{u}(r\hat{x},r\hat{z})|^2\,d\haus{n}(\hat{x},\hat{z})\, dt\\
				& = r^2\int_{0}^{r} t^{n-1-s}\,dt \int_{\partial B_1^+} \hat{z}^{1-s}|\nabla^T\widetilde{u}(r\hat{x},r\hat{z})|^2\,d\haus{n}(\hat{x},\hat{z}) \\
				& = \frac{r^{n+2-s}}{n-s} \int_{\partial B_1^+} \hat{z}^{1-s}|\nabla^T\widetilde{u}(r\hat{x},r\hat{z})|^2\,d\haus{n}(\hat{x},\hat{z}).
			\end{split}
		\end{equation*}
		{F}rom this, the change of variable~$(x,z):=(r\hat{x},r\hat{z})$ leads us to
		\begin{equation}  \label{eq::monotone_2}
			\begin{split}
				\int_{B_r^+} z^{1-s}|\nabla\widetilde{v}(x,z)|^2\,dx\,dz 
				& = \frac{r}{n-s} \int_{\partial B_1^+} (r\hat{z})^{1-s}|\nabla^T\widetilde{u}(r\hat{x},r\hat{z})|^2\,r^{n+1}
				\,d\haus{n}(\hat{x},\hat{z})\\
				& = \frac{r}{n-s} \int_{\partial B_r^+} z^{1-s}|\nabla^T\widetilde{u}(x,z)|^2\,d\haus{n}(x,z).
			\end{split}
		\end{equation}
		
		Plugging~\eqref{eq::monotone_2} into~\eqref{eq::monotone_1}, we obtain that
		\begin{equation*}
			\begin{split}
				\frac{r}{n-s} \int_{\partial B_r^+} z^{1-s}|\nabla \widetilde{u}(x,z)|^2\,d\haus{n}(x,z)
				&\geq \frac{r}{n-s} \int_{\partial B_r^+} z^{1-s}|\nabla^T\widetilde{u}(x,z)|^2\,d\haus{n}(x,z)\\
				&\geq \int_{B_r^+} z^{1-s}|\nabla \widetilde{u}|^2\,dx\,dz - \omega_n \hat{\Lambda} r^{n}\  ,
			\end{split}
		\end{equation*}
		which gives that~$\Phi_E'(r)\geq0$, as desired. 
	\end{proof}
	
	\begin{proof}[Proof of Theorem~\ref{th::Phi_monotone}]
		Let $r_1,r_2\in(0,R)$ with $r_1<r_2$. Thanks to Proposition~\ref{prop::Phi_abs_cont} and Lemma~\ref{lemma::Phi_monotone_ae}, we infer that
		\begin{equation*}
			\Phi_E(r_2)-\Phi_E(r_1) = \int_{r_1}^{r_2} \Phi_E'(\rho)d\rho \geq 0,
		\end{equation*}
		showing the result.
	\end{proof}
	
	Another relevant property of~$\Phi_E$ is that~$\Phi_E$ is bounded for any~$r$ small enough. We point out that this is not a direct consequence of Theorem~\ref{th::Phi_monotone}. In fact, we want to avoid the pathological case in which~$\Phi_E\equiv+\infty$.
	
	In order to prove the boundedness of~$\Phi_E$, it is convenient to introduce the function
	\begin{equation}\label{eq::def_Xi}
		\Xi_E(r):= r^{s-n}\int_{B_r^+} z^{1-s}|\nabla\widetilde{u}|^2\,dx\,dz.
	\end{equation}
	Hence, 
	\begin{equation}\label{eq::Phi_with_Xi}
		\Phi_E(r)=\Xi_E(r)+\frac{n-s}{s}\omega_n\hat{\Lambda} r^s.
	\end{equation}
	
	\begin{theorem} \label{th::Phi_bound}
		If~$E$ is a~$\Lambda$-minimal set in~$B_R$, for some~$R>0$, then~$\Phi_E$ is bounded for every~$r\leq R/2$.
	\end{theorem}
	
	\begin{proof}
		In light of~\eqref{eq::Phi_with_Xi}, it is sufficient to check that~$\Xi_E$ is bounded.
		
		{F}rom~\cite[Proposition~7.1]{caffarelli_roquejoffre_savin_nonlocal},
		used with~$\Omega:=B_{1/2}^{n+1}$, we have that
		$$ \Xi_E(1/2) = 2^{n-s}\int_{\Omega^+} z^{1-s}|\nabla\widetilde{u}|^2\,dx\,dz \leq C2^{n-s}\mathrm{J}_1(u)=C\Per_s(E,B_1).$$
		Notice that~$\Xi_E$ is scale invariant, that is~$\Xi_E(\mu r)=\Xi_{\mu E}(r)$, for every~$\mu>0$. 
		
		Consider a set~$A:=B_{2r}\cup(E\cap B_{2r}^c)$. Since~$E$ is an almost minimal set in~$B_{2r}\subseteq B_R$, we have that
		\begin{equation}\label{56789rtyuifghjxcvb}
			\begin{split}
				&\Xi_E(r) = \Xi_{\frac{1}{2r}E}(1/2)
				\leq C\Per_s\left(\frac{1}{2r}E,B_1\right)
				= C(2r)^{s-n}\Per_s(E,B_{2r})\\
				&\qquad\leq C(2r)^{s-n}\left(\Per_s(A,B_{2r}) + \Lambda(2r)^n\right)
				\leq C(2r)^{s-n}\left(\Per_s(B_{2r},\R^n) + \Lambda(2r)^n\right)\\
				&\qquad\leq C\left(\Per_s(B_1,\R^n)+\Lambda(2r)^{s}\right).
			\end{split}
		\end{equation}
		Therefore, for all~$r\leq R/2$,
		\begin{equation*}
			\Phi_E(r) \leq C\left(\Per_s(B_1,\R^n)+(2r)^{s}\right) + \frac{n-s}{s}\omega_n\hat{\Lambda} r^s \leq C(1+R^s),
		\end{equation*}
		as desired.
	\end{proof}
	
	\subsection{Tangent cones as blow-up limits}
	With the monotonicity formula from Theorem~\ref{th::Phi_monotone} in hand, we are able to study regular and singular points of~$E$ through a blow-up analysis. In particular, if~$E$ is an almost minimal set in~$B_1$ such that~$0\in\partial E$, we consider the dilation~$\lambda E$ and study the limit set when~$\lambda\to+\infty$. We start our discussion with a convergence result.
	
	\begin{prop} \label{prop::phi_continuous_in_E}
		Let~$\{E_k\}_k$ be a sequence of~$\Lambda_k$-minimal sets in~$B_k$. Assume that~$E_k\to E$ in~$L^1_{\loc}$, for some set~$E$
		of finite fractional perimeter, and~$\Lambda_k\to\Lambda\in[0,+\infty)$, as~$k\to+\infty$.
		
		Let~$u_k:=\chi_{E_k}-\chi_{E_k^c}$ and~$u:=\chi_E-\chi_{E^c}$, with extension~$\widetilde{u}_k$ and~$\widetilde{u}$, respectively.
		
		Then,
		\renewcommand{\theenumi}{\roman{enumi}}
		\begin{enumerate}
			\item\label{item::phi_continuous_in_E_1} $\widetilde{u}_k\to\widetilde{u}$ uniformly on every compact set~$K\comp\R^{n+1}_+$;
			\item\label{item::phi_continuous_in_E_2} $\nabla\widetilde{u}_k\to\nabla\widetilde{u}$ in~$L^2_{\loc}(\R^{n+1}_+,z^{1-s}\,dx\,dz)$.
		\end{enumerate}
		In particular, $\Phi_{E_k}\to\Phi_E$.
	\end{prop}
	
	The proof of Proposition~\ref{prop::phi_continuous_in_E} is just a variation of the proof of~\cite[Proposition~9.1]{caffarelli_roquejoffre_savin_nonlocal}. For the sake of completeness, it can be found in the Appendix~\ref{sec::prop_phi_continuous_in_E}.
	
	We now give the following:
	
	\begin{definition}
		A set~$C$ is a minimal cone if it is a~$0$-homogeneous local minimizer for the~$s$-perimeter, i.e.~$C$ is a minimizer for~$\Per_s(\cdot,\Omega)$ such that~$\lambda C=C$, for every~$\lambda>0$.
	\end{definition}
	
	We point out that~$E$ is a minimal cone with vertex at the origin, that is the~$(n+1)$-dimensional extension of~$\chi_E-\chi_{E^c}$ is a~$0$-homogeneous function, if and only if the function~$\Xi_E$, as defined in~\eqref{eq::def_Xi}, is constant (see~\cite[Corollary~8.2]{caffarelli_roquejoffre_savin_nonlocal}).
	
	Now,
	replacing~$B_k$ with~$B_{\lambda_k}$, for some divergent sequence~$\{\lambda_k\}_k$, in Proposition~\ref{prop::phi_continuous_in_E}, we infer that the limit set of a convergent blow-up sequence~$\{E_k:=\lambda_k E\}_k$ is a minimal cone, according to the following statement:
	
	\begin{theorem}[Blow up limit] \label{th::blowup_cone}
		Let~$E$ be a~$\Lambda$-minimal set in~$B_1$ such that~$0\in\partial E$, and let~$\{\lambda_k\}_k$ be a sequence of positive real numbers
		such that~$\lambda_k\to+\infty$ as~$k\to+\infty$. 
		
		If there exists a set~$C$ such that~$\lambda_k E\to C$ in~$L^1_{\loc}$, then~$C$ is a minimal cone.
	\end{theorem}
	
	
	We call the minimal cone~$C$ in Theorem~\ref{th::blowup_cone}
	a tangent cone to~$E$ at~$0$.
	This definition is well-posed in light of the forthcoming result, which guarantees the existence of a blow-up limit.
	
	\begin{theorem}[Existence of blow-up limits] \label{th::blowup_existence}
		Let~$E$ be a~$\Lambda$-minimal set in~$B_1$ such that~$0\in\partial E$. Then, the blow-up sequence~$\{\lambda_k E\}_k$ admits limit (up to a subsequence). 
		
		Namely, there exists a tangent cone~$C$ to~$E$ at the origin.
	\end{theorem}
	
	We now present the proofs of Theorems~\ref{th::blowup_cone} and~\ref{th::blowup_existence}.
	
	\begin{proof}[Proof of Theorem~\ref{th::blowup_cone}]
		We notice that~$\lambda_k E$ is a~$\lambda_k^{-s}\Lambda$-minimal set
		in~$B_{\lambda_k}$. Hence, we are in the position of exploiting
		Proposition~\ref{prop::phi_continuous_in_E} and, if~$\lambda_k E\to C$, obtain that
		$$\lim_{k\to+\infty} \Phi_{\lambda_k E}(r)=\Phi_C(r).$$
		As a consequence, recalling the notation in~\eqref{eq::def_Xi},
		\begin{align*}
			\Phi_C(r) =\lim_{k\to+\infty}\Phi_{\lambda_k E}(r) 
			= \lim_{k\to+\infty}\Xi_{\lambda_k E}(r) + \frac{n-s}{s}\omega_n\hat{\Lambda} r^{s}=\lim_{k\to+\infty} \Xi_E\left(\frac{r}{\lambda_k }\right) +\frac{n-s}{s}\omega_n\hat{\Lambda} r^{s}.
		\end{align*}
		
		Now, since~$\Phi_E$ is positive and monotone increasing, we have that~$\Phi_E(r/\lambda_k)$ admits limit as~$k\to+\infty$, and therefore
		$$\lim_{k\to+\infty} \Xi_E\left(\frac{r}{\lambda_k}\right) =\lim_{k\to+\infty} \Phi_E\left(\frac{r}{\lambda_k}\right) - \frac{n-s}{s}\omega_n\hat{\Lambda} \left(\frac{r}{\lambda_k}\right)^{s} = \Phi_E(0) = \Xi_E(0) .$$
		As a result, we find that
		$$	\Phi_C(r)=	\Xi_E(0)+\frac{n-s}{s}\omega_n\hat{\Lambda} r^{s},$$
		and therefore
		$$\Xi_C(r) 
		= \Phi_C(r) - \frac{n-s}{s}\omega_n\hat{\Lambda} r^{s} =\Xi_E(0).$$
		This shows that~$\Xi_C$ is constant, hence~$C$ is a minimal cone with vertex at~$0$.
	\end{proof}
	
	\begin{proof}[Proof of Theorem~\ref{th::blowup_existence}]
		Notice that if~$E$ is a~$\Lambda$-minimal set in~$B_1$, then~$\lambda_k E$ is a~$\lambda_k^{-s}\Lambda$-minimal set in~$B_{\lambda_k}$. Therefore, for~$k$ sufficiently large, we have that~$\lambda_k E$ is almost minimal in~$B_1$.

		Let us consider the set~$A_{k}^1:=B_1\cup(\lambda_k E\cap B_1^c)$. In this setting, we have that~$|A_{k}^1\Delta\lambda_k E|\leq|B_1|$, and, by the almost-minimality of~$\lambda_k E$,
		\begin{align*}
			&\norm{\chi_{\lambda_k E}}_{W^{1,s}(B_1)} = \norm{\chi_{\lambda_k E}}_{L^1(B_1)}+[\chi_{\lambda_k E}]_{W^{1,s}(B_1)}\\
			&\qquad\leq |B_1|+2\Per_s(\lambda_k E,B_1)\leq |B_1|+2\Per_s(A_{k}^1,B_1) + 2\lambda_k^{-s}\Lambda|B_1|\\
			&\qquad\leq (1+2\Lambda)|B_1|+2\Per_s(B_1,\R^n)  ,
		\end{align*} 
		provided that~$k$ is sufficiently large.
		
		Using the Sobolev embedding~$W^{1,s}(B_1)\hookrightarrow\hookrightarrow L^1(B_1)$, we extract a subsequence~$\lambda_k^1$ such that
		$$ (\lambda_k^1 E)\cap B_1 \to E_1 \text{ in } L^1(B_1)$$
		for some set~$E_1\subseteq B_1$. Repeating the argument with~$\{\lambda_k^1 E\}_k$ in~$B_2$, we find a subsequence~$\{\lambda_k^2\}_k\subseteq\{\lambda_k^1\}_k$ and a set~$E_2\subseteq B_2$ such  that 
		$$ (\lambda_k^2 E)\cap B_2 \to E_2 \text{ in } L^1(B_2).$$
		Moreover, by the uniqueness of the limit, we have that~$E_2\cap B_1=E_1$.
		
		Iterating this procedure~$m$ times, we obtain a subsequence~$\{\lambda_k^m\}_k$ such that~$ (\lambda_k^m E)\cap B_m \to E_m$ in~$L^1(B_m)$ as~$k\to+\infty$, and~$E_m\cap B_j=E_j$, for all~$j=1,\dots,m-1$.
		
		Using a diagonal argument, we extract a subsequence~$\{\lambda_k^k\}_k$ such that
		$$ \lambda_k^k E \to \bigcup_{m=1}^{+\infty} E_m=:C \text{ in }L^1_{\loc},$$
		and this concludes the proof.
	\end{proof}
	
	\subsection{Classification of tangent cones}
	Now, we show how the classification of the tangent cones to an almost minimal set allows to deduce regularity.
	
	\begin{theorem}\label{th::regular_tangent_cone}
		Let~$E$ be a~$\Lambda$-minimal set in~$B_1$ such that~$0\in\partial E$. Suppose that the tangent cone to~$E$ at~$0$ is a half-space. 
		
		Then, there exists~$\delta>0$ such that~$\partial E\cap B_\delta$ is a~$\cont^{1,\alpha}$-regular surface.
	\end{theorem}
	
	The proof of Theorem~\ref{th::regular_tangent_cone} will be presented here below. Now, we give the definitions of regular and singular points.
	
	\begin{definition}[Definition 9.5, \cite{caffarelli_roquejoffre_savin_nonlocal}]
		Let~$E$ be a~$\Lambda$-minimal set in~$\Omega$ and let~$x\in\partial E\cap\Omega$. We say that~$x$ is a regular point for~$E$ if the tangent cone to~$E$ at~$x$ is a half-space. 
		
		We say that~$x$ is singular if it is not regular. We denote by~$\Sigma_E$ the set of singular points for~$E$. In particular, $\Sigma_E$ is a closed set in~$\R^n$.
	\end{definition}
	
	In other words, half-spaces are the tangent cones with least local energy. Moreover, there exists an energy gap between regular and singular tangent cones. To state this result, we recall the definition of~$\Xi_E$ in~\eqref{eq::def_Xi}.
	
	\begin{theorem}[Energy gap]\label{th::energy_gap}
		Let~$\Pi:=\{x_1>0\}$. 
		Then, for every tangent cone~$C$, 
		\begin{equation}\label{eq::PileqC}
			\Xi_\Pi\leq\Xi_C,
		\end{equation}
		and equality holds if and only if~$C$ is a half-space. 
		
		In particular, if~$C$ is not a half-space, then
		there exists~$\delta_0>0$, depending only on~$n$ and~$s$, such that
		\begin{equation}\label{eq::gap}
			\Xi_\Pi\leq\Xi_C-\delta_0 .
		\end{equation}
	\end{theorem}
	
	Theorem~\ref{th::energy_gap} can be proved as~\cite[Theorem~9.6]{caffarelli_roquejoffre_savin_nonlocal}, up to using the clean ball condition for almost minimizers given in Corollary~\ref{cor::clean_ball}). For the sake of completeness, a proof of the energy gap result for almost minimal boundaries can be found in Appendix~\ref{sec::energy_gap}.
	
	\begin{proof}[Proof of Theorem~\ref{th::regular_tangent_cone}]
		By assumptions, if we consider the blow-up sequence~$\{\lambda_k E\}_k$, for some~$\lambda_k\to+\infty$, then we have that~$\lambda_k E\to \Pi$ in~$L^1_{\loc}$, for some half-space~$\Pi$. Up to rotations, we can always assume that~$\Pi=\{x_n<0\}$. 
		
		As a consequence of Corollary~\ref{cor::boundary_conv}, we have that
		$$ \partial(\lambda_k E)\cap B_1 \subseteq \mathscr{U}_\epsilon(\partial\Pi) = \{|x_n|<\epsilon\}, \quad\text{ for every~$k$ large enough.}$$
		In particular, there exists some~$\overline{k}\in\N$ such that 
		$$ \partial(\lambda_{\overline{k}} E)\cap B_1 \subseteq \{|x_n|<\epsilon\} .$$
		Thus, by Theorem~\ref{th::holder_reg_almost_min}, we have that~$\partial(\lambda_{\overline{k}} E)$ is a~$\cont^{1,\alpha}$-regular surface in~$B_{1/2}$. Therefore, we infer that~$\partial E$ is a~$\cont^{1,\alpha}$-regular surface in~$B_{1/(2\lambda_{\overline{k}})}$.
	\end{proof}
	
	\section{Hausdorff dimension of the singular set and proof of Theorem~\ref{th::haus_dim_singular_almost}} \label{sec::haus_dim}
	In this section, we focus on the proof of Theorem~\ref{th::haus_dim_singular_almost}. To do this, we follow Federer's reduction method, as developed in~\cite{caffarelli_roquejoffre_savin_nonlocal} for the fractional setting. In particular, we exploit the fact that if~$C$ is a minimal cone in~$\R^n$ with a singular point~$x_0\neq0$, then there exists a minimal cone~$C'$ in~$\R^{n-1}$ such that the origin is a singular point for~$C'$ and~$C'\times\R$ is a tangent cone to~$C$ at~$x_0$. Applying this argument iteratively, we reduce to dimension~$2$, where any~$s$-minimal cone is a half-space.
	
	\subsection{Hausdorff dimension of the singular set}
	In~\cite{savin_valdinoci_regularity}, the authors prove that there are no singular~$s$-minimal cones in dimension~$2$. The precise statement goes as follows:
	
	\begin{theorem}[Theorem 1, \cite{savin_valdinoci_regularity}]\label{th::cones_R2}
		If~$C$ is an~$s$-minimal cone in~$\R^2$, then~$C$ is a half-space.
	\end{theorem}
	
	
	In light of Theorem~\ref{th::cones_R2} above and~\cite[Theorem~10.3]{caffarelli_roquejoffre_savin_nonlocal} (fractional dimension reduction method), we deduce that if~$C$ is a singular cone in~$\R^3$, its only singularity must be at the vertex.
	
	Since we are interested in the Hausdorff dimension of the singular set of an almost minimizer, let us now recall some definitions.
	
	\begin{definition}
		Given a set~$E\subseteq\R^n$, $d>0$, and~$\delta\in(0,+\infty]$, we set 
		$$ \haus{d}_\delta(E):=2^{-d}\omega_d\inf\left\{\sum_{j=1}^{+\infty}\diam^d(A_j):  E\subseteq\bigcup_{j=1}^{+\infty}A_j,\ \diam(A_j)<\delta\right\}.$$
		We recall that the Hausdorff measure of~$E$ is defined as
		$$ \haus{d}(E):= \lim_{\delta\to0} \haus{d}_\delta(E) = \sup_{\delta\in(0,+\infty]} \haus{d}_\delta(E),$$
		and the Hausdorff dimension of~$E$ as
		$$ \dim_\haus{}(E) := \inf\{d>0: \haus{d}(E)>0\}.$$
	\end{definition} 
	For a detailed discussion about the properties of the Hausdorff measure we refer to e.g.~\cite[Chapter~11]{giusti_minimal_surfaces}, and~\cite[Chapter~2]{MR3409135}.
	
	Now, we provide two auxiliary results that are useful for the proof of Theorem~\ref{th::haus_dim_singular_almost}.
	The first one claims that if~$\{E_k\}_k$ is a sequence of almost minimal sets~$L^1_{\loc}$-convergent to an almost minimal set~$E$, then the sequence of singular sets~$\{\Sigma_{E_k}\}_k$ is uniformly convergent to~$\Sigma_E$. 
	
	\begin{prop}\label{prop::haus_dim_singular_preliminary}
		Let~$\{E_k\}_k$ be a sequence of~$\Lambda_k$-minimal sets convergent to a~$\Lambda$-minimal set~$E$ in~$L^1_{\loc}$.
		
		Then, for every compact set~$K\subseteq\Omega$ and~$\epsilon>0$, we have that
		\begin{equation} \label{eq::sing_set_unif_conv}
			\Sigma_{E_k} \cap K \subseteq \mathscr{U}_\epsilon(\Sigma_E\cap K) ,\quad\text{ for every~$k$ large enough.}
		\end{equation}
	\end{prop}
	
	\begin{rem}
		Notice that, by definition of convergence with respect to the Hausdorff measure, the uniform convergence condition is precisely
		\begin{equation*}
			\Sigma_{E_k} \cap K \subseteq \mathscr{U}_\epsilon(\Sigma_E\cap K) ,\quad\text{ for every~$k$ large enough.}
		\end{equation*}
		{F}rom this, we arrive at
		\begin{equation} \label{eq::haus_conv_sing_set}
			\limsup_{k\to+\infty} \haus{d}_\infty(\Sigma_{E_k}\cap K) \leq \haus{d}_\infty(\Sigma_E\cap K).
		\end{equation}
	\end{rem}
	
	\begin{proof}[Proof of Proposition~\ref{prop::haus_dim_singular_preliminary}]
		Consider a compact set~$K\subseteq\Omega$, and~$\epsilon>0$. Arguing by contradiction, suppose that for every~$k$ there exists~$x_k\in\Sigma_{E_k}\cap K$ such that~$\dist(x_k,\Sigma_E\cap K)\geq\epsilon$. 
		By compactness, there exists~$x_\infty\in K$ such that~$x_k\to x_\infty$, up to a subsequence. 
		
		Also, for every~$k\in\N$, we have that~$x_k\in\partial E_k$. Consequently, by Corollary~\ref{cor::boundary_conv}, we infer that~$x_\infty\in\partial E$. 
		
		Now, suppose that~$x_\infty$ is a regular point. Hence, the tangent cone to~$E$ at~$x_\infty$ is a half-space. Therefore, up to rigid transformations, the uniform convergence result in Corollary~\ref{cor::boundary_conv} gives that
		\begin{equation*}
			\partial E \cap B_\delta(x_\infty) \subseteq \{|x_n-x_{\infty,n}|<\epsilon_0/2\} ,
		\end{equation*}
		for the same~$\epsilon_0$ as in Theorem~\ref{th::holder_reg_almost_min}.
		
		By Corollary~\ref{cor::boundary_conv}, we also have that
		\begin{equation*}
			\partial E_k \cap B_{\delta/2}(x_k) \subseteq \{|x_n-x_{k,n}|<\epsilon_0\} ,\quad\text{ for every~$k$ large enough.}
		\end{equation*}
		It follows from Theorem~\ref{th::holder_reg_almost_min} that~$x_k$ is a regular point for~$\partial E_k$, which is a contradiction. Thus, it must be~$x_\infty\in\Sigma_E$. Hence, by definition of convergence, 
		\begin{equation*}
			x_k \in B_{\epsilon/2}(x_\infty) \subseteq \mathscr{U}_\epsilon(\Sigma_E\cap K) ,\quad\text{ for every~$k$ large enough.}
		\end{equation*}
		This is against our initial assumptions, proving~\eqref{eq::sing_set_unif_conv}.
	\end{proof}
	
	The second result is a technical lemma that guarantees the existence of a tangent cone whose singular set has a suitable Hausdorff dimension.
	
	\begin{lemma}\label{lemma::singular_cones}
		Let~$E$ be an almost minimal set such that~$\haus{d}(\Sigma_E)>0$, for some~$d>0$.
		
		Then, there exists a tangent cone~$C$ to~$E$ with~$\haus{d}(\Sigma_C)>0$.
	\end{lemma}
	
	\begin{proof}
		Let us consider an almost minimal set~$E$ such that~$\haus{d}(\Sigma_E)>0$. Thanks to the properties of the Hausdorff measure (see~\cite[Proposition~11.3]{giusti_minimal_surfaces}), there exist at least one point~$x\in\Sigma_E$, and an infinitesimal sequence~$\{r_k\}_k$
		such that
		\begin{equation} \label{eq::haus_prop1}
			\haus{d}_\infty(\Sigma_{E}\cap B_{r_k}(x)) \geq \omega_dr_k^d2^{-d-1}.
		\end{equation}
		Up to a rigid transformation, we assume that~$x=0$. Then, if we consider the sequence of dilated sets~$\{E_k=r^{-k}E\}_k$, $E_k$ is almost minimal in~$r_k^{-1}\Omega$, and, up to extracting a subsequence, $E_k$ converges in~$L^1_{\loc}$ to a minimal cone~$C$ (see Theorem~\ref{th::blowup_existence}). Moreover, by construction, we have that~$\Sigma_{E_k}=r_k^{-1}\Sigma_E$. 
		
		We recall the scaling property of the Hausdorff measure (see e.g.~\cite{MR3409135}), that is, for every~$d>0$ and for every set~$A$,
		\begin{equation*}\label{eq::scaling_haus}
			\haus{d}_\infty(\lambda A) = \lambda^d \haus{d}_\infty(A).
		\end{equation*}
		Therefore, using this in~\eqref{eq::haus_prop1},
		$$ \haus{d}_\infty(\Sigma_{E_k}\cap B_{1}) = r_k^{-d} \haus{d}_\infty(\Sigma_{E}\cap B_{r_k}) \geq \omega_d2^{-d-1}.$$
		{F}rom the last inequality and~\eqref{eq::haus_conv_sing_set}, taking the limit as~$k\to+\infty$, we deduce that~$\haus{d}_\infty(\Sigma_C)>0$, which in turn entails that~$\haus{d}(\Sigma_C)>0$ (see~\cite[Lemma~11.2]{giusti_minimal_surfaces}).
	\end{proof}
	
	\subsection{Proof of Theorem~\ref{th::haus_dim_singular_almost}}
	
	Let us consider an almost minimal set~$E$, and suppose that~$\haus{d}(\Sigma_E)>0$. Thanks to Lemma~\ref{lemma::singular_cones}, there exists a tangent cone~$C$ to~$E$ such that~$\haus{d}(\Sigma_C)>0$. 
	
	Consider a point~$x_0\in \Sigma_C\setminus\{0\}$ (we stress that a blow-up at~$0$ would give~$C$ itself) and define the sequence~$\{C_k:=\lambda_k(C-x_0)\}_k$, for some~$\{\lambda_k\}_k$ such that~$\lambda_k\to+\infty$ as~$k\to+\infty$. Then, according to~\cite[Theorem~10.3]{caffarelli_roquejoffre_savin_nonlocal}, its blow-up limit is a minimal cone~$A_{n-1}\times\R$, where~$A_{n-1}\subseteq\R^{n-1}$ is a minimal cone in lower dimension. Since~$C$ is~$s$-minimal, it is in particular almost minimal, hence applying Lemma~\ref{lemma::singular_cones}, we obtain
	$$ \haus{d}(\Sigma_C)>0 \implies \haus{d}(\Sigma_{A_{n-1}\times\R})>0\ \implies \haus{d-1}(\Sigma_{A_{n-1}})>0.$$
	Using iteratively this procedure, we construct a sequence of minimal cones~$A_{n-j}\subseteq\R^{n-j}$ such that 
	$$\haus{d-j}(\Sigma_{A_{n-j}})>0 ,\quad\text{ for every } j\leq d.$$ 
	Thanks to Theorem~\ref{th::cones_R2}, we conclude that~$n-d>2$, that is~$d\leq n-3$.
	
	
	\section{Proof of Theorem~\ref{th::perturbed_almost_min}} \label{sec::perturbed_almost_min}
	
	This section is devoted to the proof of Theorem~\ref{th::perturbed_almost_min}, which states that almost minimality is preserved under suitable external perturbations. This fact is a direct consequence of the properties of the~$s$-perimeter.	The technical details go as follows.
	
	\begin{proof}[Proof of Theorem~\ref{th::perturbed_almost_min}]
		Let~$\Omega=B^{n-1}_1\times\R$. Since~$\Omega$ is unbounded, recall that almost minimality in~$\Omega$ is defined as almost minimality in every bounded Lipschitz subset of~$\Omega$. Therefore, for every~$M>0$, we consider the family of truncated cylinders~$\{\Omega_M\}_M$, where
		$$ \Omega_M := B^{n-1}_1\times(-M,M).$$
		In this setting, the fractional perimeter of a set~$E$ in~$\Omega_M$ is given by
		\begin{equation*}
			\Per_s(E,\Omega_M) = \mathscr{L}(E\cap\Omega_M,E^c)+\mathscr{L}(E\cap\Omega_M^c,E^c\cap\Omega_M).
		\end{equation*}
		
		Now, we compute the fractional perimeter of the set~$\widetilde{E}:=E\dot{\cup}\Sigma$ in~$\Omega_M$:
		\begin{align*}
			\Per_s(\widetilde{E},\Omega_M) 
			&= \mathscr{L}(\widetilde{E}\cap\Omega_M,\widetilde{E}^c)+\mathscr{L}(\widetilde{E}\cap\Omega_M^c,\widetilde{E}^c\cap\Omega_M)\\
			&= \mathscr{L}\left((E\dot{\cup}\Sigma)\cap\Omega_M,E^c\cap\Sigma^c\right)+\mathscr{L}\left((E\dot{\cup}\Sigma)\cap\Omega_M^c,E^c\cap\Sigma^c\cap\Omega_M\right).
		\end{align*}
		
		Therefore, since~$\Sigma\subseteq E^c\cap\Omega_M^c$, we obtain that
		\begin{equation}\label{eq::perturb_calc1}
			\begin{split}
				&	\Per_s(\widetilde{E},\Omega_M)	\\	
				&=\mathscr{L}(E\cap\Omega_M,E^c\cap\Sigma^c)+\mathscr{L}(E\cap\Omega_M^c,E^c\cap\Omega_M)+\mathscr{L}(\Sigma\cap\Omega_M^c,E^c\cap\Omega_M)\\
				&=\mathscr{L}(E\cap\Omega_M,E^c)-\mathscr{L}(E\cap\Omega_M,E^c\cap\Sigma)+\mathscr{L}(E\cap\Omega_M^c,E^c\cap\Omega_M)\\
				&\qquad+\mathscr{L}(\Sigma\cap\Omega_M^c,E^c\cap\Omega_M)\\
				&=\Per_s(E,\Omega_M)+\left[ \mathscr{L}(\Sigma,E^c\cap\Omega_M)-\mathscr{L}(E\cap\Omega_M,\Sigma)\right].
			\end{split}
		\end{equation}
		
		Let us consider a set~$F$ such that~$F\setminus\Omega_M=E\setminus\Omega_M$. Hence, the set~$\widetilde{F}:=F\dot{\cup}\Sigma$ is such that~$\widetilde{F}\setminus\Omega_M=\widetilde{E}\setminus\Omega_M$. 
		Moreover, from~\eqref{eq::perturb_calc1}, we have that
		\begin{equation}\label{eq::perturb_calc2}
			\begin{split}
				&\Per_s(\widetilde{E},\Omega_M)-\Per_s(\widetilde{F},\Omega_M)	\\	
				&= \Per_s(E,\Omega_M)-\Per_s(F,\Omega_M)-\left[\mathscr{L}(E\cap\Omega_M,\Sigma) -\mathscr{L}(F\cap\Omega_M,\Sigma)\right] \\
				&\qquad\quad+\left[ \mathscr{L}(\Sigma,E^c\cap\Omega_M)-\mathscr{L}(\Sigma,F^c\cap\Omega_M)\right] .
			\end{split}
		\end{equation}
		
		We observe that
		\begin{equation}\label{eq::perturb_calc3}
			\begin{split}
				&\left| \mathscr{L}(E\cap\Omega_M,\Sigma)-\mathscr{L}(F\cap\Omega_M,\Sigma)\right|
				= \left| \int_\Sigma\int_{\Omega_M} \frac{\chi_{E}(x)-\chi_{F}(x)}{|x-y|^{n+s}}\,dx\,dy\right|\\
				&\qquad\qquad\leq \int_\Sigma\int_{\Omega_M} \frac{|\chi_{E}(x)-\chi_{F}(x)|}{|x-y|^{n+s}}\,dx\,dy
				\le \int_\Sigma\int_{\Omega} \frac{\chi_{E\Delta F}(x)}{|x-y|^{n+s}}\,dx\,dy .
			\end{split}
		\end{equation}
		Since~$E\Delta F = E^c\Delta F^c$, the same calculations also show that 
		\begin{equation} \label{eq::perturb_calc4}
			\left| \mathscr{L}(\Sigma,E^c\cap\Omega_M)-\mathscr{L}(\Sigma,F^c\cap\Omega_M)\right| \leq \int_\Sigma\int_{\Omega} \frac{\chi_{E\Delta F}(x)}{|x-y|^{n+s}}\,dx\,dy .
		\end{equation}
		
		We now claim that
		\begin{equation}\label{785943vbcnxmpoiuytre0987654}
			\int_\Sigma\int_{\Omega} \frac{\chi_{E\Delta F}(x)}{|x-y|^{n+s}}\,dx\,dy
			\le  \widetilde c|E\Delta F|,\end{equation}
		for some~$\widetilde c>0$ depending only on~$n$, $s$, and~$\Sigma$.
		
		To this purpose, decompose~$\Sigma$ in the sets
		\begin{equation*}
			\Sigma_1 := \Sigma\cap\Big( \big(B^{n-1}_{1+\delta}\setminus\overline{B^{n-1}_1}\big) \times\R \Big)
			\qquad	\text{and}\qquad\Sigma_2 :=\Sigma\setminus\Sigma_1,
		\end{equation*}
		where~$\delta$ is as in the assumptions. We point out that~$\Sigma_1$ is possibly empty.
		
		Then, notice that, for every~$x\in\Omega$ and~$y\in\Sigma_1$,
		$$ |x-y| \geq |x'-y'| \geq |y'|-|x'| \geq |y'|-1 .$$
		Moreover, in light of the assumptions in~\eqref{eq::E_assumption} and~\eqref{eq::Sigma_assumption},
		we have that~$E=\subgr(\psi)$ in~$\Omega^c$, for some~$\psi\in\cont^{1,\alpha}$, and
		\begin{equation*}
			\tau_n(\Sigma)\subseteq \big\{(y',y_n)\;{\mbox{ s.t. }}\; \psi(x')<y_n<\psi(x)+C(|y'|-1)^{n+s-1+\xi}\big\} \quad \text{ in } B^{n-1}_{1+\delta}\setminus \overline{B^{n-1}_1}\times\R .
		\end{equation*} 
		As a result, we see that~$\Sigma_1\subseteq \tau_n^{-1}(G_{C,\xi})$, and therefore	
		\begin{equation} \label{eq::perturb_calc11}
			\begin{split} 
				\int_{\Sigma_1} \int_{\Omega}\frac{\chi_{E\Delta F}(x)}{|x-y|^{n+s}}\, dx\,dy
				&\leq |(E\Delta F)| \int_{\Sigma_1} \frac{dy}{(|y'|-1)^{n+s}} \\
				&\leq |E\Delta F| \int_{\tau_n^{-1}(G_{C,\xi})} \frac{dy}{(|y'|-1)^{n+s}} \\
				&= |E\Delta F| \int_{B^{n-1}_{1+\delta}\setminus {B}^{n-1}_1}\int_{0}^{C(|y'|-1)^{n+s-1+\xi}} \frac{dy_ndy'}{(|y'|-1)^{n+s}} \\
				&= |E\Delta F| \int_{B^{n-1}_{1+\delta}\setminus {B}^{n-1}_1} C(|y'|-1)^{\xi-1}\, dy' \\
				&= C_1 |E\Delta F|,
			\end{split}
		\end{equation}
		for some~$C_1>0$, depending on~$n$, $s$, and~$\Sigma$.
		
		Furthermore, notice that~$\dist(\Sigma_2,\Omega)\geq\delta>0$, and accordingly
		\begin{equation*}
			\begin{split}
				&\int_{\Sigma_2}\int_{\Omega} \frac{\chi_{E\Delta F}(x)}{|x-y|^{n+s}} \,dx\,dy 
				\leq \delta^{-n-s} \int_{\Sigma_2}\int_{\Omega_{M_\Sigma}}  \chi_{E\Delta F}(x) \,dx\,dy \\
				&\qquad\qquad\leq  \delta^{-n-s} |\Sigma|\, |(E\Delta F)|\leq C_2|E\Delta F|,
			\end{split}
		\end{equation*}
		for some~$C_2>0$, depending on~$n$, $s$, and~$\Sigma$. 
		
		This and~\eqref{eq::perturb_calc11} give the desired claim in~\eqref{785943vbcnxmpoiuytre0987654}.
		
		Using~\eqref{785943vbcnxmpoiuytre0987654} into~\eqref{eq::perturb_calc3} and~\eqref{eq::perturb_calc4}, we obtain that
		\begin{equation}\label{eq::perturbcalc12}
			\begin{split}&
				\left| \mathscr{L}(\Sigma,E\cap\Omega_M)-\mathscr{L}(\Sigma,F\cap\Omega_M)\right| \leq \widetilde{c} |E\Delta F|\\ {\mbox{and }}\quad
				&
				\left| \mathscr{L}(\Sigma,E^c\cap\Omega_M)-\mathscr{L}(\Sigma,F^c\cap\Omega_M)\right| \leq \widetilde{c} |E\Delta F| .
			\end{split}
		\end{equation}
		
		Now, assume that~$E$ is~$\Lambda$-minimal for~$\Per_s(\cdot,\Omega_M)$ with external datum~$E\setminus\Omega_M$, for some~$\Lambda\geq0$. Then, for every~$\widetilde{F}$ such that~$\widetilde{F}\setminus\Omega_M=\widetilde{E}\setminus\Omega_M$, if~$F:=(\widetilde{F}\cap\Omega_M)\cup(E\cap\Omega_M^c)$, we have that
		\begin{align*}
			\Per_s(\widetilde{E},\Omega_M)-\Per_s(\widetilde{F},\Omega_M) 
			&\leq \Per_s(E,\Omega_M)-\Per_s(F,\Omega_M) + 2\widetilde{c} |E\Delta F|\\
			&\leq \Lambda|E\Delta F|+2\widetilde{c} |\widetilde{E}\Delta \widetilde{F}|\\
			&= (\Lambda+2\widetilde{c})|\widetilde{E}\Delta \widetilde{F}|.
		\end{align*}
		Conversely, if~$\widetilde{E}$ is~$\widetilde{\Lambda}$-minimal for~$\Per_s(\cdot,\Omega_M)$ with external datum~$(E\dot{\cup}\Sigma)\setminus\Omega_M$, for some~$\widetilde{\Lambda}\geq0$, then for every set~$F$ such that~$F=E$ in~$\Omega_M^c$, we set~$\widetilde{F}:=F\dot{\cup}\Sigma$, and deduce that
		\begin{align*}
			\widetilde{\Lambda}|\widetilde{E}\Delta \widetilde{F}|
			&\geq \Per_s(\widetilde{E},\Omega_M)-\Per_s(\widetilde{F},\Omega_M) \\
			&\geq \Per_s(E,\Omega_M)-\Per_s(F,\Omega_M) - 2\widetilde{c} |E\Delta F| .
		\end{align*}
		Reordering the last inequality, and observing that~$|\widetilde{E}\Delta \widetilde{F}|=|E\Delta F|$, we obtain that
		\begin{equation*} 
			\Per_s(E,\Omega_M)\leq\Per_s(F,\Omega_M) + (\widetilde{\Lambda}+2\widetilde{c})|E\Delta F|.\qedhere
		\end{equation*}
	\end{proof}
	
	
	\section{Stickiness for almost minimal sets} \label{sec::stickiness_almost_min}
	Recall that stickiness for~$s$-minimal graphs in dimension~$2$ can be read as a discontinuity at the boundary as stated in Definition~\ref{def::stickiness}. 
	
	In this section, we prove that stickiness does not characterize almost minimal surfaces. In order to do this, we address the problem of the existence of sticking and non-sticking non-local almost minimal boundaries. In particular, we investigate whether or not it is true that every time we have a non-sticking~$s$-minimal surface~$E$ in~$\Omega$, we can find a sticking almost minimizer for~$\Per_s(\cdot,\Omega)$ with external datum~$E$, and, vice versa, whether or not is it true that every time we have a sticking~$s$-minimal surface, we can find a non-sticking almost minimizer. 
	
	\subsection{Existence of sticking almost minimizers}
	In~\cite{dipierro_savin_valdinoci_generality_stickiness}, the authors prove that~$s$-minimal graphs are generally sticky, in the sense that whatever the external datum is, one should expect to observe the stickiness phenomenon, up to an arbitrarily small perturbation of the external datum (see~\cite[Theorem~1.1]{dipierro_savin_valdinoci_generality_stickiness}). 

	We use this fact and Theorem~\ref{th::perturbed_almost_min} to show that it is always possible to find a sticking almost minimal set.
	\begin{prop}[Existence of sticking almost minimal sets] \label{prop::sticking_almost_min}
		Consider the domain~$\Omega=(-1,1)\times\R$ and a~$\cont^{1,\alpha}$-function~$v$ such that~$\subgr(v)$ has finite~$s$-perimeter in~$\Omega$.
		
		Then, there exists an almost minimal set~$E$ in~$\Omega$ such that~$E=\subgr({u})$, for some~${u}\in\cont^{1,\alpha}\left((-1,1)\right)$ such that~${u}=v$ in~$\Omega^c$ and~${u}$ is discontinuous at~$\{-1,1\}$.
	\end{prop}
	
	\begin{proof}
		Let~$F:=\subgr(v)\subseteq\R^2$, and consider an even bump-like function~$w\in\cont^\infty_c(\Omega^c)$ (e.g.~$w(t):=(\chi_{-I}+\chi_I)*e^{-t^2}$, for some interval~$I\subseteq(0,+\infty)$). Thus, the assumptions of Theorem~\ref{th::perturbed_almost_min} are satisfied with~$\Sigma:=\{v(x)+w(x) \text{ s.t. }x\in\supp(w)\}$.
		
		Thanks to~\cite[Theorem~1.1]{dipierro_savin_valdinoci_generality_stickiness}, there exists a minimizer~$\widetilde{E}$ for the~$s$-perimeter in~$\Omega$ with external datum~$(F\dot{\cup}\Sigma)\setminus\Omega$ that sticks to the boundary of~$\Omega$. 
		
		Therefore, Theorem~\ref{th::perturbed_almost_min} entails that 
		$$ E := (\widetilde{E}\cap\Omega)\cup (F\cap\Omega^c)$$
		is a sticking almost minimizer for~$\Per_s(\cdot,\Omega)$ with external datum~$F\setminus\Omega$.
	\end{proof}
	
	\subsection{Existence of non-sticking almost minimal sets}
	Now, we tackle the problem of the existence of non-sticking almost minimal sets. 
	What we show here is that it is always possible to find a non-sticking almost minimizer, up to requiring some assumptions on the external datum.
	\begin{prop}[Existence of non-sticking almost minimal sets] \label{prop::non_sticking_almost_min}
		Consider the domain~$\Omega=(-1,1)\times\R$ and a~$\cont^{1,\alpha}$-function~$v$ such that~$E:=\subgr(v)$ has finite~$s$-perimeter in~$\Omega$.
		Assume that there exists a non-sticking almost minimizer for~$\Per_s(\cdot,\Omega)$ with external datum~$\widetilde{E}:=\subgr(\widetilde{v})$ in~$\Omega^c$, for some~$\cont^{1,\alpha}$-function~$\widetilde{v}$.
		
		If~$\Sigma:=(E\Delta\widetilde{E})\cap\Omega^c$ satisfies the assumptions of Theorem~\ref{th::perturbed_almost_min},
		then there exists a non-sticking almost minimizer for~$\Per_s(\cdot,\Omega)$ with external datum~$E$.
	\end{prop} 
	\begin{proof}
		Thanks to Theorem~\ref{th::perturbed_almost_min}, if a set~$A$ is a non-sticking almost minimizer for~$\Per_s(\cdot,\Omega)$ with external datum~$\widetilde{E}$, then it is also an almost minimizer for~$\Per_s(\cdot,\Omega)$ with external datum~$E$. 
		
		Moreover, in this situation, no sticking phenomena occur.
	\end{proof}
	
	In particular, we have a non-sticking almost minimal set for every suitable perturbation of the half-space. 
	
	\begin{cor}
		Let~$\Omega$ be the cylinder~$(-1,1)\times\R\subseteq\R^2$. Consider the~$2$-dimensional half-space~$H:=\{(x_1,x_2): x_2<0\}$ and a set~$\Sigma\subseteq\Omega^c$ satisfying the assumptions of Theorem~\ref{th::perturbed_almost_min}.
		
		Then, $H\dot{\cup}\Sigma$ is a non-sticking almost minimizer for~$\Per_s(\cdot,\Omega)$ with external datum~$(H\dot{\cup}\Sigma)\setminus\Omega$. 
	\end{cor}
	
	\subsection{Almost minimal surfaces can show intermediate behaviors}
	In~\cite{dipierro_savin_valdinoci_generality_stickiness}, the authors also prove that~$2$-dimensional~$s$-minimal graphs cannot show intermediate behavior between discontinuity and differentiability. In other words, an~$s$-minimal graph in~$\R^2$ can only either stick to the boundary or be globally~$\cont^1$. 
	
	\begin{theorem}[Theorem 1.2, \cite{dipierro_savin_valdinoci_generality_stickiness}] \label{prop::cont_implies_diff}
		Let~$\alpha\in(s,1)$,
		and~$u:\R\to\R$ be a~$\cont^{1,\alpha}$-function in~$[-h,0]$, for some~$h\in(0,1)$, and
		$$ E := \big\{(x_1,x_2)\in\R^2 \text{ s.t. } x_2<u(x_1)\big\}.$$
		Suppose that~$E$ is~$s$-minimal in~$(0,1)\times\R$, and that
		$$ \lim_{x\to0^-} u(x) = \lim_{x\to0^+} u(x).$$
		
		Then, $u\in\cont^{1,\eta}([-h,1/2])$, with~$\eta:=\min\left\{\alpha,\frac{1+s}{2}\right\}$.
	\end{theorem}
	
	In contrast, almost minimal sets can show intermediate behaviors. To prove this, we exploit Theorem~\ref{th::perturbed_almost_min} to provide a counterexample to Theorem~\ref{prop::cont_implies_diff} for almost minimal sets. Namely, we exhibit a non-sticking continuous, but non-$\cont^1$ almost minimizers. 
	
	The idea is that by applying a suitable external perturbation to a non-sticking~$s$-minimal graph close to the boundary of the domain we force some discontinuity in the first derivative at the boundary
	without breaking the continuity of the function.
	
	\begin{ex}[Counterexample: a non-sticking~$\cont^0\setminus\cont^1$ almost minimal set] \label{ex::nonstiking_C0_nonC1}
		This counterexample is built starting from~\cite[Theorem~1.4]{dipierro_savin_valdinoci_boundary_behavior}. For this, we recall that for every~$\eta_0>0$ there exists~$\delta_0>0$ such that if we consider the sets~$H:=\R\times(-\infty,0)$, $F_-:=(-3,-2)\times(0,\delta)$ and~$F_+:=(2,3)\times(0,\delta)$, with~$\delta\in(0,\delta_0]$, and if~$E$ is an~$s$-minimal set in~$\Omega=(-1,1)\times\R$ with smooth external datum~$M\supset F_-\cup H\cup F_+$, then
		$$ (-1,1)\times(-\infty,\delta^\beta) \subseteq E,$$
		with~$\beta=\frac{2+\eta_0}{1-2s}$. 
		
		In particular, we have that~$\partial E$ is~$\cont^\infty$ in~$\Omega$, it sticks to the boundary, and it is symmetric with respect to the vertical axis whenever~$M$ is symmetric
		(see e.g.~\cite[Lemma~A.1]{dipierro_savin_valdinoci_boundary_behavior}). Moreover, thanks to~\cite[Corollary~1.3]{dipierro_savin_valdinoci_generality_stickiness}, we have that~$E$ detaches from~$\partial\Omega$ in a~$\cont^1$ fashion. 
		
		This yields that the boundary of~$E$ in~$\Omega$
		cannot be a horizontal straight line. Thus, there exists a point~$x_0\in(0,1)$ such that, if~$E=\subgr(u)$, then~$u'(\pm x_0)\neq0$. 
		
		We restrict our domain to~$\Omega_0:=(-x_0,x_0)\times\R$ and we see that~$E$ is an~$s$-minimal set in~$\Omega_0$. We also smooth out the external datum, in such a way that we preserve almost minimality, according to Theorem~\ref{th::perturbed_almost_min}. For a qualitative picture of the situation described here, see
		Figures~\ref{fig::counterex1}--\ref{fig::counterex3}.
		
		Now, let us construct a perturbation set~$\Sigma$.
		We consider~$\Sigma_+$ such that it is enclosed between~$\partial E$ and~$\graph(\phi)$, where~$\phi$ is a~$\cont^1$-function in~$[x_0,1)$ (to be chosen appropriately in what follows). In particular, notice that~$E$ is smooth in~$\Omega_0^c$.
		
		Moreover, given~$(x_0, y_0)\in \partial E\cap \partial\Omega_0$ and some small~$\xi>0$, we choose~$\phi$ so that it agrees with~$y_0+(x- x_0)^{1+s+\xi}$ in a right neighborhood of~$x_0$. Then, we define 
		$$ \Sigma_-=\{(x,y): (-x,y)\in\Sigma_+\} \qquad	\text{and} \qquad\Sigma=\Sigma_-\cup\Sigma_+ .	$$
		According to Theorem~\ref{th::perturbed_almost_min}, $\widetilde{E}=E\cup\Sigma$ is an almost minimizer for~$\Per_s(\cdot,\Omega_0)$ with boundary datum~$\widetilde{E}\setminus\Omega_0$. Furthermore, $\widetilde{E}$ does not stick at the boundary. 
		
		In addition, $\widetilde{E}$ is a~$\cont^0$, but not~$\cont^1$, subgraph in a small neighborhood of~$\Omega_0$ (see Figure~\ref{fig::counterex4}). Indeed,
		$$ \lim_{x\to x_0^-} u'(x) \neq 0 = \frac{d}{dx}\left((x-x_0)^{1+s+\xi}\right)\big|_{x=x_0}.$$

		\begin{figure}[!h]
			\centering
			\includegraphics[width=.7\linewidth]{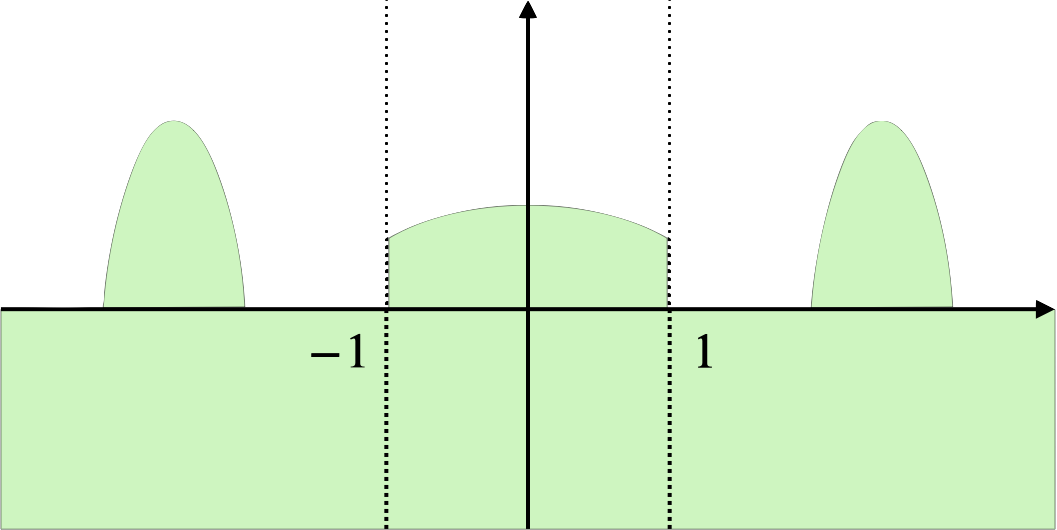}
			\caption{The sticking half-space.}\label{fig::counterex1}
		\end{figure}
		\begin{figure}[!h]
			\centering
			\includegraphics[width=.7\linewidth]{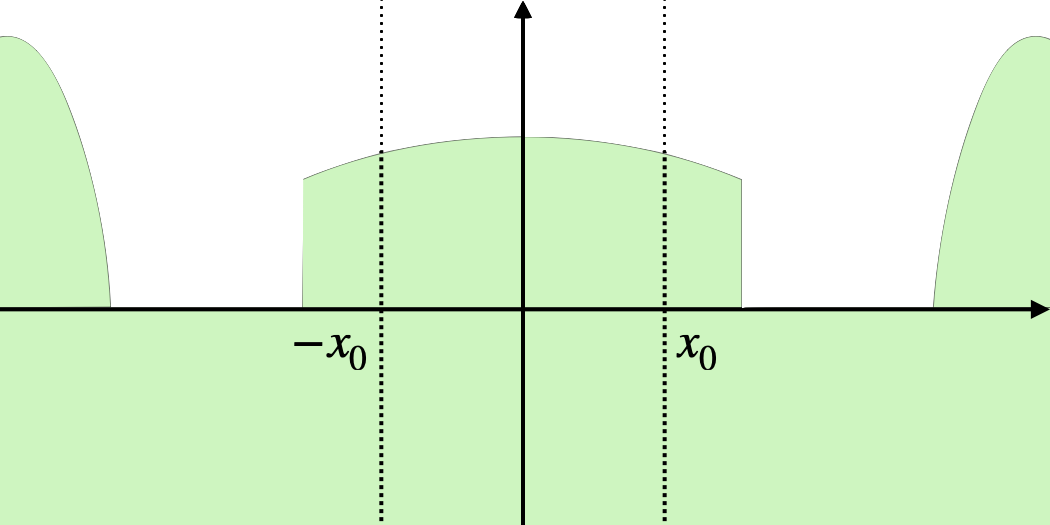}
			\caption{A zoom in of the sticking half-space.}\label{fig::counterex2}
		\end{figure}
		\begin{figure}[!h]
			\centering
			\includegraphics[width=.7\linewidth]{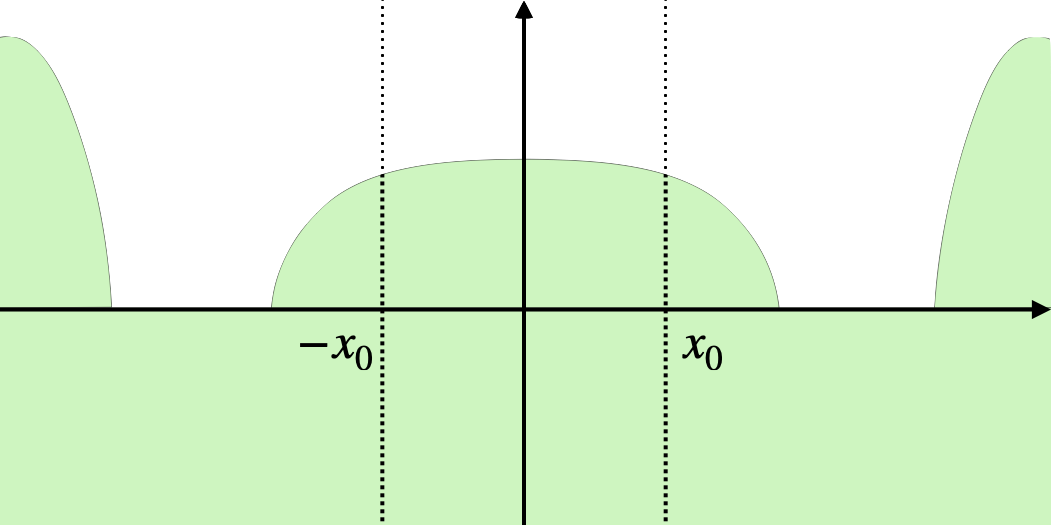}
			\caption{Smoothing out the external datum.}\label{fig::counterex3}
		\end{figure}
		\begin{figure}[!h]
			\centering
			\includegraphics[width=.7\linewidth]{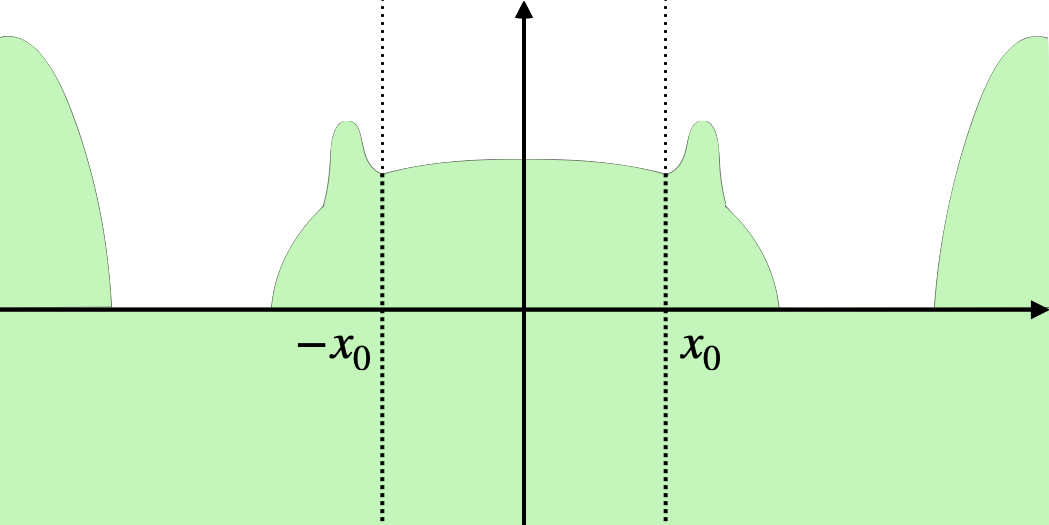}
			\caption{The non-sticking~$\cont^0\setminus\cont^1$ almost minimizer.}\label{fig::counterex4}
		\end{figure}
	\end{ex}
	
	
	\section{An interesting application: the non-local Massari's Problem
		and proofs of Theorems~\ref{th::existence}
		and~\ref{th::regularity}} \label{sec::massari}
	
	The regularity theory of almost minimizers that we have discussed so far
	finds a convenient application to the non-local version of the Massari's Problem. In particular, we will apply 
	the regularity theory developed in Section~\ref{sec::holder_reg}
	to establish the existence and regularity of sets with prescribed non-local mean curvature.
	
	
	We start with the following observation, that relates
	the non-local mean curvature (as defined in~\eqref{meancurvt5749})
	of a minimizer for~$\mathscr{J}^H_{\Omega,s}$ with the function~$H$,
	and justifies Definition~\ref{def::var_frac_min_curv}.
	
	\begin{lemma}\label{lma:visccomecaff}
		Let~$A$ be a smooth open set in~$\Omega$ that minimizes~$\mathscr{J}^H_{\Omega,s}$.
		
		Then, $H_s[A]=H$ in~$\Omega$,
		that is~$A$ has fractional mean curvature~$H$ in~$\Omega$.
	\end{lemma}
	
	\begin{proof}
		Let~$\nu$ be the exterior unit normal to~$\partial A$, and consider a smooth compactly supported perturbation function~$h$. For every~$t>0$, we define the perturbed set~$A_t$ such that
		$$ \partial A_t := \big\{x+th(x)\nu(x): x\in\partial A\big\}.$$
		Since~$A$ is a smooth minimizer for~$\mathscr{J}^H_{\Omega,s}$, we have that
		\begin{equation}\label{65748392vbncmxgfhdjytruei0} \frac{d}{dt} \mathscr{J}^H_{\Omega,s}(A_t) \big\rvert_{t=0} = 0.\end{equation}
		
		The first variation of the fractional perimeter has been computed e.g. in~\cite{MR3798717}, and is given by
		\begin{equation}\label{65748392vbncmxgfhdjytruei}
			\frac{d}{dt} \Per_s(A_t,\Omega)\big\rvert_{t=0} = -\int_{\partial A} \int_{\R^n} h(y)\frac{\chi_A(x)-\chi_{A^c}(x)}{|x-y|^{n+s}}\,dx\,
			d\haus{n-1}(y).\end{equation}
		Moreover,
		$$ \frac{d}{dt} \left(\int_{A_t\cap \Omega} H(x)\, dx\right) \big\rvert_{t=0} = \int_{\partial A\cap \Omega} h(x)H(x)\, d\haus{n-1}(x). $$
		{F}rom this, \eqref{65748392vbncmxgfhdjytruei0} and~\eqref{65748392vbncmxgfhdjytruei}, we obtain the desired result.
	\end{proof}
	
	We point out that Lemma~\ref{lma:visccomecaff} is the analog in this setting of the fractional mean curvature equation, presented in~\cite[Section~5]{caffarelli_roquejoffre_savin_nonlocal}. Indeed, if~$A$ is an~$s$-minimal surface, then
	\begin{equation*} 
		H_s[A] = 0 \text{ in }\partial A\cap\Omega,
	\end{equation*}
	in the viscosity sense (and pointwise if~$A$ is smooth).
	
	For more details about the fractional mean curvature, we refer to~\cite{abatangelo_valdinoci_nonlocal_curv}.
	
	\subsection{Existence of sets with prescribed non-local mean curvature}
	
	We use the classical direct methods in the calculus of variations to prove the existence of sets with prescribed non-local mean curvature. 
	
	For later convenience, we provide the following compactness result.
	
	\begin{theorem}[Compactness of the class of finite fractional perimeter sets] \label{th::compact_finite_Pers}
		Let~$\mathcal{M}^n$ be the family of all measurable subsets of~$\R^n$ and consider a function~$f:\mathcal{M}^n\to\R$.
		
		Then, the class 
		$$\mathcal{C}_f=\big\{ F\in \mathcal{M}^n\;{\mbox{ s.t. }}\; \Per_s(F,K)\leq f(K), \;{\mbox{ for all }}  K\comp\R^n\big\} $$ 
		is compact in~$L^1_{\loc}(\R^n)$.
	\end{theorem}
	
	\begin{proof} 
		Let~$K$ be an arbitrary compact set of~$\R^n$. Notice that, for every~$F\in\mathcal{C}_f$, we have that 
		\begin{equation*}
			\begin{split}
				&	\norm{\chi_{F}}_{H^s(K)}=\norm{\chi_{F}}_{L^2(K)}+[\chi_{F}]_{H^s(K)}\leq\norm{\chi_{F}}_{L^2(K)}+\Per_s(F,K)\\
				&\qquad\qquad\leq |K|^{1/2}+f(K) =: g(K).
			\end{split}
		\end{equation*}
		This gives that~$\chi_{F}\in B^s_{g(K)}$, 
		where~$B^s_{g(K)}$ denotes
		the ball of radius~$g(K)$ and centered at 0 in the Hilbert space~$H^s(K)$. 
		
		Moreover, $B^s_{g(K)}$ is weakly compact since~$H^s(K)$ is a Hilbert space. Therefore, for every sequence~$\{E_j\}_j$ in~$\mathcal{C}_f$, there exist some~$\psi\in B^s_{g(K)}$ and a subsequence~$\{E_{j_l}\}_l$ such that~$\chi_{E_{j_l}}\rightharpoonup \psi$ in~$H^s(K)$.
		
		Thanks to the fractional Rellich-Kondrachov's Theorem, $H^s(K)$ is compactly embedded in~$L^2(K)$. Thus, $E_{j_l}\to \psi$ strongly in~$L^2(K)$, hence in~$L^1(K)$, up to a subsequence. Since~$K$ is arbitrarily chosen, we deduce that~$\chi_{E_{j_l}}$ converges~$\psi$ in~$L^1_{\loc}(\R^n)$. 
		
		We recall now that the~$L^1_{\loc}$-convergence implies the convergence almost everywhere along a subsequence. 
		Hence, $\psi$ must be a characteristic function of some set~$E$, i.e.~$\psi=\chi_{E}$ (to see this, consider a sequence of nested sets compactly contained in~$\R^n$ and use a diagonal argument). 
		
		Moreover, by lower semi-continuity of the fractional perimeter, we have that
		$$ \Per_s(E,K)\leq\liminf_{l\to+\infty}\Per_s(E_{j_l},K)\leq f(K),\qquad
		{\mbox{for all }} K\comp\R^n.$$
		Therefore, $E\in\mathcal{C}_f$, and this concludes the proof.
	\end{proof}
	
	As a consequence of Theorem~\ref{th::compact_finite_Pers}, we infer the following:
	
	\begin{cor} \label{cor::C_Omega_comp}
		Let~$\{E_j\}_j$ be a sequence of sets such that~$E_j\setminus\Omega=M\setminus\Omega$ and
		$$\Per_s(E_j,\Omega)<c,$$
		for some constant~$c>0$ independent of~$j$.
		
		Then, there exist a subsequence~$\{E_{j_l}\}_l$ and a set~$E$ such that~$\chi_{E_{j_l}}\to \chi_{E}$ in~$L^1(\Omega)$, $\Per_s(E,\Omega)<c$, and~$E\setminus\Omega=M\setminus\Omega$.
	\end{cor}
	
	
	We can now complete the proof of Theorem~\ref{th::existence}.
	
	\begin{proof}[Proof of Theorem~\ref{th::existence}]
		Set~$m:=\inf\{\mathscr{J}^H_{\Omega,s}(F): F\in\mathscr{F}\}$. We show that~$m$ is finite.
		
		Notice that~$M\in\mathscr{F}$, so that
		$$m\leq\mathscr{J}^H_{\Omega,s}(M)\leq \Per_s(M,\Omega)+\int_{\Omega} H(x)\, dx<+\infty.$$
		Moreover, for every~$F\in\mathscr{F}$, we have that
		$$ \mathscr{J}^H_{\Omega,s}(F)\geq -\left|\int_{E\cap \Omega} H(x)\, dx\right|\geq-\norm{H}_{L^1(\Omega)}>-\infty.$$
		
		Now, consider a minimizing sequence~$\{E_j\}_j\subseteq\mathscr{F}$ such that~$\mathscr{J}^H_{\Omega,s}(E_j)\to m$. Therefore, the sequence~$\{\mathscr{J}^H_{\Omega,s}(E_j)\}_j$ is bounded in~$\R$ by a constant~$C>0$ independent of~$j$, and consequently 
		$$ \Per_s(E_j,\Omega) \leq |\mathscr{J}^H_{\Omega,s}(E_j)| + \left| \int_{E_{j_l}\cap \Omega} H(x)\,dx \right| \leq C + \norm{H}_{L^1(\Omega)}=:c_{\Omega,H}.$$
		Thanks to Corollary~\ref{cor::C_Omega_comp}, there exist a subsequence~$\{E_{j_l}\}_l$ and a set~$E\in\mathscr{F}$ such that~$E_{j_l}\to E$ in~$L^1(\Omega)$. Hence, $E_{j_l}\to E$ pointwise, up to a subsequence. 
		
		Moreover, 
		$$|(\chi_{E_{j_l}}-\chi_{E})H|\leq |H|\in L^1(\Omega).$$ 
		Then, using the Dominated Convergence Theorem, we obtain
		\begin{equation*}
			\lim_{l\to+\infty}	\left| \int_{E\cap \Omega}H(x)\,dx - \int_{E_{j_l}\cap \Omega}H(x)\,dx\right| \leq \lim_{l\to+\infty}\int_\Omega \left|\left[\chi_{E_{j_l}}(x)-\chi_{E}(x)\right]H(x)\right| \,dx =0.
		\end{equation*}
		{F}rom this and the lower semi-continuity of~$\Per_s$, we get that
		\begin{equation*}
			\begin{split}
				&		\mathscr{J}^H_{\Omega,s}(E) 
				= \Per_s(E,\Omega)+\int_{E\cap\Omega}H\,dx \\
				&\quad\leq \liminf_{l\to+\infty} \Per_s(E_{j_l},\Omega)+\int_{E_{j_l}\cap\Omega}H\,dx 
				= \liminf_{l\to+\infty} \mathscr{J}^H_{\Omega,s}(E_{j_l}) = m.
			\end{split}
		\end{equation*}
		Hence, $E$ is a minimizer for~$\mathscr{J}^H_{\Omega,s}$.
	\end{proof}
	
	\subsection{Regularity of sets with prescribed non-local mean curvature}\label{gfyeui65748123456wfgh007686}
	
	Now, we apply the almost minimizers theory to the non-local Massari's Problem to deduce regularity for sets with prescribed non-local mean curvature. We point out that, in order to achieve this goal, some additional assumption on~$H$ is needed, namely we require~$H\in L^\infty(B_1)$.
	
	\begin{proof}[Proof of Theorem~\ref{th::regularity}]
		Let~$F$ be such that~$F\setminus B_1=E\setminus B_1$. {F}rom the minimality of~$E$ for~$\mathscr{J}^H_{B_1,s}$, we see that
		\begin{align*}
			&\Per_s(E,B_1)-\Per_s(F,B_1)
			= \mathscr{J}^H_{B_1,s}(E)-\int_{E\cap B_1} H\, dx - \mathscr{J}^H_{B_1,s}(F) + \int_{F\cap B_1} H\,dx\\
			&\qquad\qquad\leq \int_{F\cap B_1} H\, dx-\int_{E\cap B_1} H\, dx
			\leq \int_{(E\Delta F)\cap B_1} |H|\, dx \\
			&\qquad\qquad \leq \norm{H}_{L^\infty(B_1)} |E\Delta F|.
		\end{align*}
		Hence, $E$ is~$\norm{H}_{L^\infty(B_1)}$-minimal for~$\Per_s$ in~$B_1$.
		{F}rom this and Theorem~\ref{th::holder_reg_almost_min}, we deduce
		the claim of Theorem~\ref{th::regularity}.
	\end{proof}
	
	
	\appendixtitleon
	\appendixtitletocon
	\begin{appendices}
		
		\section{Proof of Proposition~\ref{prop::conv_almost_min}} \label{sec::prop_conv_almost_min}
		
		Let us assume that~$\Omega$ is smooth.
		Let~$F$ be such that~$F\setminus \Omega=E\setminus \Omega$ and define 
		\begin{equation} \label{eq::competitor_choice}
			F_k:=(F\cap \Omega)\cup(E_k\setminus \Omega),\qquad {\mbox{for all }} k\in\N.
		\end{equation}
		By construction, $F_k\setminus \Omega=E_k\setminus \Omega$. 
		Thus, by the almost minimality of~$E_k$,
		\begin{equation} \label{eq::almost_min_conv1}
			\Per_s(E_k,\Omega)\leq \Per_s(F_k,\Omega)+\Lambda_k|E_k\Delta F_k| = \Per_s(F_k,\Omega)+\Lambda_k|(E_k\Delta F)\cap\Omega|.
		\end{equation}
		
		We now show that we can rewrite~\eqref{eq::almost_min_conv1} as
		\begin{equation} \label{eq::almost_min_conv1_2}
			\Per_s(E_k,\Omega)\leq\Per_s(F,\Omega)+\Lambda_k|(E_k\Delta F)\cap\Omega| +b_k,
		\end{equation}
		where~$b_k$ is an infinitesimal sequence as~$k\to+\infty$. 
		
		To check this, we observe that
		\begin{align*}
			&|\Per_s(F,\Omega)-\Per_s(F_k,\Omega)| \\
			=&\big| \mathscr{L}(F\cap \Omega,F^c)+\mathscr{L}(F\cap \Omega^c,F^c\cap \Omega)
			-\mathscr{L}(F_k\cap \Omega,F_k^c) -\mathscr{L}(F_k\cap \Omega^c,F_k^c\cap \Omega)\big|\\
			=&\big| \mathscr{L}(F\cap \Omega,F^c\cap \Omega^c)+\mathscr{L}(F\cap \Omega^c,F^c\cap \Omega) 
			-\mathscr{L}(F\cap \Omega,F_k^c\cap \Omega^c)-\mathscr{L}(F_k\cap \Omega^c,F^c\cap \Omega)\big|.
		\end{align*} 
		Writing explicitly the functional~$\mathcal{L}$, we obtain that
		\begin{align*}
			&|\Per_s(F,\Omega)-\Per(F_k,\Omega)| \\
			&\qquad =\left| \int_{F\cap \Omega}\int_{\Omega^c}\frac{\chi_{F^c}(x)-\chi_{F_k^c}(x)}{|x-y|^{n+s}}\,dx\,dy
			+\int_{F^c\cap \Omega}\int_{\Omega^c}\frac{\chi_{F}(x)-\chi_{F_k}(x)}{|x-y|^{n+s}}\,dx\,dy\right|\\
			&\qquad \leq \int_{F\cap \Omega}\int_{\Omega^c}\frac{|\chi_{F^c}(x)-\chi_{F_k^c}(x)|}{|x-y|^{n+s}}\,dx\,dy
			+\int_{F^c\cap \Omega}\int_{\Omega^c}\frac{|\chi_{F}(x)-\chi_{F_k}(x)|}{|x-y|^{n+s}}\,dx\,dy\\
			&\qquad=\int_{F\cap \Omega}\int_{\Omega^c} \frac{\chi_{F^c\Delta F_k^c}(x)}{|x-y|^{n+s}}\,dx\,dy+\int_{F^c\cap \Omega}\int_{\Omega^c}\frac{\chi_{F\Delta F_k}(x)}{|x-y|^{n+s}}\,dx\,dy.
		\end{align*} 
		Recalling that, for every sets~$A$ and~$B$, we have that~$A^c\Delta B^c=A\Delta B$, and using that~$F=E$
		and~$F_k=E_k$ in~$\Omega^c$, we arrive at
		\begin{equation*}
			\begin{split}
				&	|\Per_s(F,\Omega)-\Per(F_k,\Omega)| 
				\leq\int_{\Omega}\int_{\Omega^c}\frac{\chi_{F\Delta F_k}(x)}{|x-y|^{n+s}}\,dx\,dy\\
				&\qquad\qquad=\mathscr{L}(\Omega,(F\Delta F_k)\cap \Omega^c)=\mathscr{L}(\Omega,(E\Delta E_k)\cap \Omega^c)=:b_k.
			\end{split}
		\end{equation*}
		Using this into~\eqref{eq::almost_min_conv1}, we obtain the expression in~\eqref{eq::almost_min_conv1_2}.
		
		Hence, to complete the proof of~\eqref{eq::almost_min_conv1_2}, it remains to check that~$b_k\to0$ as~$k\to+\infty$. To this aim, let 
		\begin{align*}
			&\Omega_\rho:=\{x\in\R^n \mbox{ s.t. }\dist(x,\Omega)<\rho\}\\
			{\mbox{and}}\qquad &a_k(\rho):=\haus{n-1}\left((E\Delta E_k)\cap\partial \Omega_\rho\right).
		\end{align*}
		Then, for any~$R>0$ we have that
		\begin{equation}\label{eq::b_k}
			b_k = \mathscr{L}\left( \Omega,(E\Delta E_k)\cap( \Omega_R\setminus \Omega) \right)+\mathscr{L}(\Omega,(E\Delta E_k)\cap \Omega_R^c).
		\end{equation}
		
		We now check that
		\begin{equation}\label{6785943zaqwsxderfvb}
			\lim_{k\to+\infty}\mathscr{L}\left( \Omega,(E\Delta E_k)\cap( \Omega_R\setminus \Omega)\right)=0.
		\end{equation}
		Indeed, by the co-area formula, we have that
		\begin{equation*} 
			\begin{split}
				&\mathscr{L}\left(\Omega,(E\Delta E_k)\cap( \Omega_R\setminus \Omega)\right)
				= \int_{\Omega_R\setminus\Omega}\int_\Omega \frac{\chi_{E\Delta E_k}(x)}{|x-y|^{n+s}}\,dx\,dy\\
				&\qquad =\int_{0}^{R} \int_{\partial\Omega_\rho}\int_\Omega \frac{\chi_{E\Delta E_k}(x)}{|x-y|^{n+s}} dy\,d\haus{n-1}(x)\,d\rho\\
				&\qquad =\int_{0}^{R} \int_{\partial\Omega_\rho}\int_{\Omega-x} \frac{\chi_{E\Delta E_k}(x)}{|z|^{n+s}} dz\,d\haus{n-1}(x)\,d\rho.\\
			\end{split}
		\end{equation*}
		By the definition of~$\Omega_\rho$, we have that~$|x-y|\geq\rho$, for every~$x\in\partial\Omega_\rho$ and~$y\in\Omega$, thus we deduce that~$\Omega-x\subseteq B_\rho^c$, for any~$x\in\partial\Omega_\rho$. 
		
		Consequently,
		\begin{equation} \label{eq::ak_rho_estimate}
			\begin{split}
				&\mathscr{L}\left(\Omega,(E\Delta E_k)\cap( \Omega_R\setminus \Omega)\right)
				\leq \int_{0}^{R} \int_{\partial\Omega_\rho}\int_{B_\rho^c} \frac{\chi_{E\Delta E_k}(x)}{|z|^{n+s}} dz\,d\haus{n-1}(x)\,d\rho\\
				&\qquad \leq \frac{c_n}{s}  \int_{0}^{R} \int_{\partial\Omega_\rho} \frac{\chi_{E\Delta E_k}(x)}{\rho^s}d\haus{n-1}(x)\,d\rho = \frac{c_n}{s}\int_{0}^{R}\frac{a_k(\rho)}{\rho^s}d\rho.
			\end{split}
		\end{equation}
		
		Now, observe that~$a_k(\rho)\to0$, as~$k\to+\infty$, and
		\begin{equation*}
			\frac{a_k(\rho)}{\rho^s} \leq \frac{\haus{n-1}(\partial \Omega_\rho)}{\rho^s} \leq \frac{\haus{n-1}(\partial \Omega_R)}{\rho^s}\in L^1\left((0,R)\right).
		\end{equation*}
		Therefore, thanks to the Dominated Convergence Theorem and~\eqref{eq::ak_rho_estimate},
		\begin{equation*}
			\lim_{k\to+\infty}	 \mathscr{L}\left(\Omega,(E\Delta E_k)\cap( \Omega_R\setminus \Omega)\right)\leq\lim_{k\to+\infty} c_n\int_{0}^{R}\frac{a_k(\rho)}{\rho^s}d\rho = 0,
		\end{equation*}
		which establishes~\eqref{6785943zaqwsxderfvb}.
		
		Besides, the term~$\mathscr{L}\left(\Omega,(E\Delta E_k)\cap \Omega_R^c\right)$ in~\eqref{eq::b_k} is bounded uniformly in~$k$. Indeed, since $\Omega$ is bounded, there exists $M>0$ such that $\Omega\subseteq B_M$. Then, for every $x\in\Omega_R^c$, we have that
		\begin{equation*}
			\mbox{dist}(x,B_M) \geq  \mbox{dist}(x,\Omega) - \mbox{diam}(B_M) \geq R-2M\geq R/2,
		\end{equation*}
		whenever $R$ is large enough. Hence, $\Omega_R^c\subseteq B_{M+R/2}^c$. This entails that
		\begin{equation}  \label{eq::bk_to_0_2}
			\begin{split}
				&\mathscr{L}\left(\Omega,(E\Delta E_k)\cap \Omega_R^c\right)
				= \int_\Omega \int_{\Omega_R^c} \frac{\chi_{E\Delta E_k}(x)}{|x-y|^{n+s}}dxdy\\
				&\qquad\leq \int_{B_M} \int_{B_{M+R/2}^c} \frac{dxdy}{|x-y|^{n+s}}\\
				&\qquad\leq C \int_{R/2}^{+\infty} \frac{d\rho}{\rho^{1+s}} \leq \frac{C}{R^{s}},
			\end{split}
		\end{equation}
		for some constant $C>0$ depending only on $n$, $s$, and $M$, and possibly changing from line to line.
		
		Thus, from~\eqref{6785943zaqwsxderfvb} and~\eqref{eq::bk_to_0_2}, we deduce that 
		\begin{equation*}
			0\leq\limsup_{k\to+\infty} b_k \leq\lim_{R\to+\infty} \frac{C}{R^{s}} = 0,
		\end{equation*}
		which shows that~$b_k$ is infinitesimal and completes the proof of~\eqref{eq::almost_min_conv1_2}.
		
		Using the lower semi-continuity of the fractional perimeter and~\eqref{eq::almost_min_conv1_2}, we obtain that
		\begin{equation} \label{eq::almost_min_conv3}
			\begin{split}
				&	\Per_s(E,\Omega)
				\leq\liminf_{k\to+\infty}\Per_s(E_k,\Omega) \leq \limsup_{k\to+\infty}\Per_s(E_k,\Omega)\\
				&\qquad\leq \limsup_{k\to+\infty} \left(\Per_s(F,\Omega)+\Lambda_k|(E_k\Delta F)\cap\Omega| +b_k\right)\\&\qquad\leq \limsup_{k\to+\infty} \left(\Per_s(F,\Omega)+\Lambda_k|(E_k\Delta F)\cap\Omega|\right).
			\end{split}
		\end{equation}
		Now, an easy computation shows that
		\begin{equation*}
			\begin{split}
				& \lim_{k\to+\infty}\left| |(E_k\Delta F)\cap\Omega|-|E\Delta F|\right| \leq \lim_{k\to+\infty} \left|\left( (E_k\Delta F)\Delta(E\Delta F)\right) \cap\Omega \right|\\
				&\qquad\qquad \leq  \lim_{k\to+\infty}\left|(E_k\Delta E)\cap\Omega \right|=0.
			\end{split}
		\end{equation*}
		Therefore, this and~\eqref{eq::almost_min_conv3} give that
		\begin{equation*} 
			\Per_s(E,\Omega) \leq \Per_s(F,\Omega)+\Lambda|E\Delta F|,
		\end{equation*}
		which shows that~$E$ is~$\Lambda$-minimal in~$\Omega$.
		
		We observe that, choosing~$F:=E$ in~\eqref{eq::competitor_choice}, we obtain that
		$$\limsup_{k\to+\infty}\Per_s(E_k,\Omega) \leq\Per_s(E,\Omega),$$ 
		which gives the claim in~\eqref{lapers} when~$\Omega$ has smooth boundary.
		
		Furthermore, if~$\Omega$ is not smooth, by interior smooth approximation
		there exists a sequence~$\{\Omega^h\}_h$ of smooth sets such that~$\Omega^h\subseteq\Omega$, for every~$h$, and~$\Omega^h\to\Omega$ locally in~$\R^n$. Then, the desired result follows from the the continuity of the~$s$-perimeter with respect to the~$L^1_{\loc}$-convergence of the domain.
		
		
		\section{A general corollary of Theorem~\ref{th::unif_dens_estimates}}\label{sec::unif_density_estimates_cor}
		
		In this section, we provide a general corollary of the uniform density estimates in Theorem~\ref{th::unif_dens_estimates}, since we think that it is interesting in its own right. Also observe that from this result we deduce straight-forwardly the analogous for almost minimal sets as stated in Corollary~\ref{cor::boundary_conv}.
		
		\begin{cor} \label{cor::boundary_conv_general}
			Let~$\{E_k\}_k$ be a sequence of sets in~$\R^n$ such that~$E_k\to E$ in~$L^1_{\loc}(\R^n)$, for some set~$E$ of finite~$s$-perimeter.
			
			Suppose that
			there exist constants~$r_0>0$ and~$c_0\in(0,1)$
			such that, for any~$k\in\N$ and for any~$x_k\in(\partial E_k)\cap\Omega$ and~$r\in(0,\min\{r_0,\dist(x_k,\partial\Omega)\})$,
			\begin{equation}\label{eq::unif_dens_estimatesBIS}
				c_0r^n\leq |E_k\cap B_r(x_k)|\leq (1-c_0)r^n.
			\end{equation} 
			
			Then, for every~$\epsilon>0$ and for every compact set~$K$ of~$\Omega$, there exists~$k_0\in\N$ such that, for all~$k\geq k_0$,
			\begin{eqnarray}
				\label{eq::E_k_close_E}&\partial E_k\cap K \subseteq \mathscr{U}_\epsilon(\partial E)\cap K\\
				\label{eq::E_close_E_k}\mbox{and}&\partial E\cap K \subseteq \mathscr{U}_\epsilon(\partial E_k)\cap K.
			\end{eqnarray}
			
		\end{cor}
		
		\begin{proof} 
			We start by
			proving~\eqref{eq::E_k_close_E}. Let~$\epsilon_0>0$ and~$r>0$ be such that~$B_r\comp\Omega$. 
			
			Arguing by contradiction, let us suppose that there exists a (possibly relabeled) sequence
			of points~$x_k\in\partial E_k\cap \overline{B}_r$, with~$\dist(x_k,\partial E)>\epsilon_0$. 
			Then,
			$$ E_k\cap B_{\epsilon_0/2}(x_k) \subseteq E_k\setminus E\not=\varnothing.$$
			
			Moreover, assume that~$\epsilon_0$ is so small that~$B_{\epsilon_0/2}(x_k) \subseteq\Omega$. Therefore, the uniform density estimates in~\eqref{eq::unif_dens_estimatesBIS} give that
			$$ |E_k\cap B_{\epsilon_0/2}(x_k)| \geq c_02^{-n} \epsilon_0^n.$$ 
			Since~$x_k\in\overline{B}_r$, we have that~$B_{\epsilon_0/2}(x_k)\subseteq B_R$, with~$R:=r+\epsilon_0$, and therefore
			$$ |(E\Delta E_k)\cap B_R| \geq  |(E_k\setminus E)\cap B_R| \geq |E_k\cap B_{\epsilon_0/2}(x_k)| \geq c_02^{-n} \epsilon_0^n,$$
			against the~$L^1_{\loc}$-convergence.
			
			This proves~\eqref{eq::E_k_close_E} and
			we now show~\eqref{eq::E_close_E_k}. To this end, let again~$\epsilon_0>0$ and~$r>0$ be such that~$B_r\comp\Omega$, and assume by contradiction that there exists a (possibly relabeled)
			sequence of points~$x_k\in\partial E\cap \overline{B}_r$, with~$\dist(x_k,\partial E_k)>\epsilon_0$. This reads
			$$ E\cap B_{\epsilon_0/2}(x_k) \subseteq E\setminus E_k\not=\varnothing.$$
			
			Furthermore, assuming that~$B_{\epsilon_0/2}(x_k) \subseteq\Omega$, it follows from~\eqref{eq::unif_dens_estimatesBIS} that
			$$ |E\cap B_{\epsilon_0/2}(x_k)| \geq c_02^{-n} \epsilon_0^n.$$
			
			As shown above, since~$x_k\in\overline{B}_r$, we have that~$B_{\epsilon_0/2}(x_k)\subseteq B_R$, with~$R:=r+\epsilon_0$. It follows that
			$$ |(E\Delta E_k)\cap B_R| \geq  |(E\setminus E_k)\cap B_R| \geq |E\cap B_{\epsilon_0/2}(x_k)| \geq c_02^{-n} \epsilon_0^n,$$
			which contradicts the fact that~$E_k\to E$ in~$L^1_{\loc}$, concluding the proof.
			
		\end{proof}
		
		\section{Properties of the perturbation~$T_\epsilon$} \label{sec::perturbation_properties}
		Here, we prove some properties of the perturbation~$T_\epsilon$ introduced in Section~\ref{sec::EL_ineq}. To this purpose, recall that~$r_{x,\epsilon} := \dist(x, \partial V_{R,\epsilon})$, and 
		$$ T_\epsilon(x):=-x-2Re_n+2(R+\mbox{d}_\epsilon(x))\frac{x+Re_n}{|x+Re_n|},$$
		where~$\mbox{d}(x):=R^{-1}(1-|x'|^2)_+$, and~$\mbox{d}_\epsilon(x)=\epsilon^2\mbox{d}(x/\epsilon)$. 
		
		Then, we have the following:
		
		\begin{prop}\label{prop::perturbation_properties}
			For every~$\epsilon\in\left(0,\frac{1}{3n}\right)$, and
			for all~$x,y\in V_{2R,\epsilon}\setminus V_{0,\epsilon}$, it holds that
			\begin{equation} \label{eq::perturb_prop1}
				\|DT_\epsilon(x)-\mathcal{R}_x\|\leq\frac{2}{R}(3nr_{x,\epsilon}+2|x'|),
			\end{equation}
			and
			\begin{equation} \label{eq::perturb_prop2}
				\left| \frac{|T_\epsilon(x)-T_\epsilon(y)|}{|x-y|}-1\right|\leq
				\frac{2}{R}\max\big\{3nr_{x,\epsilon}+2|x'|,\,3nr_{y,\epsilon}+2|y'|\big\}.
			\end{equation}
			
			Moreover, we have the inclusions
			\begin{equation}\label{eq::perturb_inclusions}
				A_\epsilon^-\subseteq B_{2\epsilon}\qquad {\mbox{and}}\qquad B_{\epsilon^2/(2R)}\setminus E\subseteq A_\epsilon\subseteq B_{8\epsilon}.
			\end{equation}
		\end{prop}
		
		\begin{proof}
			Let~$x\in V_{2R,\epsilon}\setminus V_{0,\epsilon}$. Differentiating~$T_\epsilon$ at~$x$, we obtain that
			\begin{equation} \label{eq::DT_estimate1}
				\begin{split}
					\partial_j(T_\epsilon)_i(x) 
					&= \left(2\frac{R+\mbox{d}_\epsilon(x)}{|x+Re_n|}-1\right)\delta_{i,j} - 2 \frac{R+\mbox{d}_\epsilon(x)}{|x+Re_n|^3}(x+Re_n)_i(x+Re_n)_j \\
					&\qquad\quad+2\epsilon(\partial_j \mbox{d})\left(\frac{x}{\epsilon}\right) \frac{(x+Re_n)_i}{|x+Re_n|},\quad \mbox{for all },j\in\{1,\dots,n\}.
				\end{split}
			\end{equation}
			Since~$|x+Re_n|=R+\mbox{d}_\epsilon(x)\pm r_{x,\epsilon}$, this entails that
			\begin{equation}\label{eq::DT_estimate1_2}
				\begin{split}
					\partial_j(T_\epsilon)_i(x) 
					&= (\mathcal{R}_x)_i + 2\left(\frac{R+\mbox{d}_\epsilon(x)}{R+\mbox{d}_\epsilon(x)\pm r_{x,\epsilon}}-1\right)\left(\delta_{i,j}-\frac{(x+Re_n)_i}{|x+Re_n|}\frac{(x+Re_n)_j}{|x+Re_n|}\right)\\
					&\quad\qquad+2\epsilon(\partial_j \mbox{d})\left(\frac{x}{\epsilon}\right) \frac{(x+Re_n)_i}{|x+Re_n|},
				\end{split}
			\end{equation}
			Recalling that for every couple of vectors~$v$, $w\in\R^n$, $v\otimes w$ is defined as the~$n\times n$ matrix with entries~$\{v_iw_j\}_{i,j\in\{1,\dots,n\}}$, we rewrite~\eqref{eq::DT_estimate1_2} in the operatorial form as
			\begin{equation} \label{eq::DT_estimate2}
				\begin{split}
					DT_\epsilon(x) 
					&= \mathcal{R}_x + 2\left(\frac{R+\mbox{d}_\epsilon(x)}{R+\mbox{d}_\epsilon(x)\pm r_{x,\epsilon}}-1\right)\left(\mbox{Id}-\frac{x+Re_n}{|x+Re_n|}\otimes\frac{x+Re_n}{|x+Re_n|}\right)\\
					&\quad\qquad+2\epsilon\frac{x+Re_n}{|x+Re_n|}\otimes (\nabla \mbox{d})\left(\frac{x}{\epsilon}\right),
				\end{split}
			\end{equation}
			where~$\mathcal{R}_x$ is defined as in~\eqref{eq::def_reflex}. 
			
			Notice that we also have 
			\begin{eqnarray}
				\label{eq::DT_estimate3} &&\left| \frac{R+\mbox{d}_\epsilon(x)}{R+\mbox{d}_\epsilon(x)\pm r_{x,\epsilon}}-1\right| =  \frac{r_{x,\epsilon}}{R+\mbox{d}_\epsilon(x)\pm r_{x,\epsilon}} \\
				\label{eq::DT_estimate4} \text{and }	&&|\nabla \mbox{d}(x)| \leq 2\frac{|x'|}{R}.
			\end{eqnarray}
			
			Let us denote by~$\|\cdot\|$ the Euclidean norm for linear operators. Then, \eqref{eq::DT_estimate2}, together with~\eqref{eq::DT_estimate3} and~\eqref{eq::DT_estimate4}, entails that
			$$ \|DT_\epsilon(x) - \mathcal{R}_x \| \leq 3n\frac{r_{x,\epsilon}}{R+\mbox{d}_\epsilon(x)\pm r_{x,\epsilon}} + 4\frac{|x'|}{R} \leq 6n\frac{r_{x,\epsilon}}{R} +  4\frac{|x'|}{R} ,$$
			as long as~$\epsilon<R/2$, proving~\eqref{eq::perturb_prop1}.
			
			As a consequence of~\eqref{eq::perturb_prop1} and the triangle inequality, we also have that
			\begin{equation*}
				\left|\norm{DT_\epsilon(x)}-\norm{\mathcal{R}_x}\right| \leq {2}{R}(3nr_{x,\epsilon}+2|x'|).
			\end{equation*}
			Thus,
			using the fact that~$\mathcal{R}_x$ is an isometry and integrating along the segment~$[x,y]$, for any~$x$, $y\in V_{2R,\epsilon}\setminus V_{0,\epsilon}$, we obtain that
			\begin{equation*}
				\begin{split}
					&|T_\epsilon(x)-T_\epsilon(y)| \geq \left(1- \frac{2}{R}\max\{3nr_{x,\epsilon}+2|x'|,3nr_{y,\epsilon}+2|y'|\}\right)|x-y|\\
					\text{and}\quad&|T_\epsilon(x)-T_\epsilon(y)| \leq \left(1+ \frac{2}{R}\max\{3nr_{x,\epsilon}+2|x'|,3nr_{y,\epsilon}+2|y'|\}\right)|x-y|, 
				\end{split}
			\end{equation*}
			from which~\eqref{eq::perturb_prop2} follows immediately.
			
			In order to prove~\eqref{eq::perturb_inclusions}, recall that, by construction,
			$$ A_\epsilon^- \subseteq V_{R,\epsilon}\setminus B_{2R}(-2Re_n).$$
			In other words, for every~$x\in A_\epsilon^-$, we have that
			\begin{equation*}
				|x+Re_n| \leq R+\epsilon d\left(\frac{x}{\epsilon}\right) \qquad\text{and} \qquad |x+2Re_n|\geq2R.
			\end{equation*}
			These inequalities yield that 
			\begin{equation*}
				|x'|^2+x_n^2+2Rx_n\leq0\qquad{\mbox{and}}\qquad
				|x'|^2+x_n^2+4Rx_n\geq0 .
			\end{equation*}
			Hence, it must be~$d\left(\frac{x}{\epsilon}\right)\neq0$, therefore~$|x'|<\epsilon$. 
			
			Furthermore, the parabola~$-\frac{|x'|^2}{R}$ is contained in~$B_{2R}(-2Re_n)$ when~$|x'|<\epsilon$. Thus, recalling also that~$\epsilon<R$, we infer that 
			$$ -\epsilon \leq -\frac{|x'|^2}{R} \leq x_n < \epsilon.$$
			Therefore, we deduce that~$|x|<2\epsilon$, that is~$ A_\epsilon^-\subseteq B_{2\epsilon}$.
			
			Now, using the assumption~$\epsilon<1/(3n)$ with~\eqref{eq::DT_estimate1}, we obtain that
			$$ \| DT_\epsilon(x)\| \leq 1+2\epsilon\frac{3n+1}{R}\leq 4.$$
			This implies that~$T_\epsilon(A_\epsilon^-)\subseteq 4B_{2\epsilon}$, and consequently~$A_\epsilon\subseteq B_{8\epsilon}$.
			
			To conclude, if~$|x|\leq\frac{\epsilon^2}{2R}$, then~$\mbox{d}_\epsilon(x)\geq \frac{3\epsilon^2}{4R}$, from which we find that~$|x+Re_n|\leq R+\mbox{d}_\epsilon(x)$. 	
			It thereby follows that~$B_{\epsilon^2/(2R)}\subseteq V_{R,\epsilon}$, hence
			\begin{equation*}
				B_{\epsilon^2/(2R)}\setminus E\subseteq A_\epsilon^-\subseteq A_\epsilon.\qedhere
			\end{equation*}
		\end{proof}
		
		\section{Proof of Formula~\eqref{eq::r_x_epsilon_lower_bound}} \label{sec::r_x_epsilon_lower_bound}
		
		In this section, we show that
		\begin{equation*}
			\dist(\mathbb{E}_{R,\epsilon},\mathbb{E}_{R,r}) \geq \frac{1}{1+c}\frac{\epsilon^2-r^2}{R},
		\end{equation*}
		where~$c>0$ is a constant depending only on~$n$, and
		$$\mathbb{E}_{R,r}:=\left\{x=(x',x_n)\;{\mbox{ s.t. }}\; |x+Re_n|= R+\frac{r^2}{R}\left(1-\frac{|x'|^2}{r^2}\right)\right\}\setminus B_{2R}(-2Re_n),$$ 
		as claimed in~\eqref{eq::r_x_epsilon_lower_bound}.
		
		To this aim, we define the function~$\Phi:(0,+\infty)\times[0,+\infty)\to\R$ as
		\begin{equation*}
			\Phi(r,\rho):=-R+\left[\frac{\rho^4}{R^2}-\left(3+\frac{2r^2}{R^2}\right)\rho^2+\left(R+\frac{r^2}{R}\right)^2\right]^{1/2},
		\end{equation*}
		and we notice that~$\left(x',\Phi(r,|x'|)\right)\in\mathbb{E}_{R,r}$, for every~$x'\in\R^{n-1}$ such that~$|x'|<r$.
		
		Since~$\mathbb{E}_{R,\epsilon}$ is closed, for every~$x\in \mathbb{E}_{R,r}$ there exists~$y\in\mathbb{E}_{R,\epsilon}$ such that~$\dist(x,\mathbb{E}_{R,\epsilon})=|x-y|$. 
		Thanks to the triangle inequality, we obtain that
		\begin{equation*}
			\begin{split}
				|x-y|&\geq |\Phi(r,|x'|)-\Phi(\epsilon,|y'|)|\\& 
				\geq |\Phi(r,|x'|)-\Phi(\epsilon,|x'|)|-|\Phi(\epsilon,|x'|)-\Phi(\epsilon,|y'|)|\\
				& \geq |\Phi(r,|x'|)-\Phi(\epsilon,|x'|)|-\norm{\partial_\rho\Phi(\epsilon,\cdot)}_{L^\infty([0,+\infty))}\left||x'|-|y'|\right|\\ 
				& \geq |\Phi(r,|x'|)-\Phi(\epsilon,|x'|)|-\norm{\partial_\rho\Phi(\epsilon,\cdot)}_{L^\infty([0,+\infty))}|x-y|,
			\end{split}
		\end{equation*}
		which reads 
		\begin{equation} \label{eq::dist_lower_bound}
			|x-y|\geq \frac{1}{1+\norm{\partial_\rho\Phi(\epsilon,\cdot)}_{L^\infty([0,+\infty))}} |\Phi(r,|x'|)-\Phi(\epsilon,|x'|)|. 
		\end{equation}
		
		Now, notice that, for every~$0<r<R/\sqrt{2}$ and~$0\leq \rho \leq r$, we have that
		\begin{equation} \label{eq::phi_plus_R_lower_bound}
			\begin{split}
				&\Phi(r,\rho)+R = \left[\frac{\rho^4}{R^2}-\left(3+\frac{2r^2}{R^2}\right)\rho^2+\left(R+\frac{r^2}{R}\right)^2\right]^{1/2}\\
				&\qquad\geq \left(R^2+\frac{r^4}{R^2}+2r^2-3\rho^2-\frac{2r^2\rho^2}{R^2}\right)^{1/2}\geq\left(R^2-r^2-\frac{r^4}{R^2}\right)^{1/2}\\
				&\qquad\geq \left(R^2-\frac{R^2}{2}-\frac{R^2}{4}\right)^{1/2}=\frac{R}{2}.
			\end{split}
		\end{equation}
		It thereby follows that,
		for every~$0<\epsilon<R/\sqrt{2}$ and~$0\leq \rho \leq \epsilon$,
		\begin{equation} \label{eq::grad_sup_norm}
			|\partial_\rho\Phi(\epsilon,\rho)| = \left|\frac{\rho}{R^2(\Phi(\epsilon,\rho)+R)}(2\rho^2-3R^2-2\epsilon^2)\right| \leq \frac{R/\sqrt{2}}{R^3/2}4R^2=4\sqrt{2}.
		\end{equation}
		
		Now, let us define the function~$f_{r,\epsilon}:[0,+\infty)\to [0,+\infty)$ as
		$$ f_{r,\epsilon}(\rho) := |\Phi(r,\rho)-\Phi(\epsilon,\rho)|. $$ 
		Then,
		\begin{equation} \label{eq::f_r_eps_prime}
			f'_{r,\epsilon}(\rho)=\frac{\Phi(r,\rho)-\Phi(\epsilon,\rho)}{|\Phi(r,\rho)-\Phi(\epsilon,\rho)|}\big(\partial_\rho\Phi(r,\rho)-\partial_\rho\Phi(\epsilon,\rho)\big).
		\end{equation}
		
		We observe that 
		\begin{eqnarray*}	
			\partial_r \Phi(r,\rho) &=& \frac{2r}{R^2(\Phi(r,\rho)+R)}(R^2+r^2-\rho^2) \\ {\mbox{and }}\qquad
			\partial_r \partial_\rho\Phi(r,\rho) 
			&=&
			\frac{2r\rho}{R^2(\Phi(r,\rho)+R)^3}(R^2+r^2+\rho^2).
		\end{eqnarray*}
		Therefore, recalling also~\eqref{eq::phi_plus_R_lower_bound}, we infer that~$\Phi(\cdot,\rho)$ and~$\partial_\rho\Phi(\cdot,\rho)$ are increasing in the first variable. Namely, for every~$0<r<\epsilon<R/\sqrt{2}$ and~$0\leq \rho \leq r$.
		\begin{align}
			\label{eq::lemon_monotonicity}&\Phi(r,\rho)<\Phi(\epsilon,\rho), \\
			\label{eq::lemon_monotonicity_derivative}\mbox{and}\qquad &\partial_\rho\Phi(r,\rho)\leq\partial_\rho\Phi(\epsilon,\rho),
		\end{align}
		
		Inequalities~\eqref{eq::lemon_monotonicity} and~\eqref{eq::lemon_monotonicity_derivative}, together with~\eqref{eq::f_r_eps_prime}, give that~$f'_{r,\epsilon}$ is always non-negative. Hence, $f_{r,\epsilon}$ is increasing and has a minimum at~$\rho=0$. 
		It thereby follows that
		\begin{equation} \label{eq::psi_monotone_in_r}
			|\Phi(r,|x'|)-\Phi(\epsilon,|x'|)| \geq |\Phi(r,0)-\Phi(\epsilon,0)| = \frac{\epsilon^2-r^2}{R}.
		\end{equation}
		
		Therefore, plugging~\eqref{eq::grad_sup_norm} and~\eqref{eq::psi_monotone_in_r} in~\eqref{eq::dist_lower_bound}, we deduce that
		\begin{equation*}
			\dist(x,\mathbb{E}_{R,r}) = |x-y|\geq \frac{1}{1+4\sqrt{2}}\frac{\epsilon^2-r^2}{R}.
		\end{equation*}
		Hence, taking the minimum among all~$x\in\mathbb{E}_{R,\epsilon}$, we obtain
		\begin{equation*}
			\dist(\mathbb{E}_{R,\epsilon},\mathbb{E}_{R,r}) \geq \frac{1}{1+4\sqrt{2}}\frac{\epsilon^2-r^2}{R},
		\end{equation*}
		as desired.
		
		\section{Proof of Lemma~\ref{lemma::conv_unif}} \label{sec::lemma_conv_unif}
		
		We exploit Lemma~\ref{lemma::holder_est} with~$r:=1$ and~$C:=1$. In this way, up to a subsequence, $A_k^1$ converges uniformly to the graph of a H\"older-continuous function~$u$ in~$\{|x'|<1/2\}$ (see Remark~\ref{rem::unif_conv_holder}).
		
		Now, controlling the oscillation between the unit vectors~$\nu_l^k$ and~$\nu_{l+1}^k$, we repeat the same argument in larger and larger balls, and infer that~$\partial E_k^*\to\partial E^*$ uniformly in every compact subset of~$\R^n$. In order to explicit this idea, we recall that
		\begin{align*}
			&\partial E_k\cap B_{2^l}\subseteq\big\{|x\cdot\nu_l^k|\leq a_k2^{l(1+\alpha)}\big\}\\ {\mbox{and }}\quad &
			\partial E_k\cap B_{2^{l+1}}\subseteq\big\{|x\cdot\nu_{l+1}^k|\leq a_k2^{(l+1)(1+\alpha)}\big\}.
		\end{align*}
		{F}rom these inclusions we obtain that
		$$ |\nu_l^k-\nu_{l+1}^k| \leq a_k2^{l(1+\alpha)}+a_k2^{(l+1)(1+\alpha)}\leq 5\,2^{l(1+\alpha)-\alpha k}.$$
		The geometric idea behind this implication is quite simple, however, in order to not interrupt the flow of the proof, we refer to Appendix~\ref{sec::geom_idea} for an exhaustive explanation of the argument.
		
		Notice also that, for every~$l$,
		$$ \lim_{k\to+\infty}|\nu_l^k-e_n|\leq \lim_{k\to+\infty}|\nu_l^k-\nu_{l-1}^k|+\dots+|\nu_1^k-\nu_0^k|\leq  \lim_{k\to+\infty}5\,2^{-\alpha k}\sum_{m=0}^{l-1}2^{m(1+\alpha)}= 0.$$
		Therefore, for every~$x\in\partial E_k\cap B_{2^l}$,
		\begin{align*}
			&	|x_n| \leq |x\cdot\nu_l^k|+2^l|e_n-\nu_l^k| \leq a_k2^{l(1+\alpha)}+5 a_k2^l\sum_{m=0}^{l-1}2^{m\alpha} \\
			&\qquad\quad\leq 5 a_k2^l\sum_{m=0}^{l}2^{m\alpha} \leq 5 a_k2^l\frac{2^{\alpha(l+1)}-1}{2^\alpha-1}\leq c a_k 2^{l(1+\alpha)},
		\end{align*}
		so that we have the inclusion
		$$ \partial E_k\cap B_{2^l} \subseteq \big\{|x_n|\leq c a_k 2^{l(1+\alpha)}\big\}.$$
		
		Now, we apply Lemma~\ref{lemma::holder_est} with~$r:=2$ and~$C:=c2^{1+\alpha}$, and deduce that 
		$A_k^2$ is uniformly convergent to the graph of a H\"older-continuous function~$v$ in~$\{|x'|<1\}$. By uniqueness, we have that~$v$ must coincide with~$u$ in~$\{|x'|<1/2\}$.
		
		We apply iteratively Lemma~\ref{lemma::holder_est} to~$A_k^j$ and, using a diagonal argument, we infer that~$A_k^k$ uniformly converges to the graph of a H\"older-continuous function~$u:\R^{n-1}\to\R$, along a subsequence. Hence~$E_k^*$ converges uniformly to~$\subgr(u)$ in every compact set of~$\R^n$.
		
		Moreover, the fact that~$0\in\partial E_k^*$ for every~$k$ yields that~$0\in\partial E^*$. Therefore, since
		the point~$(0',u(0'))\in\partial E^*$, it follows that~$u(0')=0$.
		
		Finally, we exploit the uniform convergence of~$\{E_k^*\}_k$ to conclude that
		$$ |u(x')| \leq \widetilde{c}\big(1+|x'|^{1+\alpha}\big),$$
		for some~$\widetilde{c}>0$.
		Indeed, for every given~$l$, taking the limit as~$k\to+\infty$, we have that~$\partial E_k^*\cap B^{n-1}_{2^l}
		\subseteq \{|x_n|\leq c 2^{l(1+\alpha)}\}$ and~$\partial E_k^*\cap B^{n-1}_{2^l}$ converges to~$\partial E^*\cap B^{n-1}_{2^l}$.
		
		As a consequence, since for every~$x'\in\R^{n-1}$ there exists~$l\in\N$ such that~$x'\in B^{n-1}_{2^l}\setminus B^{n-1}_{2^{l-1}}$, we are led to
		$$ \frac{|u(x')|}{(1+|x'|^{1+\alpha})} \leq \frac{c 2^{l(1+\alpha)}}{1+2^{(l-1)(1+\alpha)}}\leq c2^{1+\alpha} =: \widetilde{c}.$$
		It follows that 
		$$ |u(x')|\leq \widetilde{c}\big(1+|x'|^{1+\alpha}\big),$$
		which concludes the proof.			
		
		\section{Proof of Lemma~\ref{lemma::linear_viscosity}} \label{sec::lemma_linear_viscosity}
		
		We prove that~$u$ is a viscosity solution of 
		\begin{equation} \label{eq::viscosity_harmonic}
			(-\Delta)^{\frac{1+s}{2}} u=0
		\end{equation}
		in the sense of~\cite[Definition~2.2]{MR2494809}. more precisely, we show that~$u$ is both viscosity super- and sub-solution of~\eqref{eq::viscosity_harmonic}.
		
		To this end, let~$\psi$ be a function touching~$u$ from below at some point~$x_0$. Without loss of generality, we can assume that~$x_0$ coincides with the origin. Moreover assume that~$\psi\in\cont^2(\overline{B}_\eta)$, for some~$\eta>0$, and~$\psi<u$ in~$B_\eta\setminus\{0\}$. 
		
		We point out that, in order to prove that~$u$ is a viscosity super-solution of~\eqref{eq::viscosity_harmonic}, it is sufficient to show that
		\begin{equation}\label{eq::lin_visc_claim}
			\int_{\R^{n-1}} \frac{u(x')-u(0')}{|x'|^{n+s}}\,d\haus{n-1}(x) \leq 0.
		\end{equation}
		Indeed, by definition, we have that
		\begin{equation*}
			\begin{split}
				&	 \int_{\R^{n-1}}  \frac{u(x')-u(0')}{|x'|^{n+s}}\,d\haus{n-1}(x) 
				\\&\qquad\geq \int_{B_\eta} \frac{\psi(x')-\psi(0')}{|x'|^{n+s}}\,d\haus{n-1}(x)  
				+ \int_{\R^{n-1}\setminus B_\eta} \frac{u(x')-u(0')}{|x'|^{n+s}}\,d\haus{n-1}(x) .
			\end{split}
		\end{equation*}
		Thus, we focus on proving~\eqref{eq::lin_visc_claim}.
		
		Thanks to Lemma~\ref{lemma::conv_unif}, for any given~$\epsilon>0$ and~$R>0$, by definition of uniform convergence, there exists~$k_\epsilon\in\N$ such that, for every~$k\geq k_\epsilon$, $\partial E_k\cap B_R$ lies in a~$a_k\epsilon$-neighborhood of~$\graph(a_ku)$, where~$a_k:=2^{-\alpha k}$. 
		
		Furthermore, thanks to the flatness hypothesis on~$E_k$, there exists a vertical translation of a parabola of opening~$-\frac{a_k}{2}$ that touches~$\partial E_k$ in~$B_\epsilon$ at some point~$x_k$.
		Clearly, $x_k$ converges to the origin as~$\epsilon\to0$. Besides, notice that~$k_\epsilon\to+\infty$ as~$\epsilon\to0$.
		
		Now, for every~$\overline{x}=(\overline{x}',\overline{x}_n)\in\R^{n-1}\times\R$ and for every~$\delta>0$, we define the truncated closed cylinder
		$$ D_r(\overline{x}):=\big\{x\in\R^n\;{\mbox{ s.t. }}\;|x'-\overline{x}'|\leq r,\; |x_n-\overline{x}_n|\leq r\}.$$
		
		By the flatness assumptions in~\eqref{eq::flatness_Ek}, we know that
		\begin{equation*}
			\partial E_k\cap D_R(x_k)\subseteq\partial E_k\cap B_{2R}(x_k)\subseteq\big\{|x_n-x_{k,n}|\leq c_R a_k\big\},
		\end{equation*} 
		where~$c_R$ is a positive constant depending on~$R$.
		Up to taking~$k$ large enough, we assume that~$c_Ra_k<\delta/2$. Then, for every~$x\in \big(D_R(x_k)\setminus D_\delta(x_k)\big) \cap \{|x_n-x_{k,n}|\leq c_R a_k\}$,
		\begin{equation*}
			|x'-x_k'| \geq \delta^2 - |x_n-x_{k,n}|^2 \geq \delta^2 - c_R^2a_k^2 \geq \frac{3}{4}\delta^2.
		\end{equation*}
		
		Moreover, we observe that
		\begin{equation*}
			\begin{split}
				\left|\frac{1}{|x-x_k|^{n+s}}-\frac{1}{|x'-x_k'|^{n+s}}\right| 
				&= \frac{n+s}{2} \int_{0}^{1} \frac{|x_n-x_{k,n}|^2}{(|x'-x_k'|^2+t|x_n-x_{k,n}|^2)^{\frac{n+s+2}{2}}} \,dt \\
				&\leq \frac{n+s}{2} c_R^2a_k^2 \int_{0}^{1} \frac{dt}{|x'-x_k'|^{n+s+2}} .
			\end{split}
		\end{equation*}
		Thus, from the last two estimates, we infer that
		\begin{equation*}
			\left|\frac{1}{|x-x_k|^{n+s}}-\frac{1}{|x'-x_k'|^{n+s}}\right|  \leq Ca_k^2 ,
		\end{equation*}
		for some positive constant~$C$, depending on~$n$, $s$, $R$, and~$\delta$.
		
		Using this and the fact that~$\partial E_k$ lies in a~$a_k\epsilon$-neighborhood of~$a_k\graph(u)$, we arrive at 
		\begin{equation}\label{eq::lin_visc1}
			\begin{split}
				&\int_{D_R(x_k)\setminus D_\delta(x_k)} \frac{\chi_{E_k}(x)-\chi_{E_k^c}(x)}{|x-x_k|^{n+s}}\,dx \\
				&\qquad= 2a_k\int_{B_R^{n-1}(x_k)\setminus B_\delta^{n-1}(x_k)} \frac{u(x')-u(x_k')+O(\epsilon)}{|x'-x_k'|^{n+s}}\,d\haus{n-1}(x) + O(a_k^3)\\
				&\qquad= 2a_k\int_{B_R^{n-1}(x_k)\setminus B_\delta^{n-1}(x_k)} \frac{u(x')-u(x_k')}{|x'-x_k'|^{n+s}} \,d\haus{n-1}(x) + O(\epsilon a_k) +O(a_k^3).
			\end{split}
		\end{equation}
		We now write the set~$D_R(x_k)\setminus D_\delta(x_k)$ as~$B_\delta^c(x_k)\setminus\left[D_R(x_k)^c\cup \left(D_\delta(x_k)\cap B_\delta^c(x_k)\right)\right]$ and prove suitable estimates for the left-hand side in~\eqref{eq::lin_visc1}.
		
		To this end, let us denote by~$P_k$ the subgraph of the tangent parabola~$x_{k,n}-\frac{a_k}{2}|x'-x_k'|^2$. Since, $P_k$ is smooth, the non-local mean curvature of~$P_k$ is well-defined at~$x_k$, for every~$k$. 
		Then, we have
		\begin{equation} \label{eq::lin_visc_parabola}
			\begin{split}
				&\liminf_{\rho\to0}\int_{D_\delta(x_k)\setminus B_\rho(x_k)} \frac{\chi_{E_k}(x)-\chi_{E_k^c}(x)}{|x-x_k|^{n+s}}\,dx \\ 
				&\quad  \geq  \liminf_{\rho\to0} \int_{D_\delta(x_k)\setminus B_\rho(x_k)} \frac{\chi_{P_k}(x)-\chi_{P_k^c}(x)}{|x-x_k|^{n+s}}\,dx  = \int_{D_\delta(x_k)} \frac{\chi_{P_k}(x)-\chi_{P_k^c}(x)}{|x-x_k|^{n+s}}\,dx .
			\end{split}
		\end{equation}
		Arguing as in~\eqref{eq::harnack_proof7}, we thereby obtain that
		\begin{equation}\label{eq::lin_visc2}
			\liminf_{\rho\to0}\int_{D_\delta(x_k)\setminus B_\rho(x_k)} \frac{\chi_{E_k}(x)-\chi_{E_k^c}(x)}{|x-x_k|^{n+s}}\,dx 
			\geq -\int_{0}^{\delta}  \frac{Ca_kr^n}{r^{n+s}} dr 
			= -Ca_k\delta^{1-s} ,
		\end{equation}
		for a positive constant~$C$ independent of~$k$.
		
		For the term involving~$D_R^c(x_k)$, recalling the estimates in~\eqref{eq::harnack_proof3}, we have
		\begin{equation}\label{eq::lin_visc3}
			\begin{split}
				\left|\int_{D_R^c(x_k)} \frac{\chi_{E_k}(x)-\chi_{E_k^c}(x)}{|x-x_k|^{n+s}}\,dx\right|
				&\leq C'a_k \int_{R}^{+\infty} \rho^{\alpha-1-s}\,d\rho + \frac{\omega_n}{s}2^{-ks}\\
				&\leq C'\big(a_kR^{\alpha-s}+ a_k^{1+\xi}\big),
			\end{split}
		\end{equation}
		for some positive exponent~$\xi$, and some positive constant~$C'$,
		depending only on~$n$, $s$, and~$\alpha$, and possibly changing from line to line.
		
		Now, let us denote by~$F^k_{x_k}(x):=\frac{\chi_{E_k}(x)-\chi_{E_k^c}(x)}{|x-x_k|^{n+s}}$ for simplicity, and observe that 
		\begin{equation*}
			\begin{split}
				&\int_{D_R(x_k)\setminus D_\delta(x_k)} F^k_{x_k}(x)\,dx \\
				&= \limsup_{\rho\to0} \left[ \int_{B_\rho^c(x_k)} F^k_{x_k}(x)\,dx - \int_{D_\delta(x_k)\setminus B_\rho(x_k)} F^k_{x_k}(x)\,dx \right] 
				- \int_{D_R^c(x_k)} F^k_{x_k}(x)\,dx\\
				& \leq \limsup_{\rho\to0} \left[ \int_{B_\rho^c(x_k)} F^k_{x_k}(x)\,dx \right] - \liminf_{\rho\to0} \left[ \int_{D_\delta(x_k)\setminus B_\rho(x_k)} F^k_{x_k}(x)\,dx \right] + \left|\int_{D_R^c(x_k)} F^k_{x_k}(x)\,dx \right|.
			\end{split}
		\end{equation*}
		
		Therefore, using this, \eqref{eq::lin_visc2} and~\eqref{eq::lin_visc3} in~\eqref{eq::lin_visc1}, we obtain that
		\begin{equation} \label{eq::lin_visc_xk_estimate}
			\begin{split}
				&\int_{B_R^{n-1}(x_k)\setminus B_\delta^{n-1}(x_k)} \frac{u(x')-u(x_k')}{|x'-x_k'|^{n+s}}\,d\haus{n-1}(x) \\
				& = \frac{1}{2a_k} \int_{D_R(x_k)\setminus D_\delta(x_k)} \frac{\chi_{E_k}(x)-\chi_{E_k^c}(x)}{|x-x_k|^{n+s}}\,dx +O(\epsilon) +O(a_k^2)\\
				& \leq \limsup_{\rho\to0}
				\frac{1}{2a_k}\int_{B_\rho^c(x_k)} \frac{\chi_{E_k}(x)-\chi_{E_k^c}(x)}{|x-x_k|^{n+s}}\,dx 
				+C\delta^{1-s}+ O(\epsilon) +O(a_k^{\xi})+O(R^{\alpha-s}).
			\end{split}
		\end{equation}
		
		Thanks to the assumptions on~$E_k$ and Theorem~\ref{th::ELeq}, we know that 
		\begin{equation*}
			\limsup_{\rho\to0}
			\frac{1}{2a_k}\int_{B_\rho^c(x_k)} \frac{\chi_{E_k}(x)-\chi_{E_k^c}(x)}{|x-x_k|^{n+s}}\,dx \leq 2^{-k(s-\alpha)-1}\Lambda.
		\end{equation*}
		Thus,
		\begin{equation}\label{eq::lin_visc_xk_estimate_2}
			\begin{split}
				&\int_{B_R^{n-1}(x_k)\setminus B_\delta^{n-1}(x_k)} \frac{u(x')-u(x_k')}{|x'-x_k'|^{n+s}}\,d\haus{n-1}(x) \\
				&\qquad\qquad\leq 2^{-k(s-\alpha)-1}\Lambda + C\delta^{1-s}+ O(\epsilon) +O(a_k^{\xi})+O(R^{\alpha-s}).
			\end{split}
		\end{equation}
		Since~$\alpha\in(0,s)$, taking the limit as~$\epsilon\to0$, and hence as~$k\to+\infty$, in~\eqref{eq::lin_visc_xk_estimate_2}, we infer that
		\begin{equation*}
			\int_{B_R^{n-1}\setminus B_\delta^{n-1}} \frac{u(x')-u(0')}{|x'|^{n+s}}\,d\haus{n-1}(x) \leq C\delta^{1-s}+O(R^{\alpha-s}).
		\end{equation*}
		Hence, taking the limits as~$\delta\to0$ and~$R\to+\infty$, we obtain~\eqref{eq::lin_visc_claim}.
		\smallskip
		
		Now, to show that~$u$ is a viscosity sub-solution of~\eqref{eq::viscosity_harmonic}, let~$\phi$ be a~$\cont^2$-function touching~$u$ at~$0$ from above. Using the same idea in~\eqref{eq::lin_visc_parabola} with a tangent parabola of opening~$\frac{a_k}{2}$, we deduce that
		\begin{equation*}
			\limsup_{\rho\to0}\int_{D_\delta(x_k)\setminus B_\rho(x_k)} \frac{\chi_{E_k}(x)-\chi_{E_k^c}(x)}{|x-x_k|^{n+s}}\,dx 
			\leq Ca_k\delta^{1-s} .
		\end{equation*}
		This, together with~\eqref{eq::lin_visc1}, \eqref{eq::lin_visc3}, and Corollary~\ref{cor::reverseELeq}, leads us to
		\begin{equation} \label{eq::lin_visc4}
			\begin{split}
				&\int_{B_R^{n-1}(x_k)\setminus B_\delta^{n-1}(x_k)} \frac{u(x')-u(x_k')}{|x'-x_k'|^{n+s}}\,d\haus{n-1}(x) \\
				& \geq \liminf_{\rho\to0} \frac{1}{2a_k}\int_{B_\rho^c(x_k)} \frac{\chi_{E_k}(x)-\chi_{E_k^c}(x)}{|x-x_k|^{n+s}}\,dx 
				-C\delta^{1-s}+ O(\epsilon) +O(a_k^{\xi})+O(R^{\alpha-s})\\
				&\geq -2^{-k(s-\alpha)-1}\Lambda - C\delta^{1-s}+ O(\epsilon) +O(a_k^{\xi})+O(R^{\alpha-s}).
			\end{split}
		\end{equation}
		Taking~$\epsilon\to0$, and consequently~$k\to+\infty$, in~\eqref{eq::lin_visc4}, we obtain that
		\begin{equation*}
			\int_{B_R^{n-1}\setminus B_\delta^{n-1}} \frac{u(x')-u(0')}{|x'|^{n+s}}\,d\haus{n-1}(x) \geq -C\delta^{1-s}+O(R^{\alpha-s}).
		\end{equation*}
		Finally, taking the limits as~$\delta\to0$ and~$R\to+\infty$, we infer that~$u$ is also a viscosity sub-solution of~\eqref{eq::viscosity_harmonic}, concluding the proof.
		
		\section{Geometric estimates for unit normals} \label{sec::geom_idea}
		In this section, we estimate the angle between the unit vectors~$\nu_j$ and~$\nu_{j+1}$ given the flatness assumptions
		\begin{equation}\label{eq::geom_flatness}
			\partial E \cap B_{2^j} \subseteq \big\{|x\cdot\nu_j|<a2^{j(1+\alpha)}\big\},\qquad{\mbox{for all }}j=0,\dots,k ,
		\end{equation}
		where~$E$ is a~$\Lambda$-minimal set in~$B_{2^k}$ such that~$0\in\partial E$, and~$a:=2^{-\alpha k}$, for some~$\alpha\in(0,1)$
		(see Figure~\ref{fig::geom1}).
		
		We claim the following:
		
		\begin{lemma}\label{lemma::geom_estimate}
			For every~$j\in\{0,\dots,k-1\}$, we have that
			$$ |\nu_j-\nu_{j+1}| \leq c2^{\alpha(j-k)},$$
			for some constant~$c>0$ independent of~$j$ and~$k$.
		\end{lemma}
		
		This fact follows straight-forwardly from the geometric properties of our problem, and to prove it, it is sufficient to estimate the worst case scenario, which occurs when the slab~$\{|x\cdot\nu_j|<a2^{j(1+\alpha)}\}$ touches~$\{|x\cdot\nu_{j+1}|<a2^{(j+1)(1+\alpha)}\}$. We now provide the technical details of this fact.
		
		\begin{proof}[Proof of Lemma~\ref{lemma::geom_estimate}]
			Since we are interested in estimating the angle between unit normals of hypersurfaces, we narrow down to the~$2$-dimensional case. Then, to our purpose, we consider~$j<k$, and observe that the flatness assumptions in~\eqref{eq::geom_flatness} yield that
			$$ \big\{x\cdot\nu_{j+1}<-a2^{(j+1)(1+\alpha)}\big\} \subseteq \big\{x\cdot\nu_j<a2^{j(1+\alpha)}\big\} \text{ in }B_{2^j}.$$
			Therefore, rescaling by a factor~$2^{-j}$, we deduce that
			\begin{equation*}\label{eq::geom_1}
				\big\{x\cdot\nu_{j+1}<-a2^{\alpha j+\alpha+1}\big\} \subseteq
				\big\{x\cdot\nu_j<a2^{\alpha j}\big\} \text{ in }B_{1}.
			\end{equation*}
			Let~$\delta_j := a2^{\alpha j} = 2^{\alpha(j-k)}<1$. Then, we rewrite the last display as
			\begin{equation*}
				\big\{x\cdot\nu_{j+1}<-2^{\alpha+1}\delta_j\big\} 
				\subseteq \big\{x\cdot\nu_j<\delta_j\big\} \text{ in }B_{1},
			\end{equation*} see Figure~\ref{fig::geom2}.
			
			{F}rom this it follows that the maximum angle between~$\nu_j$ and~$\nu_{j+1}$ is realized when~$\{x\cdot\nu_{j+1}<-2^{\alpha+1}\delta_j\} \cap \{x\cdot\nu_j<\delta_j\}\neq\varnothing$. This case is represented in Figure~\ref{fig::geom3}. 
			
			Notice that, by trivial geometric properties, we have that
			$$ \cos\eta = \nu_{j+1}\cdot \nu_j,$$
			where~$\eta$ is the angle between~$\{x\cdot\nu_j=\delta_j\}$ and~$\{x\cdot\nu_{j+1}=-2^{1+\alpha}\delta_j\}$, as shown in Figure~\ref{fig::geom4}. Since~$\nu_{j+1}\cdot \nu_j$ defines an angle at the center, while~$\eta$ is an angle at the circumference, we have the estimate 
			$$|\nu_{j+1}-\nu_j|\leq L,$$
			where~$L$ is the length of the chord between~$\{x\cdot\nu_j=\delta_j\}$ and~$\{x\cdot\nu_{j+1}=-2^{1+\alpha}\delta_j\}$ in~$B_1$. 
			
			We consider the points
			\begin{align*}
				p&\in\{\lambda\nu_j: \lambda\in\R\}\cap\{x\cdot\nu_j=\delta_j\}\\
				{\mbox{and }}\quad q&\in\{\lambda\nu_{j+1}: \lambda\in\R\}\cap\{x\cdot\nu_{j+1}=-2^{1+\alpha}\delta_j\},
			\end{align*}
			and let~$d$ be the segment between~$p$ and~$q$ (see Figure~\ref{fig::geom4}). By a similarity property, we have that~$L=2d$ and hence
			\begin{equation}\label{eq::geom_3}
				|\nu_{j+1}-\nu_j| \leq 2 d.
			\end{equation}
			
			Also, using basic trigonometry, we see that 
			\begin{equation*}
				d \leq \dist\big(0,\{x\cdot\nu_j=\delta_j\}\big) + \dist\big(0,\{x\cdot\nu_{j+1}=-2^{1+\alpha}\delta_j\}\big)
				=\delta_j + 2^{1+\alpha}\delta_j \leq 5\delta_j.
			\end{equation*}
			{F}rom this and~\eqref{eq::geom_3}, we infer that
			$$ |\nu_{j+1}-\nu_j|\leq 10\delta_j = 5\,2^{\alpha(j-k)+1},$$
			as desired. 
		\end{proof}
		
		\begin{figure}[!h]
			\begin{minipage}[t]{0.48\textwidth}
				\centering
				\includegraphics[width=.89\linewidth]{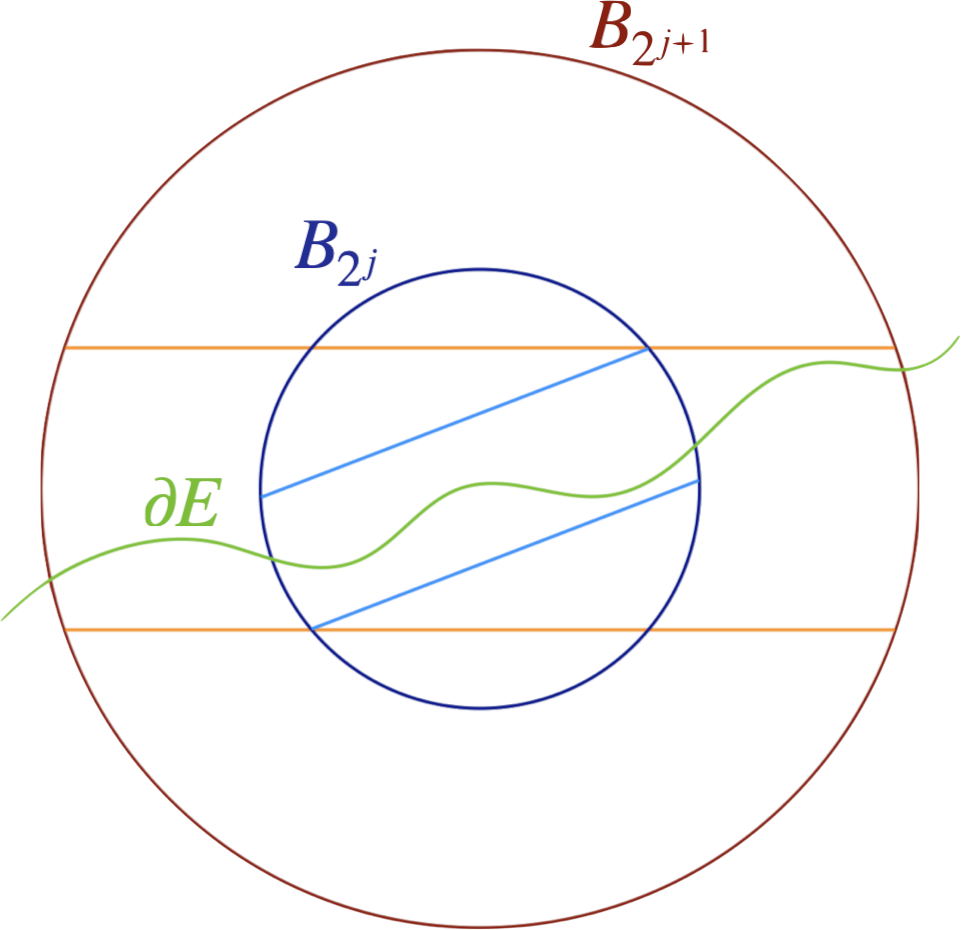}
				\caption{Initial situation with flatness assumptions on~$\partial E$.}\label{fig::geom1}
			\end{minipage}\hfill
			\begin{minipage}[t]{0.48\textwidth}
				\centering
				\includegraphics[width=.8\linewidth]{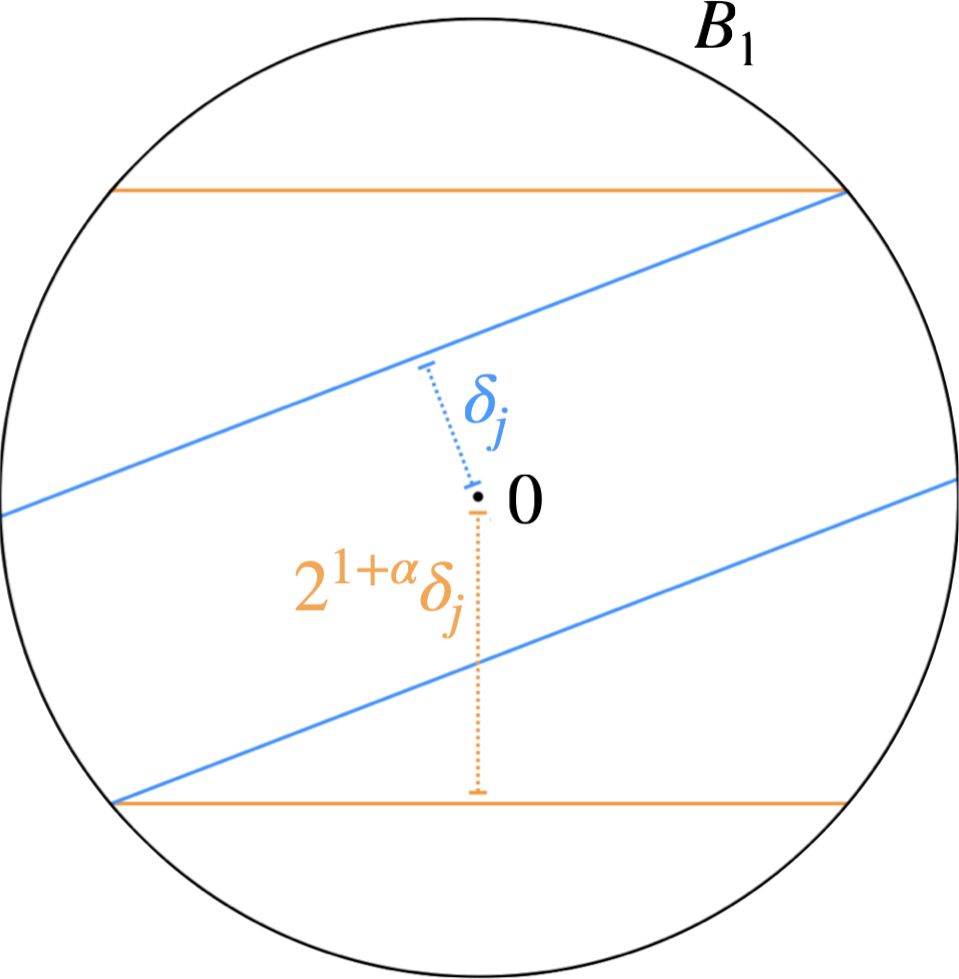}
				\caption{The rescaled setting.}\label{fig::geom2}
			\end{minipage}
		\end{figure}
		\begin{figure}[!h]
			\begin{minipage}[t]{0.48\textwidth}
				\centering
				\includegraphics[width=.8\linewidth]{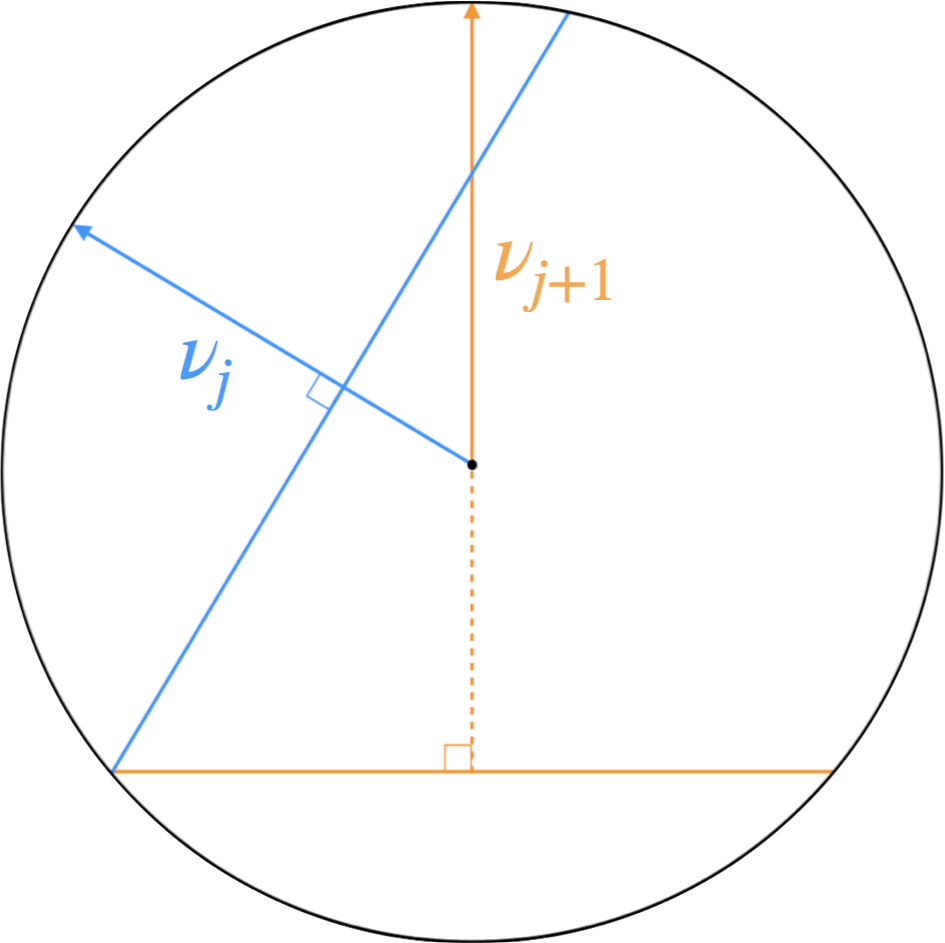}
				\caption{A case that is worse than the worst case scenario.}\label{fig::geom3}
			\end{minipage}\hfill
			\begin{minipage}[t]{0.48\textwidth}
				\centering
				\includegraphics[width=.8\linewidth]{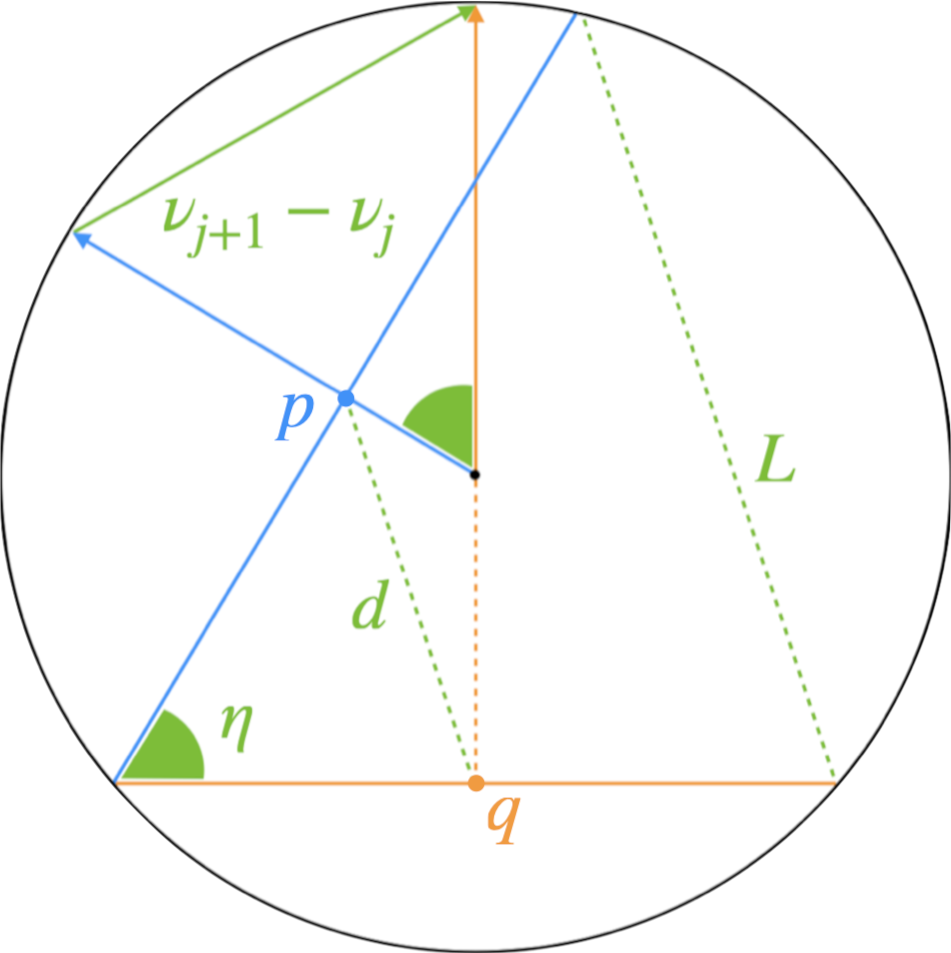}
				\caption{The geometric construction.}\label{fig::geom4}
			\end{minipage}
		\end{figure}
		
		\section{Proof of Proposition~\ref{prop::phi_continuous_in_E}} \label{sec::prop_phi_continuous_in_E}
		
		Notice that, by construction, the sequence~$\{\widetilde{u}_k\}_k$ is uniformly Lipschitz continuous on every compact set of~$\{z>0\}$. Therefore, thanks to
		the Ascoli-Arzel\`a's Theorem, there exists a subsequence~$k_j$ such that~$\widetilde{u}_{k_j}$ converges to some function~$\widetilde{v}$ uniformly on every compact set of~$\R^{n+1}_+$.
		
		Moreover, by uniform convergence, 
		\begin{equation*}
			\divergence(z^{1-s}\nabla\widetilde{v})=0 \text{ in }\R^{n+1}_+.
		\end{equation*}
		By the uniqueness of the solution of the Dirichlet problem, to prove that~$\widetilde{v}=\widetilde{u}$ (and thus obtain the claim in~\eqref{item::phi_continuous_in_E_1}), since both~$\widetilde{u}$ and~$\widetilde{v}$ are bounded, it is sufficient to show that
		\begin{equation}\label{rtyeuigfhdbvcnxvbcty4ui343}
			\trace(\widetilde{v})=u.\end{equation} 
		
		To this purpose, let~$r>0$. By Fatou's Lemma and Theorem~\ref{th::Phi_bound} (recall in particular formulas~\eqref{eq::def_Xi} and~\eqref{56789rtyuifghjxcvb}), we have that
		\begin{equation*}
			\int_{B_r^+} z^{1-s}|\nabla\widetilde{v}|^2\,dx\,dz \leq \liminf_{j\to+\infty} \int_{B_r^+} z^{1-s}|\nabla\widetilde{u}_{k_j}|^2\,dx\,dz \leq Cr^{n-s}(1+r^{s}).
		\end{equation*}
		Therefore, using the H\"older Inequality, we see that, for every~$\delta>0$,
		\begin{equation*}\label{eq::phi_cont2}
			\begin{split}
				\int_{B_r^{+}\cap\{0<z<\delta\}} |\nabla(\widetilde{u}_{k_j}-\widetilde{v})| \,dx\,dz
				&=\int_{B_r^{+}\cap\{0<z<\delta\}} z^{(s-1)/2}z^{(1-s)/2}|\nabla(\widetilde{u}_{k_j}-\widetilde{v})| \,dx\,dz\\
				&\leq C\delta^{s/2}\left(\int_{B_r^+} z^{1-s}|\nabla(\widetilde{u}_{k_j}-\widetilde{v})|^2 \,dx\,dz\right)^{1/2}\\
				&\leq C\delta^{s/2},
			\end{split}
		\end{equation*}
		for some positive constant~$C$, depending only~$n$, $s$, and~$r$, and possibly changing from line to line.
		{F}rom this, it follows that
		$$ \int_{B_r^+} |\nabla(\widetilde{u}_{k_j}-\widetilde{v})| \,dx\,dz \leq C\delta^{s/2}+ \int_{B_r^+\cap\{z\geq\delta\}} |\nabla(\widetilde{u}_{k_j}-\widetilde{v})| \,dx\,dz.$$
		
		Now, since~$\nabla\widetilde{u}_{k_j}$ converges uniformly to~$\nabla\widetilde{v}$ on every compact set of~$\{z>0\}$, taking the limit as~$j\to+\infty$, we obtain that, for every~$\delta>0$,
		\begin{equation*}
			\limsup_{j\to+\infty}\int_{B_r^+} |\nabla(\widetilde{u}_{k_j}-\widetilde{v})| \,dx\,dz \leq C\delta^{s/2}.
		\end{equation*}
		This gives that~$\widetilde{u}_{k_j}\to\widetilde{v}$ in~$W^{1,1}(B_r^+)$. Hence, thanks to the Sobolev embeddings, we have that~$u_{k_j}=\trace(\widetilde{u}_{k_j})\to\trace(\widetilde{v})$ in~$L^1(B_r^+)$. 
		
		Since~$r$ is arbitrary, we infer that~$u_{k_j}\to \trace(\widetilde{v})$ in~$L^1_{\loc}$. By the uniqueness of the limit,
		we thus obtain the claim in~\eqref{rtyeuigfhdbvcnxvbcty4ui343}.
		
		In order to prove~\eqref{item::phi_continuous_in_E_2}, recall that, by~\cite[Proposition~7.1]{caffarelli_roquejoffre_savin_nonlocal}, we know that
		$$ \limsup_{k\to+\infty} \int_{B_r^+} z^{1-s}|\nabla(\widetilde{u}_k-\widetilde{u})|^2\,dx\,dz \leq c\limsup_{k\to+\infty} \mathrm{J}_{2r}(u_k-u),$$
		where the functional~$J_r$ has been defined in~\eqref{defjeir}.
		In light of this fact, it is sufficient to show that~$\mathrm{J}_{2r}(u_k-u)\to0$ as~$k\to+\infty$, for every~$r>0$. 
		
		To this end, consider a subsequence~$\{u_{k_j}\}_j$, and define
		\begin{equation*}
			\begin{split}
				&f_{k_j}(x,y) := \frac{u_{k_j}(x)-u_{k_j}(y)}{|x-y|^{(n+s)/2}}\chi_{B_{2r}}(x)\left(\chi_{B_{2r}}(y)+\sqrt{2}\chi_{B_{2r}^c}(y)\right),\; {\mbox{for all~$j\in\N$}},\\
				{\mbox{and }}& f(x,y):=\frac{u(x)-u(y)}{|x-y|^{(n+s)/2}}\chi_{B_{2r}}(x)\left(\chi_{B_{2r}}(y)+\sqrt{2}\chi_{B_{2r}^c}(y)\right).
			\end{split}
		\end{equation*}
		In this way,
		\begin{equation*}
			\begin{split}
				& \norm{f_{k_j}}_{L^2(\R^{2n})} = \mathrm{J}_{2r}(u_{k_j}) = 8\Per_s(E_{k_j},B_{2r})\\ 
				{\mbox{and }}& \norm{f}_{L^2(\R^{2n})} = \mathrm{J}_{2r}(u) = 8\Per_s(E,B_{2r}).
			\end{split}
		\end{equation*}
		Notice that, since~$u_{k_j}\to u$ in~$L^1_{\loc}$, we have that~$u_{k_j}\to u$ pointwise up to a subsequence, which yields that~$f_{k_j}\to f$ pointwise. 
		
		Moreover, thanks to Proposition~\ref{prop::conv_almost_min}, 
		$$ \norm{f}_{L^2(\R^{2n})} = 8\Per_s(E,B_{2r}) = \lim_{k\to+\infty} 8\Per_s(E_k,B_{2r})=\norm{f_{k_j}}_{L^2(\R^{2n})}.$$
		Since pointwise convergence together with the converge of the~$L^2$-norms returns the~$L^2$-convergence, we infer that 
		$$\lim_{k\to+\infty}\mathrm{J}_{2r}(u_{k_j}-u) =\lim_{k\to+\infty} \norm{f_{k_j}-f}_{L^2(\R^{2n})}=0.$$
		This completes the proof of~\eqref{item::phi_continuous_in_E_2}.
		
		Finally, we observe that the~$L^2$-convergence with respect to the measure~$z^{1-s}dxdz$ of~$\nabla\widetilde{u}_k$ to~$\nabla\widetilde{u}$ entails that~$\Phi_{E_k}(r)\to\Phi_E(r)$ as~$k\to+\infty$.
		
		\section{Proof of Theorem~\ref{th::energy_gap}} \label{sec::energy_gap}
		
		Thanks to the clean ball condition in Corollary~\ref{cor::clean_ball}, there exists a small ball~$B\subseteq C\cap B_1$. Sliding vertically~$B$ until it touches~$\partial C$ at some point~$x_0$, we find an interior tangent ball to~$\partial C$ at~$x_0$. 
		
		As a consequence, $\partial C$ is an~$s$-minimal surface with an interior tangent ball at~$x_0$, hence it is~$\cont^{1,\alpha}$-regular in a neighborhood of~$x_0$ (see~\cite[Corollary~6.2]{caffarelli_roquejoffre_savin_nonlocal}). 
		
		Thus, by definition, the tangent cone to~$C$ at~$x_0$ is a half-space, so
		$$ \lim_{r\to0}\Xi_{C-x_0}(r) = \Xi_\Pi (0).$$
		
		Since~$C$ is~$0$-homogeneous, we have that
		$$ C^k:=\frac{1}{k}(C-x_0) = C-\frac{1}{k}x_0.$$
		By construction,we have that~$ C^k\to C$ in~$L^1_{\loc} $.
		Therefore, taking the limit as~$k\to+\infty$ in the monotonicity formula~\eqref{eq::monotinicty_sminimal}, we obtain that
		$$ \lim_{k\to+\infty}\Xi_{C-x_0}(kr) =\lim_{k\to+\infty} \Xi_{C^k}(r)= \Xi_C(r).$$
		Also, since~$\Xi_{C-x_0}$ is monotone, we have that~$\Xi_\Pi(r)=\Xi_{C-x_0}(0)\leq \Xi_C(r)$, thus proving~\eqref{eq::PileqC}.
		
		If equality in~\eqref{eq::PileqC} holds true, then~$\Xi_{C-x_0}$ is constant, hence~$C-x_0$ is a cone regular at the
		origin. {F}rom this, it follows that~$C-x_0$ is a half-space, hence~$C$ must be a half-space as well. 
		
		To prove~\eqref{eq::gap}, arguing by contradiction suppose that there exists a sequence of minimal cones~$C_k$ such that
		$$ \Xi_{C_k}\leq \Xi_\Pi + \frac{1}{k},$$
		but~$C_k$ is not a half-space, for any~$k$. 
		
		Then, thanks to Theorem~\ref{th::blowup_existence}, we have that, up to a subsequence, $ C_k\to C_\infty $ in~$L^1_{\loc}$,
		for some cone~$C_\infty$. 
		Therefore, $\Xi_{C_\infty}=\Xi_\Pi$, hence~$C_\infty$ must be a hyperplane according to the first part of the theorem. 
		
		In addition, by Corollary~\ref{cor::boundary_conv}, we have that~$\partial C_k\cap B_1$ lies in a small neighborhood of~$\partial C_\infty$ as long as~$k$ is large enough. Thus, up to rotations,
		$$ \partial C_k\cap B_1 \subseteq \{|x_n|<\epsilon_0\}.$$
		Then, from~\cite[Theorem~6.1]{caffarelli_roquejoffre_savin_nonlocal} (which is the analog of Theorem~\ref{th::holder_reg_almost_min} for~$s$-minimal surfaces), we deduce that~$\partial C_k$ is a~$\cont^{1,\alpha}$ surface in a neighborhood of the origin, hence~$C_k$ is a half-space. 
		
		However, this leads to a contradiction, concluding the proof of~\eqref{eq::gap}.
		
	\end{appendices}
	
	\section*{Acknowledgments} \label{sec::acknowledgments}	
	\addcontentsline{toc}{section}{\nameref{sec::acknowledgments}}
	We thank the referee for their very useful comments and suggestions.
	
	This work has been supported by
	the Australian Future Fellowship FT230100333 ``New perspectives on nonlocal equations'' and
	by the Australian Laureate Fellowship FL190100081 ``Minimal surfaces, free boundaries and partial differential equations.''
	
	\addcontentsline{toc}{section}{References}
	\nocite{*}
	\printbibliography
	
\end{document}